\documentclass[11pt,Letterpaper, usenames,dvipsnames,reqno]{amsart}

\usepackage{amsmath, amscd, amsthm, amssymb}

\usepackage{hyperref}

\usepackage{epic}
\usepackage{amscd}
\usepackage{amsfonts}

\usepackage[margin=1in]{geometry}

\usepackage[all,arc]{xy}
\usepackage{enumerate}
\usepackage{mathrsfs}
\usepackage{color}   
\usepackage{hyperref}
\hypersetup{
    colorlinks=true, 
    linktoc=all,     
    linkcolor=blue,  
}
\usepackage{amsthm}
\usepackage{amsmath}
\usepackage{graphicx}
\usepackage{wasysym}
\usepackage{pgf,tikz}
\usetikzlibrary{positioning}
\usepackage{ marvosym }
\usepackage{tikz-cd}
\usepackage{ textcomp }
\usepackage{bbm}
\usepackage{ stmaryrd }

\def\C{{\mathbf C}}
\def\R{{\mathbf R}}

\def\Z{{\mathbf Z}}
\def\Q{{\mathbf Q}}
\def\A{{\mathbf A}}

\def\Cv{{\mathcal{C}_v}}

\newtheorem{theorem}{Theorem}[subsection]

\newtheorem{lemma}[theorem]{Lemma}

\newtheorem{proposition}[theorem]{Proposition}

\newtheorem{corollary}[theorem]{Corollary}

\newtheorem{claim}[theorem]{Claim}

\theoremstyle{definition}

\theoremstyle{remark}

\newtheorem{example}[theorem]{Example}

\newcommand{\mm}[4]{\left(\begin{smallmatrix} #1 & #2\\ #3 & #4\end{smallmatrix}\right)}
\newcommand{\mb}[4]{\left(\begin{array}{cc} #1 & #2\\ #3 & #4\end{array}\right)}
\newcommand{\Cvone}[1]{\mathcal{C}_v(#1,\cdot)}
\newcommand{\Cvtwo}[2]{\mathcal{C}_v(#1,#2)}

\DeclareMathOperator{\tr}{tr}

\DeclareMathOperator{\Spin}{Spin}

\DeclareMathOperator{\Sp}{Sp}
\DeclareMathOperator{\GSp}{GSp}
\DeclareMathOperator{\PGSp}{PGSp}

\DeclareMathOperator{\SU}{SU}

\DeclareMathOperator{\SL}{SL}
\DeclareMathOperator{\GL}{GL}

\DeclareMathOperator{\Span}{Span}

\def\g{{\mathfrak g}}
\def\f{{\mathfrak f}}

\def\h{{\mathfrak h}}
\def\k{{\mathfrak k}}
\def\p{{\mathfrak p}}
\def\a{{\mathfrak a}}
\def\m{{\mathfrak m}}

\def\sl{{\mathfrak {sl}}}

\def\su{{\mathfrak {su}}}

\begin{document}
\title{Exceptional theta functions and arithmeticity of modular forms on $G_2$}
\author{Aaron Pollack}
\address{Department of Mathematics\\ The University of California San Diego\\ La Jolla, CA USA}
\email{apollack@ucsd.edu}
\thanks{Funding information: AP has been supported by the Simons Foundation via Collaboration Grant number 585147, by the NSF via grant numbers 2101888 and 2144021, and by an AMS Centennial Research Fellowship.  Part of this work was done while the author visited the Erwin Schrodinger Institute in Vienna, and the author thanks them for their support and hospitality.  The author would also like to thank the Isaac Newton Institute for Mathematical Sciences, Cambridge, for support and hospitality during the programme New Connections in Number Theory and Physics where some work on this paper was undertaken. The work at the INI was supported by EPSRC grant no EP/R014604/1 and 34.}

\begin{abstract} 
Quaternionic modular forms on the split exceptional group $G_2 = G_2^s$ were defined by Gan-Gross-Savin.  A remarkable property of these automorphic functions is that they have a robust notion of Fourier expansion and Fourier coefficients, similar to the classical holomorphic modular forms on Shimura varieties.  In this paper we prove that in even weight $\ell$ at least $6$, there is a basis of the space of cuspidal modular forms of weight $\ell$ such that all the Fourier coefficients of elements of this basis are in the cyclotomic extension of $\Q$.

Our main tool for proving this is to develop a notion of ``exceptional theta functions" on $G_2$.  We also develop a notion of exceptional theta functions on $\Sp_6$.  In the case of $\Sp_6$, these are level one, holomorphic vector-valued Siegel modular forms, with explicit Fourier expansions, that are the theta lifts from algebraic modular forms on anisotropic $G_2$ for the dual pair $\Sp_6 \times G_2^a \subseteq E_{7,3}$.  In the case of split $G_2$, our exceptional theta functions are level one quaternionic modular forms, with explicit Fourier expansions, that are the theta lifts from algebraic modular forms on anistropic $F_4$ for the dual pair $G_2 \times F_4^I \subseteq E_{8,4}$.
	
As further consequences of this theory of exceptional theta functions, we also obtain the following corollaries: 1) there is an algorithm to determine if any cuspidal, level one Siegel modular form on $\Sp_6$ of most weights is a lift from $G_2^a$; 2) the level one theta lifts from $F_4^I$ possess an integral structure, in the sense of Fourier coefficients; 3) in every even weight $k \geq 6$, there is a nonzero, level one Hecke eigen quaternionic cusp form on split $G_2$ with nonzero $\Z \times \Z \times \Z$ Fourier coefficient.  Finally, we obtain evidence for a conjecture of Gross relating Fourier coefficients of certain $G_2$ quaternionic modular forms to $L$-values.
\end{abstract}
\maketitle

\setcounter{tocdepth}{1}
\tableofcontents

\section{Introduction}
Holomorphic modular forms on Shimura varieties have a good notion of Fourier coefficients.  It is theorem, going back to Shimura, that the Fourier coefficients of holomorphic modular forms on Hermitian tube domains give these modular forms an algebraic structure: there is a basis of the space of holomorphic modular forms so that all the Fourier coefficients of elements of this basis are algebraic.

Outside the realm of holomorphic modular forms on Shimura varieties, there is little one can say\footnote{One exception might be the case of globally generic cohomological automorphic forms, for which the Whittaker coefficients can be directly related to Satake paremeters.} about the arithmeticity of Fourier coefficients of automorphic functions.  In fact, at present, it is not clear in general how one might define a good notion of Fourier coefficients for spaces of non-holomorphic automorphic forms.

Nevertheless, the quaternionic exceptional groups possess a special class of automorphic functions called the quaternionic modular forms, which do have a good notion of Fourier expansion and Fourier coefficients.  These automorphic forms go back to Gross-Wallach \cite{grossWallachI,grossWallachII} and Gan-Gross-Savin \cite{ganGrossSavin}. The precise shape of their Fourier expansion was determined in \cite{pollackQDS}, extending and refining earlier work of Wallach \cite{wallach}.  In particular, it is possible to define what it means for a quaternionic modular form to have Fourier coefficients in some ring $R \subseteq \C$.  In \cite{pollackAWSNotes}, we conjectured that the space of quaternionic modular forms of some fixed weight $\ell$ on a quaternionic exceptional group $G$ has a basis consisting of elements all of whose Fourier coefficients are algebraic numbers.\footnote{It seems useful to point out that there has been other recent work, specifically \cite{prasannaVenkatesh}, that conjectures the existence of surprising algebraic structures on spaces of non-holomorphic automorphic forms.}

In this paper, we provide substantial evidence toward this conjecture in the case of $G$ the split exceptional group $G_2$.  We setup the result now.

Suppose $\varphi$ is a cuspidal quaternionic modular form on $G_2$ of weight $\ell \geq 1$.  Let $N$ denote the unipotent radical of the Heisenberg parabolic of $G_2$ and $Z$ its center.  Then 
\[\varphi_Z(g_f g_\infty) = \sum_{\chi}{a_\chi(g_f)W_\chi(g_\infty)}\]
is the Fourier expansion of $\varphi$. Here the $\chi$ range over non-degenerate characters of $N(\Q)\backslash N(\A)$ and the $W_\chi$ are the generalized Whittaker functions of \cite{pollackQDS}.  The functions $a_\chi: G_2(\A_f) \rightarrow \C$ are locally constant and called the Fourier coefficients of $\varphi$.  We say that $\varphi$ has Fourier coefficients in a ring $R$ if $a_\chi(g_f) \in R$ for all characters $\chi$ of $N(\Q)\backslash N(\A)$ and all $g_f \in G_2(\A_f)$.  We write $S_{\ell}(G_2; R)$ for the space cuspidal quaternionic modular forms on $G_2$ of weight $\ell$ with Fourier coefficients in $R$.

Let $\Q_{cyc} = \Q(\mu_\infty)$ be the cyclotomic extension of $\Q$.
\begin{theorem}\label{thm:introArith} Suppose $\ell \geq 6$ is even.  Then there is a basis of the cuspidal quaternionic modular forms of weight $\ell$ with all Fourier coefficients in $\Q_{cyc}$.  In other words, $S_{\ell}(G_2,\C) = S_{\ell}(G_2,\Q_{cyc}) \otimes_{\Q_{cyc}} \C$.
\end{theorem}

Our main tool for proving Theorem \ref{thm:introArith} is a notion of ``exceptional" theta functions, that mirrors the classical theory of Siegel modular theta functions associated to pluriharmonic polynomials. Recall that these classical pluriharmonic theta functions are (often cuspidal) Siegel modular forms, with completely explicit Fourier expansions, that can be considered as arising from the Weil representation restricted to $\Sp_{2g} \times O(V)$ where $V$ is a rational quadratic space whose quadratic form is positive definite.  In other words, they are the theta lifts from algebraic modular forms on $O(V)$ with \emph{nontrivial archimedean weight} to holomorphic Siegel modular forms.  From the perspective of theta lifts and algebraic modular forms, the remarkable fact is that one can give the Fourier expansions of these lifts completely explicitly.

To prove Theorem \ref{thm:introArith}, we develop a notion of exceptional theta functions on $G_2$.  These are quaternionic modular forms on $G_2$ whose Fourier coefficients we can tightly control.  They arise as theta lifts from an anisotropic group of type $F_4$.  Using the Siegel-Weil theorem of \cite{pollackSW}, we can prove that every cuspidal quaternionic modular form of even weight at least $6$ on $G_2$ is one of our exceptional theta functions.  This establishes the theorem.

Our theory of exceptional theta functions on $G_2$ has a parallel--but easier--development on $\Sp_6$.  As this theory is a bit easier, and might be more familiar to the reader, we begin with the $\Sp_6$-case.

\subsection{Holomorphic theta functions}
Let $G_2^a$ denote the the algebraic $\Q$ group of type $G_2$ that is split at every finite place and compact at the archimedean place. Let $H_J^1$ denote the simply connected group of type $E_7$ that is split at every finite place and the group $E_{7,3}$ at the archimedean place.  There is a dual pair $\Sp_6 \times G_2^a \subseteq H_J^1$, and a corresponding theta lift from $G_2^a$ to $\Sp_6$ using the automorphic minimal representation on $H_J^1$ \cite{kimMin} that was studied by Gross-Savin \cite{grossSavin}.  This lift produces (in general) vector-valued holomorphic Siegel modular forms on $\Sp_6$.  It makes sense to ask if, given an algebraic modular form on $G_2^a$, one can give an explicit Fourier expansion of its theta lift to $\Sp_6$.  This is easy if the weight of the algebraic modular form is trivial, but is not immediate (at least to us) if the weight is nontrivial.  Theorem \ref{thm:introSMF} computes this Fourier expansion in the level one case.

We now setup this theorem.  Let $\Theta$ denote the octonions over $\Q$ with positive definite norm.  Denote by $J=H_3(\Theta)$ the $27$-dimensional exceptional cubic norm structure consisting of the $3 \times 3$ Hermitian matrices with coefficients in $\Theta$.  Let $V_7$ denote the $7$-dimensional space of trace $0$ octonions, let $W_3$ denote the standard representation of $\GL_3$ and set $V_3 = \wedge^2 W_3$.  There is a projection $\mathcal{P}:J^\vee \simeq J \rightarrow V_3 \otimes V_7$ given by taking the trace $0$ projections of the off-diagonal entries of an element of $J$.  Consider the natural map
\begin{align*}(V_3 \otimes V_7)^{\otimes k_1 + 2k_2} & \rightarrow S^{k_1}(V_3) \otimes V_7^{\otimes k_1} \otimes S^{k_2}(\wedge^2 V_3) \otimes (\wedge^2 V_7)^{\otimes k_2} \\ &\simeq S^{k_1}(\wedge^2 W_3) \otimes V_7^{\otimes k_1} \otimes S^{k_2}(W_3) \otimes (\wedge^2 V_7)^{\otimes k_2} \otimes \det(W_3)^{k_2},\end{align*}
and denote by $\mathcal{P}_{k_1,k_2}(T)$ the image of $(T)^{\otimes (k_1 + 2k_2)}$ under this map.

Let $\omega_1$ denote the highest weight of the representation $V_7$ of $G_2$, and let $\omega_2$ denote the highest weight of the $14$-dimensional adjoint representation $\g_2 \subseteq \wedge^2 V_7$.  For non-negative integers $k_1,k_2$, let $W(k_1,k_2)$ denote the representation of $G_2(\C)$ with highest weight $k_1 \omega_1 + k_2 \omega_2$, embedded in $V_7^{\otimes k_1} \otimes (\wedge^2 V_7)^{\otimes k_2} \otimes \C$. This representation is the one generated by $v_{\omega_1}^{\otimes k_1} \otimes v_{\omega_2}^{\otimes k_2}$ where $v_{\omega_1}$ and $v_{\omega_2}$ are highest weight vectors for $V_7$ and $\g_2$ for the same Borel subgroup of $G_2(\C)$. Given $\beta \in W(k_1,k_2)$, and $T \in J$, we can form the pairing
 \[\{P_{k_1,k_2}(T),\beta\} \in S^{k_1}(\wedge^2 W_3) \otimes S^{k_2}(W_3) \otimes \det(W_3)^{k_2+4},\]
where we have shifted the $k_2 \mapsto k_2 +4$ in the exponent of $\det(W_3)$.

Denote by $R_{\Theta} \subseteq \Theta$ Coxeter's order of integral octonions, and set $J_R \subseteq J$ the elements whose diagonal entries are in $\Z$ and off-diagonal entries are in $R_{\Theta}$.  Recall that Kim's modular form on $H_J^1$ has Fourier expansion
\[\Theta_{Kim}(Z) = \sum_{T \geq 0, T \in J_{R}, rk(T) \leq 1}{a(T) e^{2\pi i(T,Z)}}\]
where $a(0) = 1$ and if $T$ is rank one then $a(T) = 240 \sigma_3(d_T)$ where $d_T$ is the largest integer with $d_T^{-1}T \in J_{R}$.

Finally, set $\Gamma_{G_2} = G_2^a(\Z)$.
\begin{theorem}\label{thm:introSMF} Suppose $\beta \in W(k_1,k_2)$ and $\alpha$ is the level one algebraic modular form on $G_2^a$ with $\alpha(1) = \frac{1}{|\Gamma_{G_2}|} \sum_{\gamma \in \Gamma_{G_2}}{\gamma \cdot \beta}$.  Then the theta lift $\Theta(\alpha)$ of $\alpha$ is a vector-valued Siegel modular form of weight $(k_1+2k_2+4,k_1+k_2+4,k_2+4)$ with Fourier expansion
\[\Theta(\alpha)(Z) = \sum_{T \geq 0, T \in J_{R}, rk(T) \leq 1}{a(T)\{P_{k_1,k_2}(T),\beta\} e^{2\pi i (T,Z)}}.\]
\end{theorem}
When $k_2 > 0$ \cite{grossSavin} proves that $\Theta(\alpha)$ is a cusp form.

A simple restatement of Theorem \ref{thm:introSMF} is as follows.  Consider the projection map $J \rightarrow H_3(\Q)$ given by sending
\[\left(\begin{array}{ccc} c_1 & a_3 & a_2^* \\ a_3^* & c_2 & a_1 \\ a_2 & a_1^* & c_3 \end{array}\right) \mapsto \frac{1}{2} \left(\begin{array}{ccc} 2c_1 & \tr(a_3) & \tr(a_2) \\ \tr(a_3) & 2c_2 & \tr(a_1) \\ \tr(a_2) & \tr(a_1) & 2c_3 \end{array}\right).\]
Then if $T_0 \in H_3(\Q)$ (Hermitian $3 \times 3$ matrices with $\Q$ coefficients) is half-integral, the $T_0$ Fourier coefficient of $\Theta(\alpha)(Z)$ is
\begin{equation}\label{eqn:T0FC} a_{\Theta(\alpha)}(T_0) = \sum_{T \geq 0, T \in J_{R}, rk(T) \leq 1, T \mapsto T_0}{a(T)\{P_{k_1,k_2}(T),\beta\}}.\end{equation}
(The sum is finite and explicitly determinable.)  We verify that the $a_{\Theta(\alpha)}(T_0)$ live in the highest weight submodule
\[V_{[k_1,k_2]}:=V_{(k_1+2k_2+4,k_1+k_2+4,k_2+4)} \subseteq S^{k_1}(\wedge^2 W_3) \otimes S^{k_2}(W_3) \otimes \det(W_3)^{k_2+4}\]
so that $\Theta(\alpha)$ really is a vector-valued Siegel modular form for the representation $V_{[k_1,k_2]}$.

An important computational aspect of Theorem \ref{thm:introSMF} is that one can use $\beta$ in the pairing $\{P_{k_1,k_2}(T),\beta\}$ instead of $\alpha(1) = \frac{1}{|\Gamma_{G_2}|} \sum_{\gamma \in \Gamma_{G_2}}{\gamma \cdot \beta}$.  This enables one to compute theta lifts much more quickly than if one had to use the algebraic modular form $\alpha(1) \in W(k_1,k_2)^{\Gamma_{G_2}}$.  Moreover, \emph{even if one does not know a priori that $W(k_1,k_2)^{\Gamma_{G_2}} \neq 0$}, the theorem still holds as stated.  In particular, verifying that the right hand side of equation \eqref{eqn:T0FC} is nonzero for a single $T_0$ shows that $\frac{1}{|\Gamma_{G_2}|} \sum_{\gamma \in \Gamma_{G_2}}{\gamma \cdot \beta} \neq 0$.  The reader can, of course, easily check this claim directly.

The Fourier expansion in Theorem \ref{thm:introSMF} can be seen as completely analogous to the Fourier expansion of classical pluriharmonic theta functions.  Indeed, the $\beta \in W(k_1,k_2)$ becomes the pluriharmonic polynomial, and the rank one $T$'s become the lattice vectors over which one sums.

When the Siegel modular form $\Theta(\alpha)$ is a Hecke eigenform, it has Satake parameters $c_p$, one for each prime number $p$, which are semisimple elements in $\Spin_7(\C)$. (Because we work with level one forms, we blur the distinction between $\Sp_6$ and $\PGSp_6$.) It is proved by Gross-Savin \cite{grossSavin}, Maagard-Savin \cite{magaardSavin}, and Gan-Savin \cite{ganSavinHD} that the theta lift is functorial for spherical representations; in fact, it is functorial for all representations, see Gan-Savin \cite{ganSavinLLC}.  In particular, these conjugacy classes $c_p$ are in $G_2(\C) \subseteq \Spin_7(\C)$.  Thus Theorem \ref{thm:introSMF} can produce numerous explicit examples of level one vector-valued Siegel modular forms all of whose Satake parameters are in $G_2(\C)$.  We have taken the liberty of providing some small illustration of this, as follows.  Let $\lambda_1 = (12,8,8)$ and $\lambda_2 = (14,10,8)$.  Then it is known from Chenevier-Taibi \cite{chenevierTaibi} that the space of vector-valued, level one cuspidal Siegel modular forms of these weights are one dimensional.

\begin{corollary}\label{cor:introCor1} For both $\lambda_1 = (12,8,8)$ and $\lambda_2 = (14,10,8)$, the cuspidal Siegel modular forms of these weights are lifts from $G_2^a$.  In particular, their Satake parameters all lie in $G_2(\C)$.
\end{corollary}

Indeed, the proof of Corollary \ref{cor:introCor1} is to produce a single $\beta_1 \in W(0,4)$ and $\beta_2 \in W(2,4)$, and a single $T_0$ so that the right hand side of \eqref{eqn:T0FC} is nonzero for these $\beta_i$.  It then follows that the $\Theta(\alpha)$ are nonzero level one Siegel modular forms, and thus by the dimension computation of \cite{chenevierTaibi} are the unique cuspidal level one eigenforms of weights $\lambda_1$ and $\lambda_2$.  Producing many more such explicit examples would be possible.  We remark that one can find conjectures here \cite{bCGwebsite} of which small weight level one Siegel modular forms have their Satake parameters in $G_2(\C)$.

More than just a couple of examples, however, Theorem \ref{thm:introSMF} produces an algorithm to determine if any fixed level one Siegel modular cusp form of most weights is a lift from $G_2^a$.  To setup the result, note that for a weight $\lambda = (\lambda_1,\lambda_2,\lambda_3)$, it is known \cite{ibukiyama} that there exists explicitly determinable finite sets $\mathcal{C}_{\lambda}$ of half-integral symmetric matrices $T$ so that if $F$ is a cuspidal level one Siegel modular form of weight $\lambda$, and if $a_F(T) = 0$ for all $T \in \mathcal{C}_\lambda$, then $F=0$.

\begin{corollary}\label{cor:alg} Suppose $\lambda = (k_1 + 2k_2 + 4,k_1+k_2 + 4,k_2+4)$ with $k_2 > 0$, and $F$ is a level one Siegel modular cusp form of weight $\lambda$ on $\Sp_6$, whose Fourier coefficients $a_F(T)$ are given for all $T \in \mathcal{C}_\lambda$.  Then there is an algorithm to determine if $F = \Theta(\alpha)$ for some algebraic modular form $\alpha$ on $G_2^a$.\end{corollary}

Let us indicate now some of the ingredients that go into the proof of Corollary \ref{cor:alg}.  First, let us clarify that the theta lifts $\Theta(\alpha)$ of level one algebraic modular forms $\alpha(1) \in W(k_1,k_2)^{\Gamma_{G_2}}$ that appear in Theorem \ref{thm:introSMF}  are defined as integrals
\[\Theta(\alpha)(g) = \int_{G_2^a(\Q)\backslash G_2^a(\A)}{\{\Theta_{k_1,k_2}(g,h),\alpha(h)\}\,dh}\]
for a certain specific vector-valued element of $\Theta_{k_1,k_2}$ in the minimal representation $\Pi_{min}$ on $H_J^1$.  Here, in the integral, $\{\,,\,\}$ is a $G_2^a(\R)$-equivariant pairing valued in $S^{k_1}(V_3) \otimes S^{k_2}(\wedge^2 V_3)$. Let $\{\beta_1,\ldots,\beta_N\}$ be a spanning set of $W(k_1,k_2)$, and $\alpha_j = \frac{1}{|\Gamma_{G_2}|}\sum_{\gamma \in \Gamma_{G_2}}{\gamma \beta_j}$.  One can use Theorem \ref{thm:introSMF} to explicitly compute the Fourier coefficients of $\Theta(\alpha_j)$ associated to the various $T's$ in $\mathcal{C}_{\lambda}$, where $\lambda = (k_1+2k_2+4,k_1+k_2+4,k_2 + 4)$.  Using linear algebra, one can check algorithmically if $F$ is some linear combination of the $\Theta(\alpha)$'s.  Thus Theorem \ref{thm:introSMF} gives an algorithm to determine if the Siegel modular form $F$ is a $\Theta(\alpha)$ for some algebraic modular form $\alpha$ on $G_2^a$.  Thus Corollary \ref{cor:alg} will be proved if one knew the following claim:

\begin{claim}\label{claim:introHD} Suppose $F$ is a level one cuspidal Siegel modular form of weight $\lambda = (k_1 + 2k_2 + 4,k_1+k_2 + 4,k_2+4)$ with $k_2 > 0$.  Suppose $F$ is in the image of the theta correspondence for $G_2^a \times \Sp_6 \subseteq H_J^1$.  Then $F= \Theta(\alpha)$ for some algebraic modular form $\alpha(1) \in W(k_1,k_2)^{\Gamma_{G_2}}$. \end{claim}
We prove this claim.  The proof uses some powerful ingredients: The Howe Duality theorem of Gan-Savin \cite{ganSavinHD}; an analysis of the minimal representation, as provided by Gross-Savin \cite{grossSavin}, Magaard-Savin \cite{magaardSavin} and Gan-Savin \cite{ganSavinHD}; the existence of enough nonvanishing Fourier coefficients of $F$ of a certain form, as proved by B{\"o}cherer-Das \cite{bochererDas2021}; and another argument from Gross-Savin \cite{grossSavin}.

Our proof of Theorem \ref{thm:introSMF} is very simple.  Let $\p_{J}$ be the complexification of the $-1$ eigenspace for the Cartan involution on the Lie algebra of $E_{7,3} = H_J^1(\R)$.  Write $\p_J = \p_J^+ \oplus \p_J^-$, the natural decomposition, so that $\p_J^-$ annihilates the automorphic form on $E_{7,3}$ associated to any holomorphic modular form for this group.  Let $\{X_\alpha\}_\alpha$ be a basis of $\p_J^+$ and $\{X_\alpha^\vee\}_\alpha$ be the dual basis of $\p_J^{+,\vee} \simeq \p_J^-$.  For an automorphic form $\varphi$ on $H_J^1$, set $D\varphi(g) = \sum_{\alpha}{X_\alpha \varphi \otimes X_\alpha^\vee}.$  For an integer $m \geq 0$, let $D^m \varphi = D \circ D \circ \cdots \circ D \varphi$, so that
\[D^m \varphi = \sum_{\alpha_1,\ldots,\alpha_m}{X_{\alpha_m} \cdots X_{\alpha_1}\varphi \otimes X_{\alpha_1}^\vee \otimes \cdots \otimes X_{\alpha_m}^\vee}.\]
If (abusing notation) $\Theta_{Kim}$ denotes the automorphic form on $H_J^1$ associated to Kim's holomorphic modular form, we set $\Theta_{k_1,k_2}(g) = D^{k_1+2k_2}\Theta_{Kim}(g)$.

With this definition, using Kim's expansion of $\Theta_{Kim}$, we compute the Fourier expansion of $\Theta_{k_1,k_2}(g)$.  This is the main step in the proof of Theorem \ref{thm:introSMF}.  We then use this to compute the Fourier expansion of $\Theta(\alpha)$.  One obtains (the automorphic form associated to) a holomorphic function with the Fourier expansion as given in Theorem \ref{thm:introSMF}.  We remark that the use of differential operators as we do has some overlap with the works \cite{ibukiyamaDiff1,clerc} of Ibukiyama and Clerc.

\subsection{Quaternionic theta functions} Denote by $G_J$ the group of type $E_8$ which is split at every finite place and $E_{8,4}$ at the archimedean place and by $G_2$ the split exceptional group of this Dynkin type.  Analogous to the dual pair $\Sp_6 \times G_2^a \subseteq H_J^1$ is the dual pair $G_2 \times F_4^I \subseteq G_J$, where $F_4^I$ is a specific form of $F_4$ that is split at all finite places and compact at the archimedean place, defined as the stabilizer of the identity matrix $I \in J$.  We compute the theta lifts of certain algebraic modular forms on $F_4^I$ to $G_2$, and obtain cuspidal quaternionic modular forms on $G_2$ together with their exact Fourier expansions.

More precisely, let $J^0$ denote the trace $0$ elements of $J$.  In other words, $J^0$ consists of the $X \in J$ with $(I,X)_I = 0$, where $(\,,\,)_I$ is the symmetric non-degenerate pairing on $J$ determined by $I$.  There is a surjective $F_4^I$-equivariant map $\wedge^2 J^0 \rightarrow \mathfrak{f}_4$ from $\wedge^2J^0$ to the Lie algebra $\mathfrak{k}_4$ of $F_4^I$.  Denote by $V_{\lambda_3}$ the kernel of this map.  It is an irreducible representation of $F_4$ of dimension $273$.  For an integer $m > 0$, let $V_{m \lambda_3} \subseteq (\wedge^2 J^0)^{\otimes m}$ denote the irreducible representation of $F_4$ with highest weight $m \lambda_3$, generated by the tensor product of a highest weight vector of $V_{\lambda_3}$.  It follows from the archimedean theta correspondence calculated in \cite{HPS} that algebraic modular forms on $F_4^I$ for the representation $V_{m\lambda_3}$ should lift to quaternionic modular forms on $G_2$ of integer weight $4 +m$.  We explicitly compute the Fourier expansion of this lift, and as a result, obtain exceptional ``pluriharmonic" cuspidal quaternionic theta functions on $G_2$.

We setup the statement of the result. To do so, recall that the group $F_4^I$ has not one but two integral structures \cite{grossZgrps}, \cite{elkiesGrossIMRN}.  More precisely, if $F_4^I(\widehat{\Z})$ denotes one of these integral structures, then the double coset space $F_4^I(\Q) \backslash F_4^I(\A_f)/F_4^I(\widehat{\Z})$ has size two.  Because of this, algebraic modular forms for $F_4^I$ can be described as follows. Denote by $M_J^1$ the subgroup of $\GL(J)$ fixing the cubic norm. Set $\Gamma_{I}$ to be the subgroup of $M_J^1$ preserving the lattice $J_R$ and fixing the element $I$.  Recall the element $E \in J_R$ of norm one from \cite{elkiesGrossIMRN} or \cite{grossZgrps}.  Let $\Gamma_E$ denote the subgroup of $M_J^1$ preserving the lattice $J_R$ and fixing the element $E$.  Fix an element $\delta_{E}^{\Q} \in M_J^1(\Q)$ with $\delta_E^\Q E = I$.  If $V$ is a representation of $F_4^I(\R)$, let $\Gamma_E$ act on $V$ via $\gamma \cdot_E v = (\delta_E^\Q v \delta_E^{\Q,-1})v$, where the conjugation takes place in $M_J^1$.  Then a level one algebraic modular form for $F_4^I$ can be considered as a pair
\[(\alpha_I,\alpha'_E) \in V^{\Gamma_{I}} \oplus V^{\Gamma_E}.\]

In the case of $V = V_{m\lambda_3}$, we rephrase this as follows.  Let $(\,,\,)_E$ be the symmetric non-degenerate pairing on $J$ determined by $E$; one has
\[(u,v)_E = \frac{1}{4}(E,E,u)(E,E,v) - (E,u,v)\]
where $(\cdot,\cdot,\cdot)$ is the symmetric trilinear form on $J$ satisfying $(x,x,x) = 6n(x)$, $6$ times the cubic norm on $J$.  Set $J_E^0$ to be the perpendicular space to $E$ under the pairing $(\,,\,)_E$.  One easily verifies that $\delta_{E}^{\Q,-1} J^0 = J_E^0$.  Thus
\[\alpha_E := \delta_E^{\Q,-1} \alpha'_E \subseteq (\wedge^2 J_E^0)^{\otimes m}\]
and is $\Gamma_E$ invariant for the natural action of $\Gamma_E$.  It will be convenient for us to consider the algebraic modular form to be the pair $(\alpha_I,\alpha_E)$.

Now, for $x,y, b \in J$ and $c\in J^\vee$, write
\[\langle x \wedge y, b \wedge c \rangle_I = (x,b)_I(y,c)-(x,c)(y,b)_I.\]
We use the same notation for the pairing $(\wedge^2 J)^{\otimes m} \otimes (\wedge^2 J)^{\otimes m} \rightarrow \C$ that extends this one via $\langle z_1 \otimes \cdots \otimes z_m, z_1' \otimes \cdots \otimes z_m' \rangle_I = \prod_{j=1}^{m} \langle z_j, z_j' \rangle_I$.  Thus if $\beta \in V_{m \lambda_3}$, $b \in J$ and $c \in J^\vee$, we can compute the quantity $\langle \beta, (b \wedge c)^{\otimes m} \rangle \in \C$.  We similarly define $\langle \,,\,\rangle_E$, by replacing the pairing $(\,,\,)_I$ with $(\,,\,)_E$.

Set $W_{J_R} = \Z \oplus J_R \oplus J_R^\vee \oplus \Z$.  For $w \in W_{J_R}$ of rank one, set $a(w) = \sigma_4(d_w)$ where $d_w$ is the largest integer with $d_w^{-1}w \in W_{J_R}$.  In \cite{pollackE8}, we proved that the minimal modular form $\Theta_{Gan}$ on $G_J$ has Fourier expansion
\[\Theta_{Gan,Z}(g) = \Theta_{Gan,N}(g) + \sum_{w \in W_{J_R}, rk(w) =1}{a(w) W_{2\pi w}(g)}.\]
Here $N \supseteq Z \supseteq 1$ is the unipotent radical of the Heisenberg parabolic of $E_{8,4}$, $\Theta_{Gan,Z}$, $\Theta_{Gan,N}$ denote constant terms, and $W_{2\pi w}$ is the generalized Whittaker function of \cite{pollackQDS}.

Now, if $w \in W_{J}$ and $m > 0$, set $P_m(w) = (b \wedge c)^{\otimes m} \in (\wedge^2 J)^{\otimes m}$ if $w = (a,b,c,d)$.  Moreover, with this notation, let $p_I(w)$ and $p_E(w)$ be the binary cubic forms given as
\[p_I(w)(u,v) = au^3 + (b,I^\#)u^2 v + (c,I) uv^2 + dv^3;\,\,\, p_E(w)(u,v) = au^3 + (b,E^\#)u^2v + (c,E) uv^2 + dv^3.\]
We prove the following.

\begin{theorem}\label{thm:introQMF} Suppose $\alpha$ is a level one algebraic modular form on $F_4^I$ for the representation $V_{m\lambda_3}$ with $m >0$.  Represent $\alpha$ as a pair $(\alpha_I,\alpha_E)$ with $\alpha_I \in (\wedge^2 J)^{\otimes m}$ being $\Gamma_I$ invariant and $\alpha_E \in (\wedge^2 J)^{\otimes m}$ being $\Gamma_E$ invariant.  Let $\beta_I,\beta_E$ be in $V_{m \lambda_3}$, respectively $\delta_E^{\Q,-1} V_{m \lambda_3}$ so that $\alpha_I = \frac{1}{|\Gamma_I|}\sum_{\gamma \in \Gamma_I}{\gamma \beta_I}$ and $\alpha_E = \frac{1}{|\Gamma_E|}\sum_{\gamma \in \Gamma_E}{\gamma \beta_E}.$  Then the theta lift $\Theta(\alpha)$ is a cuspidal, level one, quaternionic modular form on $G_2$ of weight $4+m$ with Fourier expansion
\begin{align*}\Theta(\alpha)_Z(g) &= \frac{1}{|\Gamma_I|}\sum_{w \in W_{J_R}, rk(w) =1}{a(w) \langle P_m(w), \beta_I\rangle_I W_{2\pi pr_I(w)}(g)}\\ &\,\,+ \frac{1}{|\Gamma_E|}\sum_{w \in W_{J_R}, rk(w) =1}{a(w) \langle P_m(w), \beta_E\rangle_E W_{2\pi pr_E(w)}(g)}.\end{align*}
\end{theorem}

A simple restatement of Theorem \ref{thm:introQMF} is as follows.  If $w_0$ is an integral binary cubic form, then the $w_0$ Fourier coefficient of $\Theta(\alpha)$ is
\[a_{\Theta(\alpha)}(w_0) = \frac{1}{|\Gamma_I|}\sum_{w \in W_{J_R}, pr_{w,I}(w) = w_0}{a(w) \langle P_m(w),\beta_I\rangle_I} + \frac{1}{|\Gamma_E|}\sum_{w \in W_{J_R}, pr_{w,E}(w) = w_0}{a(w) \langle P_m(w),\beta_E\rangle_E}.\]
These sums are finite.

The reason we do not state Theorem \ref{thm:introQMF} in the case $m=0$ is because then the theta lifts will be non-cuspidal.  In fact, the theta lifts obtained for $m=0$ are exactly the automorphic forms obtained in \cite[Section 10]{ganGrossSavin}.  Thus Theorem \ref{thm:introQMF} may be considered a generalization of \cite[Section 10]{ganGrossSavin}.

Again, just like Theorem \ref{thm:introSMF}, the Fourier expansion given by Theorem \ref{thm:introQMF} is completely parallel to the classical pluriharmonic theta functions: The $\beta's$ in $V_{m \lambda_3}$ are the pluriharmonic polynomial, and the sum over $w \in W_{J_R}$ of rank one is the sum over lattice vectors.

We now state a few corollaries of Theorem \ref{thm:introQMF}.  For the first corollary, we can partially refine Theorem \ref{thm:introArith} in the case of level one.
\begin{corollary}\label{cor:introIntegrality} There is a lattice $L_m^I \subseteq V_{m\lambda_3}$ and a lattice $L_m^E \subseteq \delta_{E}^{\Q,-1}V_{m\lambda_3}$ so that the level one theta lifts of elements of these lattices to $G_2$ have Fourier coefficients that are integers when evaluated at $g_f=1$.
\end{corollary}

For the second corollary, recall that it was proved in \cite{CDDHPRcompleted} that if $\pi$ is a cuspidal automorphic representation on $G_2(\A)$ that corresponds to a level one quaternionic modular form $\varphi_\pi$ of even weight $\ell$, then the completed standard $L$-function $\Lambda(\pi,Std,s)$ satisfies the exact functional equation $\Lambda(\pi,Std,s) = \Lambda(\pi,Std,1-s)$, so long as the $w_0 = u^2 v- uv^2$ Fourier coefficient of $\varphi_\pi$ is nonzero.  At the time of the writing of \cite{CDDHPRcompleted}, it was not known whether such $\pi$ exist.  Using Theorem \ref{thm:introQMF}, one easily obtains the following.
\begin{corollary}\label{cor:introG2FC} Suppose $\ell \geq 6$ is even.  Then there is a cuspidal automorphic representation $\pi$ on $G_2(\A)$ that corresponds to a level one quaternionic modular form $\varphi_\pi$ of weight $\ell$ with nonzero $w_0 = u^2 v- uv^2$ Fourier coefficient.\end{corollary}
This corollary is, in fact, an ingredient in the proof of Theorem \ref{thm:introArith}.

The third corollary we state has to do with the Fourier expansion of a particular cuspidal quaternionic modular form.  To setup this corollary, recall that Dalal \cite{dalal} has recently given an explicit formula for the dimension of the level one quaternionic cuspidal modular forms of weights at least $3$.  From his dimension formula, one has that the first such nonzero cusp form appears in weight $6$, and the space of weight $6$ cuspidal quaternionic modular forms is one-dimensional, spanned by a cusp form $\Delta_{G_2}$.  Combining the (proof of) Corollary \ref{cor:introG2FC} with Corollary \ref{cor:introIntegrality}, one obtains that $\Delta_{G_2}$ can be normalized to have integer Fourier coefficients, and that the Fourier coefficient associated to the cubic ring $\Z \times \Z \times \Z$ is nonzero.

Now, Benedict Gross has suggested that the Fourier coefficients of certain non-tempered cuspidal quaternionic modular forms on $G_2$ of weight $\ell$ should be related to square-roots of twists of $L$-values of classical modular forms of weight $2\ell$ by Artin motives associated to totally real cubic fields.  As pointed out to the author by Mundy, the quaternionic modular form $\Delta_{G_2}$ should be one of these non-tempered lifts, to which Gross's conjecture applies. Now, for $D$ congruent to $0$ or $1$ modulo $4$, denote by $\Z_D$ the quadratic ring of discriminant $D$.   Then on the one hand, in the case of $\Delta_{G_2}$, Gross's conjecture implies that the square Fourier coefficients $a_{\Delta_{G_2}}(\Z \times \Z_D)^2$ associated to the cubic ring $\Z \times \Z_D$ should be related to the central\footnote{We here use the classical normalization of $L$-functions, instead of the automorphic normalization.} $L$-value $L(\Delta,D,6)$ of the twist of Ramanujan's $\Delta$ function by the quadratic character associated to $D$.  Denote by $\delta(z) \in S_{13/2}(\Gamma_0(4))^{+}$, $\delta(z) = \sum_{D \equiv 0,1 \pmod{4}}{\alpha(D)q^D}$ the Shimura lift of $\Delta(z)$.  On the other hand, following Waldspurger \cite{waldspurger}, Kohnen-Zagier \cite{kohnenZagier} relate the squares of Fourier coefficients $\alpha(D)$ to the same $L$-value, $L(\Delta,D,6)$.  It thus makes sense to ask, in light of Gross's conjecture, if there is some relationship between $a_{\Delta_{G_2}}(\Z \times \Z_D)$ and $\alpha(D)$.  It turns out, the numbers are equal:
\begin{corollary}\label{cor:Shimura} Normalize $\Delta_{G_2}$ so that $a_{\Delta_{G_2}}(\Z \times \Z \times \Z) =1$, and normalize $\delta(z)$ so that $\alpha(1) = 1$.  Then $a_{\Delta_{G_2}}(\Z \times \Z_D) = \alpha(D)$ for all $D$.
\end{corollary}

\subsection{Acknowledgements} It is pleasure to thank Wee Teck Gan, Nadya Gurevich, and Gordan Savin for engaging with the author in a ``Research in Teams" project in Spring 2022 at the Erwin Schrodinger Institute, which helped to stimulate thinking about exceptional theta correspondences.  We thank them for fruitful discussions.  We thank Tomoyoshi Ibukiyama for sending us his note \cite{ibukiyama}, and we also thank Gaetan Chenevier, Chao Li, Finley McGlade, Sam Mundy, Cris Poor and David Yuen for helpful discussions.

\section{Groups and embeddings}\label{sec:groups}
In this section, we explain various group theoretic facts regarding the groups with which we work.

\subsection{Some exceptional groups} We begin by defining the group $H_J^1$.  Thus let $J$ be our exceptional cubic norm structure, $W_J = \Q \oplus J \oplus J^\vee \oplus \Q$.  A typical element of $W_J$ we write as an ordered four-tuple $(a,b,c,d)$, so that $a, d \in \Q$, $b \in J$ and $c \in J^\vee$.  We put on $W_J$ Freudenthal's symplectic form, and quartic form.  The group $H_J^1$ is defined as the algebraic $\Q$-group preserving these two forms.  We let $H_J$ be the group that preserves the forms on $W_J$ up to similitude, and let $\nu: H_J \rightarrow \GL_1$ be this similitude.

The Siegel parabolic $P_J$ of $H_J^1$ is defined as the stabilizer of the line $\Q (0,0,0,1)$.  Write $M_J$ for the group of linear automorphisms of $J$ that preserve the norm on $J$, up to scaling.  Let $\lambda: M_J \rightarrow \GL_1$ be this scaling factor, and $M_J^1$ the kernel of $\lambda$.  A Levi subgroup of $P_J$ can be defined as the subgroup that also stabilizes the line $\Q (1,0,0,0)$.  This is isomorphic to the group of pairs $(\delta,m) \in \GL_1 \times M_J$ with $\delta^2 = \lambda(m)$.  Such a pair acts on $W_J$ as $(a,b,c,d) \mapsto (\delta^{-1}a,\delta^{-1} m(b),\delta \widetilde{m}(c),\delta d)$.  We write $M(\delta,m)$ for this group element of $H_J^1$.

\subsection{The first dual pair} We now explain how $\GSp_6 \times G_2^a$ embeds in $H_J$.  To accomplish this, we describe a linear isomorphism between $\wedge^3_0 W_6 \otimes \nu^{-1} \oplus W_6 \otimes \Theta^0$ and $W_J$.  Here $\Theta^0 = V_7$ is the trace $0$ octonions, $W_6$ is the $6$-dimensional defining representation of $\GSp_6$, and $\wedge^3_0 W_6$ is the kernel of the contraction map $\wedge^3 W_6 \rightarrow W_6 \otimes \nu$.  We let $e_1,e_2,e_3,f_1,f_2,f_3$ be the standard symplectic basis of $W_6$.
\begin{itemize}
	\item If $v_1,v_2,v_3,v_1',v_2',v_3' \in \Theta^0$, we map $v_1' e_1 + v_2' e_2 + v_3' e_3 + v_1 f_1 + v_2 f_2 + v_3 f_3$ to $(0,Y,Y',0)$ where $Y=\left(\begin{array}{ccc} 0 & v_3 & -v_2 \\ -v_3 & 0 & v_1 \\ v_2 & -v_1 & 0 \end{array}\right) \in J$ and $Y' = - \left(\begin{array}{ccc} 0 & v_3' & -v_2' \\ -v_3' & 0 & v_1' \\ v_2' & -v_1' & 0 \end{array}\right)$.
	\item We map $f_1 \wedge f_2 \wedge f_3$ to $(1,0,0,0)$ and $e_1 \wedge e_2 \wedge e_3$ to $(0,0,0,1)$.
	\item Set $f_i^* = f_{i+1} \wedge f_{i-1}$ and $e_i^* = e_{i+1} \wedge e_{i-1}$ (with indices taken modulo $3$). We map $\sum_{i,j}{b_{ij} f_i^* \wedge e_j}$ to $(0,(b_{ij}),0,0)$ and $\sum_{i,j}{c_{ij} e_i^* \wedge f_j}$ to $(0,0,(c_{ij}),0)$.
\end{itemize}
Via the natural action of $\GSp_6 \times G_2^a$ on $\wedge^3_0 W_6 \otimes \nu^{-1} \oplus W_6 \otimes \Theta^0$, we obtain an action of this group on $W_J$. It is clear that the action is faithful. To obtain the embedding into $H_J$, we must check that $\GSp_6 \times G_2^a$ preserves the symplectic and quartic form on $W_J$, up to scaling:

\begin{proposition}\label{prop:dualPair1} The defined action of $\GSp_6 \times G_2^a$ on $W_J$ gives an embedding $\GSp_6 \times G_2^a \subseteq H_J$.
\end{proposition}
To prove the proposition, we first make a few lemmas.  We work over a general field $F$ of characteristic $0$.

\begin{lemma}  Suppose $X = \left(\begin{array}{ccc} s_1 & u_1 & u_2 \\ u_3 & s_2 & u_1 \\ u_2 & u_1 & s_3 \end{array}\right)$ is in $H_3(F)$ and $v_j \in \Theta^0$ so that the element $Y=\left(\begin{array}{ccc} 0 & v_3 & -v_2 \\ -v_3 & 0 & v_1 \\ v_2 & -v_1 & 0 \end{array}\right)$ is in $J$.  Then $X \times Y = -\left(\begin{array}{ccc} 0 & v_3' & -v_2' \\ -v_3' & 0 & v_1' \\ v_2' & -v_1' & 0 \end{array}\right)$ with
	\begin{enumerate}
		\item $v_1' = s_1 v_1 + u_3 v_2 + u_2 v_3$
		\item $v_2' = u_3 v_1 + s_2 v_2 + u_1 v_3$
		\item $v_3' = u_2 v_1 + u_1 v_2 + s_3 v_3$.
	\end{enumerate}
\end{lemma}
\begin{proof} This is direct computation.\end{proof}

Let now $X = (X_{ij}) \in H_3(F)$ and consider the element $n_{L,\Sp_6}(X) = \mm{0}{X}{0}{0}$ in the Lie algebra of $\Sp_6$. Recall the element $n_L(X)$, in the Lie algebra of $H_J^1$, that acts on $(a,b,c,d) \in W_J$ as $n_L(X)(a,b,c,d) = (0,aX,b\times X,(c,X))$.
\begin{lemma} Under the above identification $\wedge^3_0 W_6 \otimes \nu^{-1} \oplus W_6 \otimes \Theta^0 \simeq W_J$, the operator $n_{L,\Sp_6}(X)$ acts as $n_{L}(X)$.\end{lemma}
\begin{proof} We first compute how $n_{L,\Sp_6}(X)$ acts on $(0,(b_{ij}),0,0) = \sum_{ij}{b_{ij} f_i^* \wedge e_j}$.  Writing out the action of $n_{L,\Sp_6}(X)$ on this element, we obtain
\[ \sum_{ij}{b_{ij}( \sum_{m}X_{m,i-1} f_{i+1} \wedge e_m \wedge e_j + \sum_{\ell} X_{\ell,i+1} e_{\ell} \wedge f_{i-1} \wedge e_j)}.\]

The coefficient of $e_1^* \wedge f_1 = e_2 \wedge e_3 \wedge f_1$ comes from $4$ terms:
\begin{itemize}
	\item $i=3,m=2,j=3$: $b_{33} X_{22}$
	\item $i=3, m=3,j=2$: $-b_{32} X_{32}$
	\item $i=2, \ell=3, j=2$: $b_{22} X_{33}$
	\item $i=2,\ell=2,j=3$: $-b_{23}X_{23}$
\end{itemize}

The coefficient of $e_2^* \wedge f_3 = e_3 \wedge e_1 \wedge f_3$ again comes from $4$ terms:
\begin{itemize}
	\item $i=2,m=3,j=1$: $b_{21} X_{31}$
	\item $i=2,m=1,j=3$: $-b_{23} X_{11}$
	\item $i=1,\ell=3,j=1$: $-b_{11} X_{32}$
	\item $i=1,\ell=1,j=3$: $b_{13}X_{12}$.
\end{itemize}
Putting the above computations together, one obtains $n_{L,\Sp_6}(X)(0,b,0,0) = (0,0,c,0)$ where $c = b \times X$.

Finally, one immediately obtains that $n_{L,\Sp_6}(X)(0,0,c,0) = (0,0,0,(c,X))$ and $n_{L,\Sp_6}(X)(1,0,0,0) = (0,X,0,0)$. The lemma follows. \end{proof}

The following lemma is immediate.
\begin{lemma} The element $\mm{0}{-1}{1}{0}$ of $\Sp_6$ acts on $W_J$ as $(a,b,c,d) \mapsto (d,-c,b,-a)$.  The element $\mm{\nu}{}{}{1}$ acts on $(a,b,c,d)$ as $(\nu^{-1}a,b,\nu c, \nu^2 d)$.\end{lemma}

\begin{lemma}\label{lem:LeviSp6action} Suppose $\mm{m}{}{}{n} \in \Sp_6$, so that $n = \,^tm^{-1}$.  The action of this element on $W_J$ is in $M_J^1$; in particular, it preserves the symplectic and quartic form on $W_J$.
\end{lemma}
\begin{proof} We compute the action of $\mm{m}{}{}{n}$ on $(a,b,c,d)$ when $n = \,^tm^{-1}$.  First, one computes that $n$ acts on $f_i^*$ taking $f_i^*$ to $\sum_{k}{(n_{k+1,i+1}n_{k-1,i-1}-n_{k-1,i+1}n_{k+1,i-1}) f_k^*} = \sum_{k}{c(n)_{ki} f_k^*}$, where $c(n) = \det(n) \,^tn^{-1}$ is the cofactor matrix of $n$.  This gives that the $\wedge^3_0 W_6 \otimes \nu^{-1}$ part of $b$ maps to $c(n) b \,^tm = \det(n) \,^tn^{-1}b n^{-1}$.  The vector (i.e., $\Theta^0$) part of $b$ moves as $\sum_{j}{v_j f_j} \mapsto \sum_{k}{ (\sum_j{n_{kj}v_j}) f_k}$.

Completely similarly (it is the same calculation), the $\wedge^3_0 W_6 \otimes \nu^{-1}$ part of $c$ maps to $c(m) c \,^tn = \det(m) \,^tm^{-1} c m^{-1}$.  The vector part of $c$ moves as $\sum_{j}{v_j' e_j} \mapsto \sum_{k}{ (\sum_j{m_{kj}v_j'}) e_k}$.

To finish the proof of the lemma, one must verify that the map $J \simeq H_3(F) \oplus \Theta^0 \otimes V_3 \rightarrow J$ given by 
\[(X,v) \mapsto (\,^tn^{-1}Xn^{-1},\det(n)^{-1}n v)\]
scales the norm on $J$ by $\det(n)^{-2}$.  This can be done, for example, by repeatedly applying the Cayley-Dickson construction and using the formulas of \cite[Section 8.1]{pollackLL}.
\end{proof}

We can now prove Proposition \ref{prop:dualPair1}.
\begin{proof}[Proof of Proposition \ref{prop:dualPair1}] Exponentiating the action of $n_{L,\Sp_6}(X)$, we see that the action of $\mm{1}{X}{}{1} \in \Sp_6$ lands in $H_J^1$.  The group $\GSp_6$ is generated by these elements, together with the $\mm{\nu}{}{}{1}$ and with $\mm{0}{-1}{1}{0}$.  The proposition thus follows from the above lemmas. \end{proof}

\subsection{Action of the maximal compact} Let $K_{\Sp_6}$ denote the standard maximal compact subgroup of $\Sp_6(\R)$, so that
\[K_{\Sp_6} = \left\{\mb{A}{-B}{B}{A}: A+iB \in U(3)\right\}.\]  

Let $K_{E_7}$ denote the subgroup of $H_J^1(\R)$ that fixes the line spanned by $(1,i,-1,-i) \in W_J(\C)$.  This is a maximal compact subgoup of $H_J^1(\R)$.  Let $\p_J$ be the complexification of the $-1$ part for the Cartan involution on the Lie algebra of $H_J^1(\R)$ for this choice of maximal compact.  Write $\p_J^{\pm}$ for its two $K_{E_7}$ factors.  We now work out how $K_{\Sp_6}$ acts on $\p_J^{\pm }$. 

We begin by discussing the Cayley transform $C_h \in H_J^1(\C)$ \cite[section 5]{pollackQDS}. To do, we review some notation from \cite{pollackQDS}:
\begin{itemize}
	\item $n_G(X) := \exp(n_L(X))$;
	\item $n_G^\vee(\gamma) := \exp(n_L^\vee(\gamma))$, with, for $\gamma \in J^\vee$, $n_L^\vee(\gamma)(a,b,c,d) = ((b,\gamma),c \times \gamma,d \gamma,0)$;
	\item for $\lambda \in \GL_1$, $\eta(\lambda)(a,b,c,d) = (\lambda^3 a,\lambda b, \lambda^{-1}c,\lambda^{-3}d)$;
	\item $\m_J$ is the Lie algebra of $M_J$, which acts on $W_J$ as given in \cite[section 3.4]{pollackQDS}.
\end{itemize}
The Cayley transform is defined as $C_h = n_G(-i) n_G^\vee(-i/2) \eta(2^{-1/2})$, so that $C_h^{-1}= \eta(2^{1/2}) n_G^\vee(i/2) n_G(i)$.  It satisfies:
\begin{enumerate}
	\item $C_h^{-1} n_L(J \otimes \C) C_h = \p_J^+$
	\item $C_h^{-1} n_L^{\vee}(J\otimes \C) C_h = \p_J^{-}$
	\item $C_h^{-1} (\m_J \otimes \C) C_h = \k_{E_7}$.
\end{enumerate}

For $Z \in J \otimes \C$, we set $r_1(Z)  \in W_J \otimes \C$ as
\[r_1(Z) = (1,Z,Z^\#,n(Z)) = r_0(-Z).\]
For $g \in H_J(\R)$ and $Z \in \mathcal{H}_J$, the factor of automorphy $j(g,Z) \in \C^\times$ and the action of $g$ on $\mathcal{H}_J$ is defined as $g r_1(Z) = j(g,Z) r_1(g \cdot Z)$.

Recall that we let $M(\delta,m)$, for $m \in M_J$ and $\delta \in \GL_1$ such that $\delta^2 = \lambda(m)$, act on $W_J$ as 
\[M(\delta,m)(a,b,c,d) = (\delta^{-1}a, \delta^{-1}m(b),\delta \widetilde{m}(c), \delta(d)).\]
For $Y \in J_{>0}$ set $M_Y = M(n(Y)^{1/2}, U_{Y^{1/2}})$. Here $Y^{1/2}$ is the positive definite square root of $Y$, and for $x \in J$, $U_x: J \rightarrow J$ is the map defined as $U_x(z) = - x^\# \times z + (x,z)x$.  For $Z = X+iY \in \mathcal{H}_J$ set $g_Z = n_G(X) M_Y$.  Then $g_Z r_1(i) = n(Y)^{-1/2} r_1(Z)$.

\begin{lemma}\label{lem:Chef}
	One has
	\begin{enumerate}
		\item $C_h^{-1} (1,0,0,0) = \frac{1}{2\sqrt{2}} r_1(i)$
		\item $C_h^{-1} (0,0,0,1) = \frac{1}{2\sqrt{2}} r_1(-i)$.
	\end{enumerate}
	Consequently, $\k = C_h^{-1} \m_J C_h$.  Moreover,
	\[C_h^{-1}(0,z,0,0) = \frac{1}{2\sqrt{2}}(i(1,z),2z-(1,z)1,i(-2z+(1,z)1),(-1,z))\]
	and
	\[C_h^{-1}(0,0,E,0) = \frac{1}{2\sqrt{2}}(-(1,E),i((1,E)-2E),2E-(1,E),i(1,E)).\]
\end{lemma}
\begin{proof} The first parts are direct verifications.  For the second, recall that $K$ is the subgroup of $H_J^1(\R)$ that stabilizes the lines $\C r_1(i)$ and $\C r_1(-i)$, while $M_J$ is the subgroup of $H_J^1(\R)$ that stabilizes the lines spanned by $(1,0,0,0)$ and $(0,0,0,1)$.  The last part is again a direct verification.
\end{proof}

Here is the statement of the result.

\begin{proposition}  Suppose $k = \mm{A}{-B}{B}{A} \in K_{\Sp_6}$, so that $A + iB \in U(3)$.  Let $E \in J$, with $E$ perpendicular to $H_3(\R)$, so that $E = v_1 \otimes x_1 + v_2 \otimes x_2 + v_3 \otimes x_3$.  Then $Ad(k) C_h^{-1} n_L^\vee(E) C_h = C_h^{-1} n_L^\vee(E') C_h$ with $E' = \det(A+iB) \,^t(A+iB)^{-1} E$.
\end{proposition}
\begin{proof} We know from lemmas above that $Ad(k) C_h^{-1} n_L^\vee(E) C_h = C_h^{-1} n_L^\vee(E') C_h$ for some $E' \in J$.  To compute it, we apply both sides to  $ r_0(i) = r_1(-i)$.  The right-hand side gives $2\sqrt{2} C_h^{-1}(0,0,E',0)$, which is $(0,-2iE',2E',0)$.  The left-hand side gives $j(k^{-1},i)^* k \cdot (0,-2iE,2E,0)$.

To further compute this left-hand side, note that we have $(0,iE,-E,0)$ is identified with $(e_1+if_1) \otimes x_1 + (e_2+if_2) \otimes x_2 + (e_3+if_3) \otimes x_3$.  Applying $k$, one gets the action of $A-iB = \,^t(A+iB)^{-1}$ on the $e_j + i f_j$. 
\end{proof}

\subsection{More exceptional groups} We define the group $F_4^I$ to be the stabilizer of $I$ inside of $M_J^1$.  For the group $G_J$, recall the Lie algebra $\g(J)$ of \cite{pollackQDS}, so that $G_J$ is the identity component of the group of automorphisms of $\g(J)$.  Let us remark that an integral model of the group $G_J$ is described in \cite{ganSW}.  We use this integral model.

\subsection{The second and third dual pairs}\label{subsec:dualpairs23} We will now define two more dual pairs, $\GL_2 \times F_4^I \subseteq H_J$ and $G_2 \times F_4^I \subseteq G_J$.

For the second dual pair, $\GL_2 \times F_4^I \subseteq H_J$, we explicate an identification $W_\Q \oplus V_2 \otimes J^0 \simeq W_J$, as follows. Here $V_2$ is the standard representation of $\GL_2$, with standard basis $e,f$. Identify $(a,b,c,d) + f \otimes B + e \otimes C$ with $(a,bI + B,cI -C,d)$.  We claim that under this identification, and the natural action of $\GL_2 \times F_4^I$ on the left, this group action preserves the symplectic and quartic form on $W_J$.

It is clear that $F_4^I$ preserves the symplectic and quartic form.  For the action of $\GL_2$, the proof is similar to (but much easier than) that given above for $\GSp_6 \times G_2$.  One simply needs to observe that $n_{L}(b)(0,B,0,0) = (0,0,bI \times B,0) = (0,0,-bB,0)$ and $n_{L}(b)(f \otimes B) = b e \otimes B$.

For the third dual pair, $G_2 \times F_4^I \subseteq G_J$, we proceed as follows.  Observe that $F_4^I \subseteq G_J$ from the action of $F_4^I$ on $J$ and $J^\vee$.  Now notice that $\g(J)^{F_4^I} = \g_2$.  Because $G_2$ is simply connected, we thus obtain a map from $G_2$ to the centralizer of the group $F_4^I$ in $G_J$.  Thus we have a unique map $G_2 \times F_4^I \rightarrow G_J$ so that $F_4^I$ acts on $J, J^\vee$ naturally, and the differential of the map $G_2 \rightarrow G_J$ is the Lie algebra embedding $\g_2 \rightarrow \g(J)$.

We now relate these two dual pairs.  Thus let $\GL_2^s$ be the Levi of the Heisenberg parabolic in $G_2$.  Observe that $\GL_2^s$ fixes the line spanned by $E_{13}$ in $\g(J)$, because it does so in $\g_2$.  Thus because $H_J$ is the Levi of the Heisenberg parabolic in $G_J$, we obtain a map $\GL_2^s \times F_4^I \rightarrow H_J \rightarrow \GL(W_J)$. 
\begin{proposition} The above-described two maps $\GL_2^s \times F_4^I \rightarrow H_J$ are identical.\end{proposition}
\begin{proof} Indeed, it is clear that the $F_4^I$'s act exactly the same way.  As for $\GL_2^s$, we have two algebraic representations of $\GL_2^s$ on $W_J$.  By the formulas for the Lie bracket on $\g(J)$, it is easy to see that the differential of these representations of $\GL_2^s$ are identical.  Thus, they agree on the level of algebraic groups, as desired.\end{proof}

\section{Facts about $F_4^I$}\label{sec:F4}
We set down some notations and results we will need concerning the representations $V_{m \lambda_3}$ and algebraic modular forms on $F_4^I$.

\subsection{Special elements of representations of $F_4$}
Let $J^0 \subseteq J$ be the trace $0$ subspace.  Let $\iota: J \rightarrow J^\vee$ be the identification of the $J$ with its dual given by the trace pairing.  Recall that if $\gamma \in J^\vee$ and $x \in J$ then $\Phi_{\gamma,x} \in End(J)$ is defined as
\[\Phi_{\gamma,x}(z) = - \gamma \times (x \times z) + (\gamma,z)x + (\gamma,x)z.\]
If $X,Y \in J$ then one defines $\Phi_{X \wedge Y} = \Phi_{\iota(X),Y} - \Phi_{\iota(Y),X}$.  This defines a map $\wedge^2 J \rightarrow \a(J) = \f_4$. (Here $\a(J)=f_4$ is the Lie algebra of the subgroup of $M_J^1$ that preserves the trace pairing, or equivalently, fixes the element $I \in J$.) As $\Phi_{\iota(1),x} = \Phi_{\iota(x),1}$, this map factors through the projection $\wedge^2 J \rightarrow \wedge^2 J^0$.

Set
\[V_{\lambda_3} = \ker \{\wedge^2 J^0 \rightarrow \f_4\}.\]
One can construct special elements of $V_{\lambda_3}$ using the following two lemmas.
\begin{lemma}  Suppose $x \in J$ is rank one.  Suppose $z \in J$ and $\gamma = z \times x$.  Then $\Phi_{\gamma,x} = 0$.
\end{lemma}
\begin{proof}
	Observe that $(\gamma,x) = (z \times x,x)= 2(z,x^\# = 0)$ so that
	\[\Phi_{\gamma,x}(z') = -(x \times z) \times (x \times z') + (x,z,z')x.\]	
	
	Keeping this expression in mind, we now recall the identity
	\[ (u \times v)^\# + u^\# \times v^\times = (u,v^\#)u + (v,u^\#)v\]
	valid for all $u,v$ in $J$.  If $u^\# = 0$, then symmetrizing this identity in $v$ gives
	\[ (u \times v_1) \times (u \times v_2) = (u,v_1,v_2) u.\]
	Consequently, taking $u = x$, $v_1 = z$ and $v_2 = z'$ we obtain
	\[(x \times z) \times (x \times z') = (x,z,z')x.\]
	This gives the lemma.
\end{proof}

\begin{lemma}  Suppose $x \in J$ and $\gamma \in J^\vee$ are such that $\Phi_{\gamma,x} = 0$.  Then $\Phi_{\iota(x),\iota(\gamma)} = 0$.
\end{lemma}
\begin{proof} The point is that one has $(x,\gamma) = 0$ and $\Phi_{\gamma,x} = 0$ (in fact, $\Phi_{\gamma,x} = 0 $ implies $(\gamma,x) = 0$) if and only if $(0,x,\gamma,0)$ is rank at most one in $W_J$.  But $J_2(0,x,\gamma,0) = (0,-\iota(\gamma),\iota(x),0)$ so the lemma follows.
	
	Of course, one can also give a more direct proof:  One has
	\begin{align*} (\iota(y),\Phi_{\iota(x),\iota(\gamma)}(\iota(\mu))) &= (\iota(y),-\iota(x) \times (\iota(\gamma) \times \iota(\mu)) + (x,\mu)\iota(\gamma) + (x,\gamma) \iota(\mu)) \\ &= (- \gamma \times (x \times y),\mu) + ((y,\gamma)x,\mu) + ((x,\gamma)y,\mu) \\ &= (\Phi_{\gamma,x}(y),\mu).
	\end{align*}
	Thus if $\Phi_{\gamma,x} = 0$, then $\Phi_{\iota(x),\iota(\gamma)} = 0$.
\end{proof}

We now write down some specific elements $x,y \in J^0$ with the following property:
\begin{itemize}
	\item The four dimensional space
	 \[E_{x,y}:=\{(0,\alpha_1 x + \alpha_2 y, \alpha_3\iota(x) + \alpha_4 \iota(y),0): \alpha_i \in F\}\]
	 is isotropic and singular in the sense that symplectic form on $W_J$ restricted to $E_{x,y}$ is $0$ and every element of $E_{x,y}$ has rank at most one.
\end{itemize}
Note that $E_{x,y}$ is isotropic and singular is equivalent to:
\begin{itemize}
	\item $(x,x) = (x,y) = (y,y) = 0$
	\item $x,y$ are rank at most one, and $x \times y = 0$
	\item $\Phi_{\iota(x),x} = 0$, $\Phi_{\iota(x),y} = 0$, $\Phi_{\iota(y),x} = 0$ and $\Phi_{\iota(y),y} = 0$.
\end{itemize}
If $x,y$ satisfy the above properties, then $x \wedge y \in V_{\lambda_3}$ is a highest weight vector for some Borel.  Indeed, one can show that the span of $x,y$ is ``amber", in the sense of Aschbacher \cite[9.3-9.5]{aschbacherI}, see also \cite[Definition 7.2, Proposition 7.3, Proposition 7.4(1)]{magaardSavin}.  Consequently, such $x,y$ will allow us to construct explicit elements of $V_{m \lambda_3}.$

\begin{lemma} Suppose $x \in J^0$ is rank one, and that $z \in J$ is such that $(z,x) = 0$ and $(x,z^\#) = 0$.  Set $y = \iota(z \times x)$.  Then $x,y \in J^0$ and $E_{x,y}$ is isotropic and singular.
\end{lemma}
\begin{proof}
	First suppose $x \in J^0$ is rank one and $z \in J$ is arbitrary.  Then if $y = \iota(z \times x)$, then $\Phi_{\iota(y),x} = 0$ and $\Phi_{\iota(x),y} = 0$.
	
	Moreover, observe that if $v \in J^0$ is rank one then $1 \times v = (1,v)1-v = -v$, so that $\Phi_{\iota(v),v} = 0$.  Thus $\Phi_{\iota(x),x} = 0$, and if we arrange that $y \in J^0$ is rank one, then $\Phi_{\iota(y),y} = 0$.
	
	To ensure that $y \in J^0$ we use $(z,x) = 0$.  Indeed,
	\[ (y,1) = (z \times x,1) = (z, x\times 1) = (z,(1,x)1-x) = -(z,x).\]  To ensure $y$ has rank at most one, we use $(x,z^\#) = 0$.  Indeed,
	\[y^\# = (x \times z)^\# = (x,z^\#)x.\]
	
	Because $\Phi_{\iota(x),x} = 0$ and $\Phi_{\iota(y),x}= 0$ and $\Phi_{\iota(y),y}=0$, we get for free that $(x,x) = (x,y) = (y,y) = 0$.  (Of course, one can also check this directly.)
	
	To complete the proof, we must verify that $x \times y = 0$.  For this, we compute
	\begin{align*} (z', x\times y) &= (z' \times x,y) \\ &= (z'\times x,-1 \times y) \\ &= ((z \times x) \times (z' \times x), -1) \\ &= ((z,z',x)x, -1) = 0.\end{align*}
\end{proof}

\begin{example} As an example of such $x,y,z \in J \otimes \C$ we can take:
\begin{itemize}
	\item $x = \left(\begin{array}{ccc} 1 & a_3 & a_2^* \\ a_3^* & -1 & a_3^* a_2^* \\ a_2 & a_2 a_3 & 0 \end{array}\right)$ with $n(a_2) = 0$, $a_2 \neq 0$, and $n(a_3) = -1$.
	\item $z = \left(\begin{array}{ccc} 0 & 0 & (a_2')^* \\ 0 & 1 & 0 \\ a_2' & 0 & 0 \end{array}\right)$ with $n(a_2') = 1$ and $(a_2',a_2) = 1$.  Note that $z^\# = -z$ and $(z,x) = 0$, so that $(z^\#,x) = 0$ as well.
	\item Then $y = z \times x = \left(\begin{array}{ccc} 0 & (a_2')^*(a_2 a_3) & * \\ * & -1 & a_3^*(a_2')^* \\ a_2'-a_2 & * & 1 \end{array}\right)$.
\end{itemize}
We will use this example to prove Corollaries \ref{cor:introG2FC} and \ref{cor:Shimura}.
\end{example}

\subsection{Algebraic modular forms}  Suppose $V$ is a representation of $F_4^I(\R)$.  By an algebraic modular form for $F_4^I$, we mean an automorphic form $\alpha:F_4^I(\Q)\backslash F_4^I(\A) \rightarrow V$ satisfying $\alpha(gk) = k^{-1} \cdot \alpha(g)$ for all $g \in F_4^I(\A)$ and $k \in F_4^I(\R)$.  If $\alpha$ has level one, then because the double coset $F_4^I(\Q)\backslash F_4^I(\A_f)/F_4^I(\widehat{\Z})$ has size two, such $\alpha$ can be described by two elements of $V$.  In this subsection, we make this identification explicit.

Recall the elements $I, E \in J_R$ of norm $1$, see \cite{elkiesGrossIMRN}.  Define $F_4^I$ to be the stabilizer of $I$ in $M_J^1 \simeq E_{6}$ and $F_4^E$ to be the stabilizer of $E$ in $M_J^1$.  From the point of view of double cosets, the element $E \in J_R$ arises as follows.  Let $\{1,\gamma_E\}$ be representatives for $F_4^I(\Q)\backslash F_4^I(\A_f) /F_4^I(\widehat{\Z})$.  Using strong approximation on $M_J^1$, we can write $\gamma_E = \delta_E^\Q (\delta_E^\R)^{-1} \delta_E^{\widehat{\Z}}$ with $\delta_E^\Q \in M_J^1(\Q)$ etc.  We can choose $\gamma_E$ and $\delta_E^{?}$ so that $E = (\delta_E^{\Q})^{-1} \cdot I$.  Indeed, observe that
\begin{enumerate}
	\item $E = (\delta_E^\Q)^{-1} I \in J_R \otimes \Q$
	\item $(\delta_E^{\Q})^{-1} = (\delta_E^{\R})^{-1} \delta_E^{\widehat{\Z}} \gamma_E^{-1}$ so that $(\delta_E^\Q)^{-1} I$ has finite part in $J_R \otimes \widehat{\Z}$.
\end{enumerate}
We have $F_4^E = (\delta_E^{\Q})^{-1} F_4^I \delta_E^\Q$.  If $V$ is a representation of $F_4^I(\R)$, we let $F_4^E(\R)$ act on $V$ via $g \cdot_E v = (\delta_E^{\R}g(\delta_E^{\R})^{-1}) v$.

We make the following notations:
\begin{itemize}
	\item Let $\Gamma_I$ be the image of $F_4^I(\Q) \cap (F_4^I(\widehat{\Z}) F_4^I(\R))$ in $F_4^I(\R)$.  Thus
	\[\Gamma_I = (F_4^I(\Q) F_4^I(\widehat{\Z})) \cap F_4^I(\R).\]
	\item Let $\Gamma_E$ be the image of $F_4^E(\Q) \cap (F_4^E(\widehat{\Z})F_4^E(\R))$ in $F_4^E(\R)$.  Thus
	\[\Gamma_E = (F_4^E(\Q) F_4^E(\widehat{\Z})) \cap F_4^E(\R).\]
	\item Let $\Gamma_{\gamma_E}$ be the image in $F_4^I(\R)$ of $\gamma_E^{-1} F_4^I(\Q)\gamma_E \cap (F_4^I(\widehat{\Z}) F_4^I(\R))$.  Thus
	\[\Gamma_{\gamma_E} = ((\gamma_E^{-1} F_4^I(\Q)\gamma_E)F_4^I(\widehat{\Z})) \cap F_4^I(\R).\]
\end{itemize}

\begin{lemma} One has $\delta_E^\R \Gamma_E \delta_E^{\R,-1} = \Gamma_{\gamma_E}.$\end{lemma}
\begin{proof} Observe
	\[\Gamma_E = (F_4^E(\Q) M_J^1(\widehat{\Z})) \cap F_4^E(\R)\]
	and
	\[\Gamma_{\gamma_E} = ((\gamma_E^{-1} F_4^I(\Q)\gamma_E)M_J^1(\widehat{\Z})) \cap F_4^I(\R).\]
	Using that $\delta_E^\R = \delta_E^{\widehat{\Z}} \gamma_E^{-1} \delta_E^{\Q}$, the lemma follows.
\end{proof}

Suppose $V$ is a representation of $F_4^I(\R)$ and
\[\alpha: F_4^I(\Q)\backslash F_4^I(\A) \rightarrow V\]
satisfies $\alpha(gk_f k_\R) = k_\R^{-1} \alpha(g)$ for all $g \in F_4^I(\A)$, $k_f \in F_4^I(\widehat{\Z})$ and $k_\R \in F_4^I(\R)$.  Then $\alpha$ is determined by its values at $g=1$ and $g = \gamma_E$.  Moreover, because $F_4(\widehat{\Z})$ acts freely on $F_4^I(\Q)\backslash F_4^I(\A)$, we have
\[F_4^I(\Q) \backslash F_4^I(\A) = \left(\Gamma_I\backslash F_4^I(\R)\right) \cdot F_4^I(\widehat{\Z}) \bigsqcup \left(\Gamma_{\gamma_E}\backslash F_4^I(\R)\right) \cdot F_4^I(\widehat{\Z}).\]
Note that the measures of the two open sets are in the proportion $\frac{1}{|\Gamma_I|} : \frac{1}{|\Gamma_E|}$.  Consequently, if $V = V_{m \lambda_3}$, one has
\[\int_{[F_4^I]} \{D^{2m}\Theta(g,h),\alpha(h)\}\,dh = \frac{1}{|\Gamma_I|} \{D^{2m}\Theta(g,1),\alpha(1)\} + \frac{1}{|\Gamma_E|} \{D^{2m}\Theta(g,\gamma_E),\alpha(\gamma_E)\}.\]

Set $\alpha_I = \alpha(1)$.  Consider $V_{m \lambda_3} \subseteq (\wedge^2 J^0)^{\otimes m} \subseteq (\wedge^2 J)^{\otimes m}$, and let $\alpha_E = \delta_E^{\Q,-1}\alpha(\gamma_E)$. Observe the following lemma:
\begin{lemma}\label{lem:Epair} If $b,x \in J$ then $(\delta_E^{\Q}b,x)_I = (b,\delta_E^{\Q,-1} x)_E$.
\end{lemma}
\begin{proof} One has
	\begin{align*} (\delta_E^{\Q}b,x)_I &= \frac{1}{4}(I,I,\delta_E^{\Q}b)(I,I,x) - (I,\delta_E^{\Q}b,x) \\ &= \frac{1}{4}(E,E,b)(E,E,\delta_E^{\Q,-1}x) - (E,b,\delta_E^{\Q,-1}x) \\ &= (b,\delta_E^{\Q,-1}x)_E.
	\end{align*}
\end{proof}
If one writes $J_E^0$ to be the perpendicular space to $E$ under the pairing $(\,,\,)_E$, then $\alpha_E \in (\wedge^2 J_E^0)^{\otimes m}$.  This follows from the lemma with $b=E$.  Moreover, observe that, for the action of $M_J^1$ on $(\wedge^2 J)^{\otimes m}$, $\alpha_E$ is stabilized by the action of $\Gamma_E$.  Thus, we can think of our algebraic modular form as being the pair
\[(\alpha_I,\alpha_E) \in [(\wedge^2 J)^{\otimes m}]^{\Gamma_I} \oplus [(\wedge^2 J)^{\otimes m}]^{\Gamma_E}.\]

\section{Theorems on Siegel modular forms} In this section, we give the proof of Theorem \ref{thm:introSMF} and its corollaries, Corollary \ref{cor:introCor1} and Corollary \ref{cor:alg}.  We do this assuming Theorem \ref{thm:DiffExp} stated below, which is the main technical ingredient in the proof of Theorem \ref{thm:introSMF}.  Theorem \ref{thm:DiffExp} will be proved in section \ref{sec:expder}.

To setup the statement of Theorem \ref{thm:DiffExp}, suppose $V$ is a $K_{E_7}$ representation, and $\varphi: H_J^1(\R) \rightarrow V$ is a function satisfying $\varphi(gk) = k^{-1} \cdot \varphi(g)$.  In this scenario, let $D\varphi$ be the $V \otimes \p_{J}^{+,\vee}$-valued function defined as 
\[D\varphi = \sum_{\alpha}{X_{\alpha}\varphi \otimes X_{\alpha}^\vee}\]
where $\{X_{\alpha}\}_\alpha$ is a $\C$-basis of $\p_J^+$.  It is easily checked that $D\varphi$ is again $K_{E_7}$-equivariant.  Recall also that $j(g,Z): H_J^1(\R) \times \mathcal{H}_J \rightarrow \C^\times$ is the factor of automorphy defined in section \ref{sec:groups}.  Let $\rho_{[k_1,k_2]}$ be the representation of $\GL_3(\C)$ on $S^{k_1}(V_3) \otimes S^{k_2}(\wedge^2 V_3)$.

\begin{theorem}\label{thm:DiffExp} Suppose $g \in \Sp_{6}(\R)$ and $\beta \in W(k_1,k_2)$. Let $\ell \geq 0$ be an integer.  There is a nonzero constant $B_{k_1,k_2}$, independent of $g$ and $T$, so that 
	\[j(g,i)^{\ell} \rho_{[k_1,k_2]}(J(g,i))\{D^{k_1+2k_2}(j(g,i)^{-\ell} e^{2\pi i (T, g\cdot i)}),\beta\} = B_{k_1,k_2} \{P_{k_1,k_2}(T),\beta\} e^{2\pi i (pr(T),g\cdot i)}.\]
Moreover, $\{P_{k_1,k_2}(T),\beta\}$ lies in the highest weight submodule $S^{[k_1,k_2]}$ of $S^{k_1}(V_3) \otimes S^{k_2}(\wedge^2 V_3)$.
\end{theorem}

Recall that if $\alpha \in \mathcal{A}(G_2^a) \otimes W(k_1,k_2)$ is a level one algebraic modular form for the representation $W(k_1,k_2)$, then we defined the theta lift of $\alpha$ as 
\[\Theta(\alpha)(g) := \int_{G_2^a(\Q)\backslash G_2^a(\A)}{\{D^{k_1+2k_2}\Theta_{Kim}((g,h)),\alpha(h)\}\,dh} = \frac{1}{|\Gamma_{G_2}|}\{D^{k_1+2k_2}\Theta_{Kim}((g,1)),\alpha(1)\}\]
where we have normalized the measure so that $G_2^a(\widehat{\Z}) G_2^a(\R)$ has measure $1$.  By rescaling $\alpha$ or this measure, we can (and will) ignore the term $\frac{1}{|\Gamma_{G_2}|}$.

We first state the fact that $\Theta(\alpha)$ is the automorphic form corresponding to a Siegel modular form of weight $[k_1,k_2]$.
\begin{proposition}\label{prop:ThetaSMF} For $g \in \Sp_6(\R)$ and $Z \in \mathcal{H}_3$ the Siegel upper half-space of degree three,  define 
\[f_{\Theta(\alpha)}(Z) = \rho_{[k_1,k_2]}(J(g,i))\Theta(\alpha)(g)\]
for any $g$ with $g \cdot (i 1_3) = Z$ in $\mathcal{H}_3$.  Then $f_{\Theta(\alpha)}(Z)$ is well-defined, and is a level one Siegel modular form of weight $[k_1,k_2]$.  If $k_2 > 0$, it is a cusp form.
\end{proposition}
\begin{proof} The fact that $f_{\Theta(\alpha)}(Z)$ is well-defined comes from the $K_{\Sp_6} \simeq U(3)$ action on $\p_J^{+,\vee}$, which was determined in section \ref{sec:groups}.  To see that it's a holomorphic modular form of the correct weight, we use \cite[Theorem 3.5]{grossSavin}.  The cuspidality when $k_2 > 0$ is proved in \cite[Corollary 4.9]{grossSavin}.
	
In more detail, if $\ell \in S^{[k_1,k_2],\vee}$ is a linear form on $S^{[k_1,k_2]}$, then one can write $\ell(\Theta(\alpha)(g))$ as a sum of terms of the form $\Theta(v_j,\alpha_j)$, where
\[\Theta(v_j,\alpha_j) = \int_{[G_2^a]}{\Theta_{v_j}((g,h))\alpha_j(h)\,dh}\]
is a usual scalar-valued theta lift.  By \cite[Theorem 3.5]{grossSavin}, as functions of $g$, these lifts all lie in the holomorphic discrete series representation $\pi(k_1,k_2)$ with minimal $K_{\Sp_6}$-type $S^{[k_1,k_2]} \det^{4}$.  Moreover, by the $K_{\Sp_6}$-equivariance that proves that $f_{\Theta(\alpha)}(Z)$ is well-defined, the vector-valued function $\Theta(\alpha)$ exactly encompasses the minimal $K_{\Sp_6}$-type in $\pi(k_1,k_2)$.  Here we use that this minimal $K_{\Sp_6}$-type appears in $\pi(k_1,k_2)$ with multiplicity one.  Consequently, $f_{\Theta(\alpha)}(Z)$ is a holomorphic Siegel modular form.  It is clearly level one.  Finally, \cite[Corollary 4.9]{grossSavin} shows that all the $\Theta(v_j,\alpha_j)$ are cusp forms if $k_2>0$, thus so is $\Theta(\alpha)$.  

This completes the proof.
\end{proof}

\begin{proof}[Proof of Theorem \ref{thm:introSMF}]  Theorem \ref{thm:introSMF} follows directly from Proposition \ref{prop:ThetaSMF} and Theorem \ref{thm:DiffExp} by plugging in the Fourier expansion of Kim's modular form $\Theta_{Kim}(Z)$.\end{proof}

We now explain the proofs of Corollary \ref{cor:introCor1} and Corollary \ref{cor:alg}, especially that of Claim \ref{claim:introHD}.

\begin{proof}[Proof of Corollary \ref{cor:introCor1}] Let $u,v$ span a null, two-dimensional subspace of the trace zero elements $\Theta^0 \otimes \C$.  That they are null means that $u^2 = uv = vu = v^2 = 0$.  We set $\beta = u^{\otimes k_1} \otimes (u \wedge v)^{\otimes k_2}$.  It is easy to see that $\beta \in W(k_1,k_2)$.  Indeed, we can choose a Borel subgroup of $G_2(\C)$ to be the one that stabilizes the flag $\C u \subseteq \C u \oplus \C v$, and then it is clear that $\beta$ is a highest weight vector in $W(k_1,k_2)$.  A computer calculation shows that, if $\beta_1$ is defined as above with $k_1 = 0$ and $k_2 = 4$, then the 
\[T_0 := \frac{1}{2}\left(\begin{array}{ccc} 2 & 1 &1 \\ 1 & 2 & 1 \\ 1 &1 & 2 \end{array}\right)\]
Fourier coefficient of $\Theta(\beta_1)$ is nonzero.  Similarly, if $\beta_2$ is defined as above with $k_1 = 2$ and $k_2 =4$, then the $T_0$ Fourier coefficient of $\Theta(\beta_2)$ is nonzero.

To actually do the computation on a computer, we proceed as follows.  First, we set $H$ to be the quaternion algebra over $\Q$ ramified at $2$ and the archimedean place.  Let $1,i,j,k$ be its usual basis.  We obtain the octonion algebra $\Theta = H \oplus H$ via the Cayley-Dickson construction using $\gamma =-1$.  This means that the addition in $\Theta$ is component-wise and the multiplication is 
\[(x_1,y_1) (x_2,y_2) = (x_1 x_2 + \gamma y_2^* y_1,y_2 x_1 + y_1 x_2^*).\]
Set $e = (0,1)$ and $h = \frac{1}{2}(i+j+k+e)$.  Then, the following are a $\Z$ basis of $R_\Theta$, Coxeter's ring \cite{coxeter}: $jh, e,-h,j,ih,1,eh,ke$.  These are the simple roots of the $E_8$ root lattice, with $jh$ the extended node, $1$ the branch vertex, and $e,-h,j,ih,1,eh,ke$ going along longways.

For $u,v$ we take elements inside of $\Theta \otimes \Q(\sqrt{-1})$ as
\[u = \frac{1}{2}((0,1)-\sqrt{-1}(0,i));\,\,\, v = \frac{1}{2}((0,j)-\sqrt{-1}(0,k)).\]

Finally, to compute the $T_0$ Fourier coefficient, where
\[T_0 = \frac{1}{2}\left(\begin{array}{ccc} 2a & f & e \\ f & 2b & d \\ e & d & 2c\end{array}\right),\]
we must explain how to enumerate the rank one $T \in J_R$ with $pr(T) = T_0$.  The point is that 
\[T = \left(\begin{array}{ccc} c_1 & x_3 & x_2^* \\ x_3^* & c_2 & x_1 \\ x_2 & x_1^* & c_3\end{array}\right)\]
being rank one with $pr(T) = T_0$ implies  $n(x_1) = c_2 c_3 = bc,$ $n(x_2) = c_3 c_1 = ca$, $n(x_3) = c_1 c_2 = ab$.  Thus one only must search through a finite list of $x_i$ (namely, those with these norms) to find all such $T$.
\end{proof}

In order to prove Corollary \ref{cor:alg}, we require Claim \ref{claim:introHD}, which we prove now.
\begin{proof}[Proof of Claim \ref{claim:introHD}] Recall that we assume $k_2 >0$, that $F$ is a level one Siegel modular form of weight $[k_1,k_2]$ and we wish to prove that if $F$ is in the image of the theta correspondence from $G_2^a$, then $F = \Theta(\alpha)$ for a level one algebraic modular form $\alpha(1) \in W(k_1,k_2)^{\Gamma_{G_2}}$.  

We first use the Howe Duality theorem of Gan-Savin \cite{ganSavinHD} to reduce to the case of eigenforms.  Thus write $F = \sum_j{F_j}$ as am orthogonal sum of eigenforms with each $F_j$ nonzero.  Then the Petersson inner product $(F_j,F) \neq 0$ for each $j$.  Let $\pi_{F_j}$ be the automorphic cuspidal representation of $\PGSp_6$ generated by $F_j$.  Then, by changing the order of the theta integral, one sees that the big theta lift $\Theta(\pi_{F_j})$ of $\pi_{F_j}$ to $G_2^a$ is nonzero.  By Howe Duality \cite{ganSavinHD} and the argument in \cite[Proposition 3.1]{ganAWSNotes}, $\tau_j:=\Theta(\pi_{F_j})$ is an irreducible representation of $G_2^a$.  

We now argue that $\Theta(\tau_j) = \pi_{F_j}$.  Let $W(k_1',k_2')$ the archimedean component of $\tau_j$, which a priori could depend upon $j$.  For any fixed vector $y_j$ in the finite part of $\tau_j \simeq \tau_{j,f} \otimes W(k_1',k_2')$, and any fixed vector $v_j$ in the finite part of $\Pi_{min} \simeq \Pi_{min,f} \otimes \Pi_{min,\infty}$ one can consider the map 
\[\Pi_{min,\infty} \otimes W(k_1',k_2') \rightarrow \mathcal{A}(\Sp_6)\]
given by the theta lift.  By the archimedean correspondence proved in Gross-Savin \cite{grossSavin}, the image is a holomorphic discrete series representation $\pi(k_1',k_2')$ with lowest $K_{\Sp_{6}}$-type $S^{[k_1',k_2']}$.  Because for some choice of $y_j$ and $v_j$ we must obtain a theta lift that is not orthogonal to $F_j$, we must have $k_1' = k_1$ and $k_2' = k_2$.

It now follows from \cite{grossSavin}, because $k_2' = k_2 > 0$, that $\Theta(\tau_j)$ consists of cusp forms.  Thus again by Howe Duality \cite{ganSavinHD} and the argument in \cite[Proposition 3.1]{ganAWSNotes}, $\Theta(\tau_j)$ is irreducible, so $\Theta(\tau_j) = \pi_{F_j}$.  By the results of \cite{grossSavin,magaardSavin,ganSavinHD} that apply to spherical representations,  it must be that $\tau_j$ is unramified at every finite place.  Thus we have a level one algebraic modular form $\alpha_j(1) \in W(k_1,k_2)^{\Gamma_{G_2}}$, and we must show that $\Theta(\alpha_j)$ is a nonzero multiple of $F_j$.  We have now reduced ourselves to the case of eigenforms, so will drop the $j$ from $\alpha_j$, $F_j$ and $\tau_j$.

Let $b_1,\ldots, b_N$ be a basis of $W(k_1,k_2)$, $\varphi_j$ the level one automorphic form on $G_2^a(\A)$ corresponding to $b_j$, and let $w_1,\ldots, w_{M}$ be a basis of $S^{[k_1,k_2]}$.  By changing variables in the integral defining the theta lift from $G_2^a$ to $\Sp_6$, one sees that there exists $v_{i,j}$ in $\Pi_{min}$ so that 
\[F = \sum_{i,j}{\Theta(v_{i,j},\varphi_j) \otimes w_i}.\]
Set $v' = \sum_{i,j}{v_{i,j}' \otimes w_i \otimes b_j} \in \Pi_{min} \otimes S^{[k_1,k_2]} \otimes W(k_1,k_2)$ and $\alpha = \sum_{j} \varphi_j \otimes b_j^\vee$.  Then we have
\[F = \int_{[G_2^a]}{\{\Theta_{v'}((g,h)),\alpha(h)\}\,dh}.\]
Let
\[v = \int_{K_{\Sp_6} \times G_2^a(\R)}{(k_1,k_2) \cdot v'\,dk_1\,dk_2}\]
where the action is diagonal on the minimal representation and the vector space $S^{[k_1,k_2]}\otimes W(k_1,k_2)$.  Because $\alpha$ is $G_2^a(\R)$-equivariant and $F$ is $K_{\Sp_{6}}$-equivariant, one has
\[F = \int_{[G_2^a]}{\{\Theta_{v}((g,h)),\alpha(h)\}\,dh}.\]
Because the $K_{\Sp_6} \times G_2^a(\R)$-type $(S^{[k_1,k_2]} \otimes W(k_1,k_2))^\vee$ appears in $\Pi_{min,\infty}$ with multiplicity one \cite{grossSavin}, we can write
\[v = \sum_{i,j}{v_{i,j,f} \otimes v_{i,j,\infty} \otimes w_i \otimes b_j}\]
where $v_{i,j,\infty} \in \Pi_{min,\infty}$ is the basis of $(S^{[k_1,k_2]} \otimes W(k_1,k_2))^\vee \subseteq \Pi_{min,\infty}$ dual to the basis $w_i \otimes b_j$.  We wish to show that we can set $v_{i,j,f}$ to be the spherical vector in $\Pi_{min,f}$ for all $i,j$ in this equality.

Fixing the vectors $v_{i,j,\infty}$ and letting the $v_{i,j,f}$ vary, the theta integral gives an equivariant map
\[\Pi_{min,f} \otimes \tau_f \rightarrow \pi_{F,f}.\]
This map is nonzero by assumption, so by $p$-adic Howe Duality again \cite{ganSavinHD} the map is uniquely determined up to scalar multiple.  Our goal is to show that the image of the spherical vector on the left-hand side is nonzero on the right-hand side.  We will do this by a global argument. 

It suffices to show that, for each prime $p$, in the unique up to scalar map $\Pi_{min,p} \otimes \tau_p \rightarrow \pi_{F,p}$, the image of the spherical vector is nonzero.  Because this map is unique up to scalar multiple, we must only find some $v \in \Pi_{min}$ which is spherical at $p$, some $\varphi \in \tau$ which is spherical at $p$, and so that $\Theta(v,\varphi) \neq 0$.  By \cite[Proposition 4.5]{grossSavin}, it suffices to check that $\alpha$ has an appropriate period.  More precisely, let $q$ be an odd prime with $q \neq p$, and let $B_q$ be the quaternion algebra over $\Q$ ramified at $q$ and infinity.  Then it suffices to show that $\alpha$ has a $B_q$ period.  But finally, by Bocherer-Das \cite{bochererDas2021}, for some odd prime $q \neq p$, $F$ has a nonzero Fourier coefficient associated to the maximal order in $B_q$.  Consequently, $\alpha$ does have a $B_q$ period and we have shown that the unique map $\Pi_{min,p} \otimes \tau_p \rightarrow \pi_{F,p}$ is nonzero on the spherical vector.  

Let
\[v_0 = \sum_{i,j}{v_{0,f} \otimes v_{i,j,\infty} \otimes w_i \otimes b_j}\]
where $v_{0,f}$ is the spherical vector in $\Pi_{min,f}$. We have shown
\[\int_{[G_2^a]}{\{\Theta_{v_0}((g,h)),\alpha(h)\}\,dh}.\]
is nonzero, and thus a nonzero multiple of $F$.  Because
\[\Theta(\alpha)(g):=\int_{[G_2^a]}{\{D^{k_1+2k_2}\Theta_{Kim}((g,h)),\alpha(h)\}\,dh} =  \int_{[G_2^a]}{\{\Theta_{v_0}((g,h)),\alpha(h)\}\,dh},\]
we obtain $\Theta(\alpha)$ is nonzero, and thus is nonzero multiple of $F$.  The claim is proved.
\end{proof}

\begin{proof}[Proof of Corollary \ref{cor:alg}]
We have explained, given $\beta \in W(k_1,k_2) \subseteq V_7^{\otimes k_1} \otimes (\wedge^2 V_7)^{\otimes k_2}$, how to compute individual Fourier coefficients of $\Theta(\beta)$.  It remains to explain how to enumerate a spanning set of $W(k_1,k_2)$.  To do this, we define elements $e_1,e_2,e_3, e_1^*, e_2^*,e_3^*,u_0 = \epsilon_1-\epsilon_2$ which are a basis of $V_7 \otimes \Q(\sqrt{-1})$, as follows.
\begin{itemize}
	\item $e_2 = \frac{1}{2}((0,1)-\sqrt{-1}(0,i))$
	\item $e_3^* = \frac{1}{2}((0,j)-\sqrt{-1}(0,k))$
	\item $e_3 = \frac{1}{2}((0,-j)-\sqrt{-1}(0,k))$.
	\item $e_2^* = \frac{1}{2}((0,-1)-\sqrt{-1}(0,i))$
	\item $\epsilon_1 = \frac{1}{2}((1,0)-\sqrt{-1}(i,0))$
	\item $\epsilon_2 = \frac{1}{2}((1,0)+\sqrt{-1}(i,0))$
	\item $e_1 = \frac{1}{2}((j,0)-\sqrt{-1}(k,0))$
	\item $e_1^* = \frac{1}{2}((-j,0)-\sqrt{-1}(k,0))$
\end{itemize}
Now, in terms of these elements, a basis for root spaces of $\g_2$ can be found in \cite{pollackG2}.  In particular, for the standard Borel chosen in \cite{pollackG2}, $v(k_1,k_2):=e_1^{\otimes k_1} \otimes (e_1 \wedge e_3^*)^{\otimes k_2}$ will be a highest weight vector for $W(k_1,k_2)$.  By the Poincare-Birkoff-Witt theorem, the representation $W(k_1,k_2)$ is spanned by $\{y_1^{m_1} y_2^{m_2} v(k_1,k_2)| m_1,m_2  \leq N(k_1,k_2)\}$ where here
\begin{itemize}
	\item $y_1,y_2$ span the negative root spaces of the simple roots of $\g_2$;
	\item the element $y_1^{m_1} y_2^{m_2} v(k_1,k_2)$ can be explicitly computed using the formulas of \cite{pollackG2};
	\item one can come up with easy bounds for the integer $N(k_1,k_2)$.	
\end{itemize}
The corollary follows.
\end{proof}

\section{Theorems on quaternionic modular forms}\label{sec:QMFmainthms}
In this section we prove Theorem \ref{thm:introQMF}, assuming a crucial technical result, Theorem \ref{thm:WqmfDer}, which is proved in section \ref{sec:qmfDer}.  We also prove Corollaries \ref{cor:introIntegrality}, \ref{cor:introG2FC}, and \ref{cor:Shimura}.
	
For $w \in W_J(\R)$ that is positive semi-definite, and $\ell \geq 1$ an integer, let $W_{w,\ell}$ be the generalized Whittaker function \cite{pollackQDS} on $G_J(\R) = E_{8,4}$ associated to $w$ and $\ell$.  Similarly, for $w_0$ a real binary cubic form that is positive semi-definite, let $W_{w_0,\ell}$ be the associated generalized Whittaker function on $G_2(\R)$.

Let now $\p$ denote the complexification of the $-1$ part for the Cartan involution on $\g(J)\otimes \R$ for the Cartan involution defined in \cite{pollackQDS}.  Recall that if $\varphi: G_J(\R) \rightarrow V$ is a smooth function, then $D \varphi \in C^\infty(G_J(\R);V \otimes \p^\vee)$ is defined as $D\varphi = \sum_{\alpha}{X_\alpha \varphi \otimes X_\alpha^\vee}$, where $\{X_{\alpha}\}_\alpha$ is a basis of $\p$.

\begin{theorem}\label{thm:WqmfDer}  Suppose $\ell = 4$ and $\beta \in V_{m \lambda_3}$ with $m \geq 0$.  Then there is a nonzero constant $B_{m}$ so that for all $w \in W_J$ rank one and $g \in G_2(\R)$, one has
	\[\{D^{2m}W_{w,\ell}(g),\beta\} =B_m \langle P_m(w),\beta \rangle_I W_{pr_I(w),\ell+m}(g).\]
\end{theorem}

Here, the $F_4$-equivariant pairing $\{\,,\,\}$ is defined as follows.  By virtue of the exceptional Cayley transform of \cite{pollackQDS} and the explanations of subsection \ref{subsec:dualpairs23}, one has $\p \simeq V_2^{\ell} \otimes W_J$, and $W_J = W_\Q \oplus V_2^s \otimes J^0$.  Thus
\[\p^{\otimes 2m} \rightarrow (V_2^{\ell})^{\otimes 2m} \otimes (V_2^s \otimes J^0)^{\otimes 2m} \rightarrow S^{2m}(V_2^{\ell}) \otimes \det(V_2^s)^{\otimes m} \otimes (\wedge^2 J^0)^{\otimes m}.\]
Thus if $\beta \in V_{m \lambda_3} \subseteq (\wedge^2 J^0)^{\otimes m}$, and $r \in S^{2\ell}(V_2^{\ell}) \otimes \p^{\otimes 2m}$ we obtain an element $\{r,\beta\}$ in $S^{2m+2\ell}(V_2^{\ell})$.

We can use the theorem to compute the Fourier expansion of the theta lift $\Theta(\alpha)$ of a level one algebraic modular form $\alpha$ on $F_4^I$.  We begin with the statement that these lifts are quaternionic modular forms of weight $4+m$ on $G_2$.

\begin{proposition}\label{prop:thetaQMF} Suppose $m \geq 0$, and $\beta \in V_{m \lambda_3}$.  Then
	\[\Theta(\beta):=\{D^{2m} \Theta_{Gan}(g,1),\beta\}\]
is a quaternionic modular form on $G_2$ of weight $4+m$.  If $m > 0$, it is a cusp form.
\end{proposition}
\begin{proof}  First note that, if $\gamma \in \Gamma_{I}$, then
\[\{D^{2m}\Theta_{Gan}(g,1),\gamma \cdot \beta\} = \{\gamma^{-1} \cdot D^{2m}\Theta_{Gan}(g,1),\beta\} = \{D^{2m}\Theta_{Gan}(g,\gamma), \beta\} = \{D^{2m}\Theta_{Gan}(g,1),\beta\}.\]
Thus if $\alpha$ is the level one algebraic modular form on $F_4^I$ with $\alpha(1) = \sum_{\gamma \in \Gamma_{I}}{\gamma \cdot \beta}$ and $\alpha(\gamma_E) = 0$, then 
\[\Theta(\alpha) = \int_{[F_4^I]}{\{D^{2m}\Theta_{Gan}(g,h),\alpha(h)\}\,dh} = \frac{1}{|\Gamma_{I}|}\{D^{2m}\Theta_{Gan}(g,1),\alpha(1)\} = \Theta(\beta).\]

Now, for $v \in \Pi_{min,\infty}$ and $w \in V_{m\lambda_3}$, consider the theta lift
\[\Theta(v,\varphi_{w})(g) = \int_{[F_4^I]}{\Theta_v(g,h)\varphi_{w}(h)\,dh}\]
where $\varphi_{w}(h) = \{\alpha(h),w\}$ is the level one automorphic function on $F_4^I(\A)$ associated to $w \in V_{m\lambda_3}$.
This lift gives an equivariant pairing $\Pi_{min,\infty} \otimes V_{m \lambda_3} \rightarrow \mathcal{A}(G_2)$.  By \cite{HPS}, it thus gives a map $\pi_{4+m} \rightarrow \mathcal{A}(G_2)$.  Finally, by the $K_{G_2}$-equivariance of $\Theta(\alpha)$, we see that $\Theta(\alpha)$ is the minimal $K$-type of this copy of $\pi_{4+m}$, so it is a quaternionic modular form of weight $4+m$.

We now show the cuspidality of the theta lifts if $m > 0$.  Let $P$ and $Q$, respectively, be the two standard maximal parabolic subgroups of $G_2$, so that $P$ is the Heisenberg parabolic.  Let $P_J$ be the Heisenberg parabolic of $G_J$, and $Q_J$ the standard maximal parabolic with $Q_J \cap G_2 = Q$.  In terms of the $F_4$ root system underlying the group $G_J$, with long roots $\alpha_1, \alpha_2$ and short roots $\alpha_3,\alpha_4$, $Q_J$ is the parabolic for which $\alpha_2$ is in its unipotent radical.  The Levi subgroup of $Q_J$ is of absolute Dynkin type $A_1 \times E_6$.  Write $P = M N$, $Q= L V$, $P_J = M_J N_J$, and $Q_J = L_J V_J$ for the Levi decompositions.

We must check that the constant terms $\Theta(v,\varphi)_N$ and $\Theta(v,\varphi)_V$ are $0$, if $\varphi$ is an automorphic form in a representation $\tau$ with $\tau_\infty = V_{m\lambda_3}$ with $m > 0$.  We first observe the following claim:
\begin{claim}\label{claim:constTerms} One has an equality of constant terms $\Theta_{v,N} = \Theta_{v,N_J}$ and $\Theta_{v,V} = \Theta_{v,V_J}$.
\end{claim}
Granting the claim for the moment, we obtain that $\Theta(v,\varphi)_N(g) = \int_{[F_4^I]}{\Theta_{v,N_J}(g,h)\varphi(h)\,dh}$ and $\Theta(v,\varphi)_V(g) = \int_{[F_4^I]}{\Theta_{v,V_J}(g,h)\varphi(h)\,dh}$.  The constant terms $\Theta_{v,N_J}$ and $\Theta_{v,V_J}$, restricted to their Levi subgroups, were determined in \cite{ganSW}.  For the first one, see page 174 of \cite{ganSW}, it is a sum of terms from a one-dimensional representation of $M_J$ and the minimal representation of $M_J$.  Both of these have integral $0$ against $\varphi$, because by \cite{HPS}, the representation $V_{m \lambda_3}$ with $m > 0$ does not participate in the theta correspondence for the dual pair $\SL_2 \times F_4^I \subseteq H_J^1$.  For the second one, see page 176 of \cite{ganSW}, the constant term restricted to $F_4^I$ is the trivial representation.  Thus this too has integral $0$ against $\varphi$.

It remains to explain the proof of Claim \ref{claim:constTerms}.  For the equality $\Theta_{v,N} = \Theta_{v,N_J}$, note that $\Theta_{v,N}$ is a sum of terms of the form $\sum_{w \in W_J: rk(w) \leq 1, pr_I(w) = 0}{\Theta_w(g)}$.  But, using that if $T \in J$ is rank one with $\tr(T) = 0$ then $T=0$, we find that the only $w \in W_J$ with $rk(w) \leq 1$ and $pr_I(w) = 0$ is $w=0$.  Thus $\Theta_{v,N}(g) = \Theta_{v,0}(g) = \Theta_{v,N_J}(g)$.  One makes a completely similar argument for the constant term $\Theta_{v,V}(g)$.  The proposition is proved.
\end{proof}

\begin{proof}[Proof of Theorem \ref{thm:introQMF}] Suppose $\alpha_I \in V_{m\lambda_3}^{\Gamma_I}$, and $\alpha$ is the level one algebraic modular form on $F_4^I$ with $\alpha(1) = \alpha_I$ and $\alpha_E = 0$.  Then it follows from Proposition \ref{prop:thetaQMF} that the theta lift $\Theta(\alpha)$ of $\alpha$ to $G_2$ is a quaternionic modular form of weight $4+m$, and cuspidal if $m > 0$.  Up to the constant $B_m$, its Fourier expansion is given exactly as in the statement of Theorem \ref{thm:introQMF}.  The general case, where $\alpha_E \neq 0$, is explained in subsection \ref{subsec:F4Ethetalifts} below.\end{proof}

We explain the proof of Corollary \ref{cor:introIntegrality}.
\begin{proof}[Proof of Corollary \ref{cor:introIntegrality}] Over $\C$, we have a decomposition $(\wedge^2 J^0)^{\otimes m} \otimes \C = V_{m \lambda_3} \oplus V'$, where $V'$ is $F_4^I$-stable.  Because $F_4^I$ is a pure inner form of split $F_4$, one can use the results of \cite[sections 2.1, 2.2]{BGW} to give such a decomposition over $\Q$: $(\wedge^2 J^0)^{\otimes m} = V_{m\lambda_3,\Q} \oplus V_{\Q}'$, where $V_{m\lambda_3,\Q}$ is a rational representation of the algebraic group $F_4^I$ whose complexification is $V_{m\lambda_3}$.  Now let $L_m$ be the intersection of $V_{m\lambda_3,\Q}$ with $(\wedge^2 J_R)^{\otimes m}$; it is immediately seen to be an integral lattice in $V_{m\lambda_3}$, so that $L_m \otimes \C = V_{m\lambda_3}$.  But now, from the explicit formula from Theorem \ref{thm:introQMF}, it is clear that if $\beta_I \in |\Gamma_I| L_m$, then $\Theta(\beta_I)$ has integral Fourier coefficients.  One makes a similar argument for $\Theta(\beta_E)$.  This proves the corollary.
\end{proof}

We now explain the proof of Corollary \ref{cor:introG2FC}.  To do so, we first construct a special $\beta_m \in V_{m\lambda_3}$.  Thus let
\begin{itemize}
	\item $K$ be an imaginary quadratic field, so that $H \otimes K$ is split, such as $K = \Q(\sqrt{-1})$.
	\item $a_2 \in \Theta \otimes K$, with $a_2 \neq 0$ but $n(a_2) = 0$
	\item $a_3 \in \Theta \otimes K$ with $n(a_3) = -1$
	\item $a_2' \in \Theta \otimes K$ with $n(a_2') = 1$ and $(a_2',a_2) = 1$.
\end{itemize}
Then, as in section \ref{sec:F4}, we set
\begin{itemize}
	\item $x = \left(\begin{array}{ccc} 1 & a_3 & a_2^* \\ a_3^* & -1 & a_3^* a_2^* \\ a_2 & a_2 a_3 & 0 \end{array}\right)$ 
	\item $z = \left(\begin{array}{ccc} 0 & 0 & (a_2')^* \\ 0 & 1 & 0 \\ a_2' & 0 & 0 \end{array}\right)$ 
	\item and $y = z \times x = \left(\begin{array}{ccc} 0 & (a_2')^*(a_2 a_3) & * \\ * & -1 & a_3^*(a_2')^* \\ a_2'-a_2 & * & 1 \end{array}\right)$.
\end{itemize}
We set $\beta_{K,m} = (x \wedge y)^{\otimes m}$.  It is proved in section \ref{sec:F4} that $\beta_{K,m} \in V_{m\lambda_3}$.

We require the following lemma.
\begin{lemma}\label{lem:w0nonzero} Let the notation be as above, with $m > 0$ even. Set $w_0 = u^2v-uv^2$.  Then
\[\left(\sum_{w \in W_{J_R}: rk(w) =1, pr_I(w) = w_0}{\langle P_m(w),\beta_{K,m}\rangle_I}\right) = 6.\]
\end{lemma}
\begin{proof} As explained in \cite[proof of Corollary 2.5.2]{pollackE8}, it follows from \cite[Proposition 5.5]{elkiesGrossIMRN} that the set of $w \in W_{J_R}$, with $rk(w) =1, pr_I(w) = w_0$ consists of the six elements $(0,e_{ii},-e_{jj},0)$, $i \neq j$, where $e_{kk}$ is the diagonal matrix in $J$ with a $1$ in the $k^{th}$ place and $0$'s elsewhere.  Now observe that
\begin{enumerate}
	\item $\langle e_{22} \wedge e_{33}, x \wedge y \rangle_I = -1$
	\item $\langle e_{33} \wedge e_{11}, x \wedge y \rangle_I = -1$
	\item $\langle e_{11} \wedge e_{22}, x \wedge y \rangle_I = -1$
\end{enumerate}
The lemma follows directly.
\end{proof}

\begin{proof}[Proof of Corollary \ref{cor:introG2FC}] The corollaries now follow immediately from Theorem \ref{thm:introQMF} and Lemma \ref{lem:w0nonzero}.\end{proof}

We now explain the proof of Corollary \ref{cor:Shimura}.
\begin{proof}[Proof of Corollary \ref{cor:Shimura}]
First note that the cubic ring $\Z \times \Z[t]/(t^2-pt-q)$ is associated to the binary cubic form $-y(x^2+pxy+qy^2)$.  Indeed, setting $\omega = (1,0)$ and $\theta = (0,t)$ in $\Z \times \Z[t]/(t^2-pt-q)$, this is a good basis, and computing its multiplication table gives rise to the binary cubic form $-y(x^2+pxy+qy^2)$.  The ring $\Z \times \Z_D$ is of this form with $p = D$ and $q= \frac{D-D^2}{4}$.

To explicitly compute the Fourier coefficients of $\Delta_{G_2}$, we now make a specific choice of $x,z,y$ as in the proof of Corollary \ref{cor:introG2FC}.  Namely, we take $K = \Q(\sqrt{-1})$, $a_2 = e_2$, $a_3 = u_0$, and $a_2' = e_2 - e_2^*$.  We obtain
\[x = \left(\begin{array}{ccc} 1 & x_3 & x_2^*\\ x_3^* & -1 & x_1 \\ x_2 & x_1^* & 0 \end{array}\right), \,\,\,\, y = \left(\begin{array}{ccc} 0 & y_3 & y_2^* \\ y_3^* &-1 & y_1 \\ y_2 & y_1^* & 1 \end{array}\right)\]
with
\begin{itemize}
	\item $x_1 = \frac{1}{2}(0,1-\sqrt{-1}i)$
	\item $x_2 = \frac{1}{2}(0,1-\sqrt{-1}i)$
	\item $x_3 = -\sqrt{-1}(i,0)$
	\item $y_1 = -\sqrt{-1}(0,i)$
	\item $y_2 = \frac{1}{2}(0,1+\sqrt{-1}i)$
	\item $y_3 = -\frac{1}{2}(1+\sqrt{-1}i,0)$.
\end{itemize}

Now, given the binary cubic $-y(x^2+pxy + qy^2)$, in order to compute the associated Fourier coefficient of $\Delta_{G_2}$, we must compute the set of $(0,T_1,T_2,q) \in W_{J_R}$ so that $\tr(T_2) = p$, $\tr(T_1) = -1$, and $(0,T_1,T_2,q)$ rank one.  We assume $q \neq 0$.  Then $(0,T_1,T_2,q)$ is rank one if and only if $n(T_2) = 0$ and $T_2^\# = q T_1$, which implies that $T_1$ is rank one.  As mentioned above, the set of rank one $T_1$ in $J_R$ with $\tr(T_1) = -1$ consists just of the three elements $-e_{11},-e_{22},-e_{33}$.  

Suppose that $T_1 = - e_{11}$.  Then $0 =(T_1,T_2)$ because $(T_2,T_2^\#) = 3n(T_2) = 0$.  So, the $(1,1)$ entry of $T_2 = 0$.  It now follows easily, using that $T_2^\# = -qe_{11}$, that $T_2$ is of the form
\[T_2 = \left(\begin{array}{ccc} 0 & 0 & 0 \\ 0 & c_2 & x \\ 0 & x^* & c_3 \end{array}\right)\]
with $c_2 c_3 - n(x) = -q$ and $c_2 + c_3 = p$.  Substituting $p = D$, $q = \frac{D-D^2}{4}$, $c_2 = \frac{1}{2}(v+D)$, $c_3 = \frac{1}{2}(-v+D)$, one finds that $v$ is an integer and $v^2+4n(x) = D$.  The contribution to the Fourier coefficient $a_{\Delta_{G_2}}(\Z \times \Z_D)$ for such a term is computed to be $(v+\sqrt{-1}(x,(0,i)))^2$.

One makes a completely similar calculation if $T_1 = -e_{22}$ or $T_1 = -e_{33}$.  One obtains
\[a_{\Delta_{G_2}}(\Z \times \Z_D) = \sum{[v+\sqrt{-1}(x,(0,i))]^2+[v-\sqrt{-1}(x,(0,i))]^2 + [v+\sqrt{-1}(x,(-i,0))]^2}\]
where the sum is over pairs $(v,x)$ in $\Z \oplus R_\Theta$ such that $v^2 + 4n(x) = D$.

But, this is clearly the Fourier coefficient of a harmonic theta function in $S_{13/2}(\Gamma_0(4))^{+}$, associated to the lattice $\langle 2 \rangle \perp 2E_8$.  This space has dimension $1$, spanned by $\delta(z)$.  By looking at the coefficient of $D=1$, one obtains the corollary.
\end{proof}

\subsection{Theta lifts}\label{subsec:F4Ethetalifts} In this subsection, we complete the proof of Theorem \ref{thm:introQMF}.  More specifically, we will now use our knowledge of the computation of $\{D^{2m}\Theta(g,1),\beta\}$ for arbitrary $\beta$ to compute $\{D^{2m}\Theta(g,\gamma_E),\beta'\}$. 

Let $\Theta = \Theta_{Gan}$ and denote by $\Theta_w$ the $w$-Fourier coefficient of this automorphic form.  Suppose $g_\infty \in G_J(\R)$. To begin, observe that
\begin{align*} a(w)(\gamma_E)W_{2\pi w}(g_\infty) &= \Theta_{w}(g_\infty \gamma_E) \\ &= \Theta_{w}(\gamma_E g_\infty) \\ &= \Theta_{w}(\delta_E^\Q \delta_E^{\R,-1}\delta_E^{\widehat{\Z}}g_\infty) \\ &= \Theta_{\delta_E^{\Q,-1} w}(\delta_E^{\R,-1}g_\infty) \\ &= a(\delta_E^{\Q,-1} w) W_{2\pi w \cdot \delta_E^{\Q}}(\delta_E^{\R,-1}g_\infty) \\ &= a(\delta_E^{\Q,-1} w) W_{2\pi w}(g_\infty).\end{align*}

Now, if $g \in G_2(\R)$, we prove
\[\{D^{2m}W_w(g),\beta\} = \langle P_m(w),\beta\rangle W_{w_0,4+m}(g)\]
where $w_0$ is the binary cubic $pr_I(w) = au^3 + (b,I^\#)u^2v + (c,I)uv^2 + dv^3$ if $w= (a,b,c,d)$.  Observe that $pr_I(\delta_Q w) = pr_E(w)$, where $pr_E(w) = au^3 + (b,E^\#)u^2 v + (c,E)uv^2 + dv^3$.  Combining these two facts, one obtains
\begin{align*}
\{D^{2m}\Theta(g,\gamma_E),\beta\} &= \sum_{rk(w) = 1} a(\delta_{E}^{\Q,-1} w) \langle P_m(w),\beta \rangle W_{pr_I(w)}(g) \\ &= \sum_{rk(w) = 1} a(w) \langle P_m(\delta_{E}^\Q w),\beta \rangle W_{pr_E(w)}(g) \\ &= \sum_{rk(w) = 1} a(w) \langle P_m(w), \delta_E^{\Q,-1}\beta \rangle_E W_{pr_E(w)}(g).
\end{align*}
Here, if $b \in J, c \in J^\vee \simeq J$, $x, y \in J$ then
\[\langle b \wedge c, x \wedge y\rangle_E = (b,x)_E(c,y)-(b,y)_E(c,x)\]
and one extends to $(\wedge^2 J)^{\otimes m}$.  We have used Lemma \ref{lem:Epair} to get
\[\langle P_w(\delta_E^\Q w), \beta\rangle_I = \langle P_m(w), \delta_{E}^{\Q,-1} \beta \rangle_E.\]
Thus, putting everything together,
\begin{align*}
\Theta(\alpha)(g) &:=\int_{[F_4^I]} \{D^{2m}\Theta(g,h),\alpha(h)\}\,dh \\ &= \frac{1}{|\Gamma_I|} \{D^{2m}\Theta(g,1),\alpha(1)\} + \frac{1}{|\Gamma_E|} \{D^{2m}\Theta(g,\gamma_E),\alpha(\gamma_E)\} \\ &= \frac{1}{|\Gamma_I|} \sum_{rk(w) =1}{a(w)\langle P_m(w),\alpha_I \rangle_I W_{pr_I(w)}(g)} + \frac{1}{|\Gamma_E|} \sum_{rk(w) =1}{a(w)\langle P_m(w),\alpha_E \rangle_E W_{pr_E(w)}(g)}.
\end{align*}

This completes the proof of Theorem \ref{thm:introQMF}.

\section{The exponential derivative}\label{sec:expder}
The purpose of this section is to prove Theorem \ref{thm:DiffExp}, which we restate here:

\begin{theorem} Suppose $g \in \Sp_{6}(\R)$ and $\beta \in W(k_1,k_2)$. Let $\ell \geq 0$ be an integer. There is a nonzero constant $B_{k_1,k_2}$, independent of $g$ and $T$, so that 
	\[j(g,i)^{\ell} \rho_{[k_1,k_2]}(J(g,i))\{D^{k_1+2k_2}(j(g,i)^{-\ell} e^{2\pi i (T, g\cdot i)}),\beta\} = B_{k_1,k_2} \{P_{k_1,k_2}(T),\beta\} e^{2\pi i (pr(T),g\cdot i)}.\]
	Moreover, $\{P_{k_1,k_2}(T),\beta\}$ lies in the highest weight submodule $S^{[k_1,k_2]}$ of $S^{k_1}(V_3) \otimes S^{k_2}(\wedge^2 V_3)$.
\end{theorem}

Our proof of this theorem is to, essentially completely explicitly, calculate the derivatives 
\begin{equation}\label{eqn:manyDerivs}X_{\alpha_1} \cdots X_{\alpha_n}(j(g,i)^{-\ell} e^{2\pi i (T, g \cdot i)}).\end{equation}
To do this, we use the Iwasawa decomposition $\h_{J} = n(J) + \m_J + \k_{E_7}$, write each $X_{\alpha}$ as a sum in terms of this decomposition, and calculate the derivatives for each piece.  

\subsection{Preliminaries}
Recall from above the Cayley transform $C_h \in H_J(\C)$, which satisfies
\begin{enumerate}
	\item $C_h^{-1} n_L(J \otimes \C) C_h = \p_J^+$
	\item $C_h^{-1} n_L^{\vee}(J\otimes \C) C_h = \p_J^{-}$
	\item $C_h^{-1} (\m_J \otimes \C) C_h = \k_{E_7}$.
\end{enumerate}

Given $\phi \in \m_J$, let $M(\phi)$ denote its action on $W_J$ \cite[section 3.4]{pollackQDS}.
\begin{proposition}\label{prop:Chidents} One has the following identities:
	\begin{enumerate}
		\item $C_h^{-1} n_L(x) C_h = n_L(x) - \frac{i}{2} M(\Phi_{1,x}) + \frac{1}{2}(n_L^\vee(x) - n_L(x))$
		\item $C_h^{-1} n_L^\vee(\gamma) C_h = n_L(\gamma) + \frac{i}{2} M(\Phi_{1,\gamma}) + \frac{1}{2}(n_L^\vee(\gamma)- n_{L}(\gamma))$
		\item If $\phi(1) = 0$, then $C_h^{-1}M(\phi) C_h = M(\phi)$.
		\item If $\phi = \Phi_{1,z}$, then $C_h^{-1}M(\phi) C_h = -i n_L(z) + i n_L^\vee(z)$.
	\end{enumerate} 
\end{proposition}
\begin{proof} Some of the facts needed to prove this are in Lemma 3.4.1 of \cite{pollackQDS}.  Also useful is the identity $\Phi_{Y,1} = \Phi_{1,Y}$ (see Theorem 4.0.10 in \cite{pollackNotes}).  To help with the proof of the second statement, one computes that 
	\[n_G^\vee(\delta) M(\phi) n_G^\vee(-\delta) = M(\phi) - n_L^\vee(\widetilde{\phi}(\delta)).\]
We omit the rest of the proof.
\end{proof}

\subsection{Some computations}
For $E \in J \otimes \C$, define $X^+(E) = i C_h^{-1} n_L(E) C_h$.  To warm up, we will compute $X^{+}(E)(j(g,i)^{-\ell}e^{2\pi i (T, g \cdot i)})$.  Let $h$ denote the function $h(g) = j(g,i)^{-\ell}e^{2\pi i (T, g \cdot i)}$.  Let $m_Y$ denote any element of $M_J$ for which $m_Y \cdot i = Y i \in \mathcal{H}_J$.

\begin{lemma}\label{lem:derivPieces} One has the following computations:
\begin{enumerate}
	\item $n_L(E) h(M(\delta,m)) = 2\pi i(T,m(E))h(M(\delta,m))$;
	\item $M(\Phi_{1,E}) e^{-2\pi (T, m(1))} = - 4\pi(T,m(E)) e^{-2\pi(T,m(1))}$;
	\item $M(\Phi_{1,E})j(m_Y,i)^{-\ell} = \ell(1,E)n(Y)^{\ell/2} = \ell(1,E)j(m_Y,i)^{-\ell}$;
	\item if $k = n_L^\vee(E) - n_L(E)$, then $kj(g,i)^{-\ell} = -\ell i (1,E) j(g,i)^{-\ell}$.
\end{enumerate}
\end{lemma}
\begin{proof} The first statement follows from the identity $M(\delta,m) n_L(E) M(\delta, m)^{-1} = n_{L}(m(E))$.  The second statement follows from the fact that $\Phi_{1,E}(E') = \{E,E'\}$, so that $\Phi_{1,E}(1) = 2 E$.  For the third statement, we have 
	\[M(\Phi_{1,E}) j(m_Y,i)^{-\ell} = \frac{d}{dt}|_{t=0} (\langle m_Y \exp(t \Phi_{1,E}) r_0(i), f \rangle)^{-\ell}.\]
Now $M(\Phi_{1,E})r_1(i) = (-(1,E),\ldots)$ so $M(\Phi_{1,E}) j(m_Y,i)^{-\ell}$ is the coefficient in $t$ of $(n(Y)^{-1/2}(1-(1,E)t))^{-\ell}$, which is $n(Y)^{\ell/2} \cdot \ell (1,E)$.  Thus $M(\Phi_{1,E}) j(m_Y,i)^{-\ell} = \ell (1,E) n(Y)^{\ell/2}$ as claimed.  For the fourth statement, observe that $n_L^\vee(E) r_1(i) = (i(1,E),\ldots)$, $n_L(E) r_1(i) = (0,\ldots)$, so $e^{tk}r_1(i) = (1+i(1,E)t,\ldots) + O(t^2)$.  As $j(ge^{tk},i) = j(g,i) j(e^{tk},i)$, the statement follows.
\end{proof}

Putting everything together, from Lemma \ref{lem:derivPieces} and Proposition \ref{prop:Chidents} one obtains
\begin{align*} (C_h^{-1} n_L(E) C_h) h(m) &=  2 \pi i (T,m(E)) h(m) + (-i/2)(-4\pi (T,m(E)) + \ell(1,E))h(m) \\ &\,+ \frac{1}{2}(-\ell i (1,E))h(m)\\ & = (4\pi i(T,m(E)) -\ell i (1,E))h(m) \\
(C_h^{-1} n_L^\vee(E) C_h) h(m) &= 2 \pi i (T,m(E)) h(m) + (i/2)(-4\pi (T,m(E)) + \ell(1,E))h(m) \\ &\,+ \frac{1}{2}(-\ell i (1,E))h(m) \\ &= 0.
\end{align*}
Thus
\[X^+_E h(m) = (-4\pi(T,m(E))+ \ell(1,E))h(m).\]

We now explain the computation of \eqref{eqn:manyDerivs}.  To setup the result, define $P_k: J^{\otimes k} \rightarrow T(J)$ inductively as follows: $P_0 =1$, and for $k \geq 0$,
\begin{align*}P_{k+1}(E_1,\ldots,E_k,E_{k+1}) &= P_{k}(E_1,\ldots E_k) \otimes E_{k+1} + \ell(1,E_{k+1}) P_k(E_1,\ldots,E_k) + \frac{1}{2}\{E_{k+1},P_k(E_1,\ldots E_k)\} \\ &\,+\frac{1}{2}P_k(\{E_{k+1},E_1,\ldots,E_k\}).\end{align*}
Here we write 
\[\{E,V_1 \otimes \cdots \otimes V_r\}:= \sum_{j=1}^r V_1 \otimes \cdots \otimes \{E,V_j\} \otimes \cdots \otimes V_r\]
and we interpret $\{E,C\} = 0$ if $C \in T^0(J)$ is constant.  Thus
\begin{enumerate}
	\item $P_1(E_1) = E_1 + \ell(1,E_1)$.
	\item $P_2(E_1,E_2) = E_1 \otimes E_2 + \ell(1,E_2) E_1 + \ell(1,E_1) E_2 + \{E_1,E_2\} + \ell^2(1,E_1)(1,E_2) + \frac{\ell}{2}(1,\{E_1,E_2\}).$
\end{enumerate}
Define now $w_{T,m}: T(J) \rightarrow \C$ as 
\[w_{T,m}(V_1 \otimes \cdots \otimes V_r) = (-4\pi)^r (T,m(V_1)) \cdots (T,m(V_r))\]
and extending to $T(J)$ by linearity.

\begin{proposition}\label{prop:hmderiv} Let the notation be as above.  Then $X_{E_{k}} \cdots X_{E_1} h(m) = w_{T,m}(P(E_1,\ldots,E_k)) h(m)$.
\end{proposition}
\begin{proof} We proceed by induction.  Observe that
	\begin{enumerate}
		\item $n_L(E_{k+1}) X_{E_{k}} \cdots X_{E_1} h(m) = 2\pi i (T,m E_{k+1}) X_{E_{k}} \cdots X_{E_1} h(m)$.
		\item $M(\Phi_{1,E})w_{T,m}(V_1 \otimes \cdots \otimes V_r) = w_{T,m}(\{E, V_1 \otimes \cdots \otimes V_r\})$.
		\item $M(\Phi_{1,E}) h(m) = (\ell (1,E) - 4\pi(T,m(E))) h(m)$.
		\item If $\mu \in K$, then $\mu X_{E_k} \cdots X_{E_1} h(m) = j(\mu,i)^{-\ell} X_{\mu \cdot E_k} \cdots X_{\mu \cdot E_1} h(m)$.
	\end{enumerate}
Now, if $k = n_L^\vee(E) - n_L(E)$, then 
		\begin{align*} [k, C_h^{-1} n_L(E') C_h] &= [-i C_h^{-1} M(\Phi_{1,E}) C_h, C_h^{-1} n_L(E') C_h] \\&= -i C_h^{-1} [M(\Phi_{1,E}),n_L(E')] C_h \\&= -i C_h^{-1} n_L(\{E,E'\}) C_h.\end{align*}
 Thus $[i k/2, X_{E'}] = \frac{1}{2} X_{\{E,E'\}}$. Consequently, 
		\[ (ik/2) (X_{E_k} \cdots X_{E_1}) h(m) = \frac{\ell}{2}(1,E) (X_{E_k} \cdots X_{E_1}) h(m) + \frac{1}{2} (\sum_{j=1}^{k} X_{E_k} \cdots X_{\{E,E_j\}} \cdots X_{E_1}) h(m).\]
As
\[X_E = i C_h^{-1} n_L(E) C_h = in_L(E) + \frac{1}{2}M(\Phi_{1,E}) + \frac{i}{2}(n_L^\vee(E)-n_L(E)),\]
one can now easily verify the proposition.
\end{proof}

We now prove:
\begin{lemma}\label{lem:Pnleading} Let $P_n: J^{\otimes k} \rightarrow T(J)$ be as above and let $\beta \in W(k_1,k_2) \subseteq V_7^{\otimes k_1} \otimes (\wedge^2 V_7)^{\otimes k_2}$.  Set $n = k_1 + 2k_2$.  Then
	\begin{equation}\label{eqn:Psum}\sum_{\alpha_i}{P_n(E_{\alpha_1},\ldots,E_{\alpha_n}) \{E_{\alpha_1}^\vee \otimes \cdots \otimes E_{\alpha_n}^\vee,\beta\}} = \sum_{\alpha_i}{E_{\alpha_1}\otimes \cdots \otimes E_{\alpha_n} \{E_{\alpha_1}^\vee \otimes \cdots \otimes E_{\alpha_n}^\vee,\beta\}} .\end{equation}
In other words, only the leading term contributes.  
\end{lemma}
\begin{proof} Suppose $h \in G_2(\R)$.  Observe
	\begin{align*} \sum_{\alpha_i}P_{n}(E_{\alpha_1},\ldots, E_{\alpha_{n}}) \{E_{\alpha_1}^\vee \otimes \cdots \otimes  E_{\alpha_{n}}^\vee,h \cdot \beta\} &= \sum_{\alpha_i}P_{n}(E_{\alpha_1},\ldots, E_{\alpha_{n}}) \{ h^{-1}E_{\alpha_1}^\vee \otimes \cdots \otimes  h^{-1}E_{\alpha_{n}}^\vee,\beta\} \\ &= \sum_{\alpha_i}P_{n}(h E_{\alpha_1},\ldots, h E_{\alpha_{n}}) \{E_{\alpha_1}^\vee \otimes \cdots \otimes  E_{\alpha_{n}}^\vee,\beta\} \\ &= \sum_{\alpha_i} h \cdot P_{n}(E_{\alpha_1},\ldots, E_{\alpha_{n}}) \{E_{\alpha_1}^\vee \otimes \cdots \otimes  E_{\alpha_{n}}^\vee,\beta\}
\end{align*}
where we have used that $P_{n}$ is equivariant for the action of $F_4 \supseteq G_2$.  (This follows, for example, from the recursive formula.)

Now, because of the above equivariance, we can assume $\beta = e_2^{k_1} \otimes (e_2 \wedge e_3^*)^{k_2}$.  Indeed, the two sides of the equality to be proved are linear in $\beta$ and $ e_2^{k_1} \otimes (e_2 \wedge e_3^*)^{k_2} \in W(k_1,k_2)$.

We now compute in the basis $\{u_0,e_1,e_2,e_3,e_1^*,e_2^*,e_3^*\}$ of $\Theta^0$.  More precisely, we take a basis of $J$ that is a basis of $H_3(F)$ union a basis $v_i \otimes u_j$ of $V_3 \otimes \Theta^0$, where $u_j \in \{u_0,e_1,e_2,e_3,e_1^*,e_2^*,e_3^*\}$ and $v_1,v_2,v_3$ is a basis of $V_3$.  We compute in this basis.  Then the only terms that contribute to the left-hand side of \eqref{eqn:Psum} are those with $E_{\alpha_i} = v_{\alpha_i} \otimes u_{\alpha_i}$ for all $i$.  Then, still in order for these terms to contribute in a nonzero way to \eqref{eqn:Psum}, we must have $u_{k} = e_2$ for $1 \leq k \leq k_1$ and $\{u_{\alpha_{2j-1}}, u_{\alpha_{2j}}\} = \{e_2,e_3^*\}$ for $2j-1 > k_1$.  Consequently, all the $E_{\alpha_i}$ that contribute to the sum in our basis satisfy $(1,E_{\alpha_i}) = 0$, $E_{\alpha_i}^\# = 0$, and $\{E_{\alpha_i},E_{\alpha_j}\} = 0$.  It now follows from our recursive formula for $P_n$ that only the leading term contributes.
\end{proof}

It follows immediately from Lemma \ref{lem:Pnleading} that 
\begin{align*}\label{eqn:Psum}\sum_{\alpha_i}{w_{T,m}(P_n(E_{\alpha_1},\ldots,E_{\alpha_n})) \{E_{\alpha_1}^\vee \otimes \cdots \otimes E_{\alpha_n}^\vee,\beta\}} &= \sum_{\alpha_i}{w_{T,m}(E_{\alpha_1}\otimes \cdots \otimes E_{\alpha_n}) \{E_{\alpha_1}^\vee \otimes \cdots \otimes E_{\alpha_n}^\vee,\beta\}}\\ &= (-4\pi)^n \{\widetilde{m}^{-1}(T^{\otimes n}),\beta\}.\end{align*}

Combining this with Proposition \ref{prop:hmderiv}, we obtain
\begin{proposition}\label{prop:diffExp} Suppose $\beta \in W(k_1,k_2)$, $n = k_1+2k_2$ and $m \in M_J$.  Then
\[\{D^n (j(m,i)^{-\ell}e^{-2\pi (T,m(1))}),\beta\} = B_{k_1,k_2} j(m,i)^{-\ell} \{\widetilde{m}^{-1}(T^{\otimes n}),\beta\} e^{-2\pi  (T,m(1))}\]
for a nonzero constant $B_{k_1,k_2}.$
\end{proposition}

We can now prove Theorem \ref{thm:DiffExp}.
\begin{proof}[Proof of Theorem \ref{thm:DiffExp}]
Observe that by appropriate $K_{\Sp_{6}}$-equivariance, and by equivariance for the unipotent radical of the Siegel parabolic, it suffices to check the equality of the statement of Theorem \ref{thm:DiffExp} when $g \in M$, the Levi of the Siegel parabolic.  But this follows from Proposition \ref{prop:diffExp} and the definition of $\rho_{[k_1,k_2]}(J(m,i))$; see the proof of Lemma \ref{lem:LeviSp6action} for the action of $\widetilde{m}^{-1}$ on $T$.  The proof of the theorem now follows from Lemma \ref{lem:HWsub} below.\end{proof}

Recall that $S^{[k_1,k_2]}$ is the kernel of the contraction
\[S^{k_1}(V_3) \otimes S^{k_2}(\wedge^2 V_3) \rightarrow S^{k_1-1}(V_3) \otimes S^{k_2-1}(\wedge^2 V_3) \otimes \det(V_3).\]
\begin{lemma}\label{lem:HWsub} If $\beta \in W(k_1,k_2)$, then $\{P_{k_1,k_2}(T),\beta\} \in S^{[k_1,k_2]}$.
\end{lemma}
\begin{proof} By equivariance and linearity, it suffices to verify the claim of the lemma for $\beta = e_2^{\otimes k_1} \otimes (e_2 \wedge e_3^*)^{\otimes k_2}$.  Now suppose $T = T_0 + x_1 \otimes v_1 + x_2 \otimes v_2 + x_3 \otimes v_3$.  Then 
	\[(T,e_2) = (x_1,e_2) v_1 + (x_2,e_2) v_2 + (x_3,e_2) v_3\]
and
\[(T \otimes T, e_2\wedge e_3^*) = (x_2 \wedge x_3, e_2 \wedge e_3^*) v_2 \wedge v_3 + (x_3 \wedge x_1,e_2\wedge e_3^*) v_3\wedge v_1 + (x_1 \wedge x_2, e_2\wedge e_3^*) v_1 \wedge v_2.\]
Thus contracting yields the term
\[(x_1,e_2)(x_2 \wedge x_3,e_2\wedge e_3^*)+(x_2,e_2) (x_3\wedge x_1,e_2 \wedge e_3^*) + (x_3,e_2) (x_1\wedge x_2, e_2\wedge e_3^*).\]
This is $(x_1 \wedge x_2 \wedge x_3, e_2 \wedge e_2 \wedge e_3^*) = 0$.  The lemma follows.
\end{proof}

\section{The quaternionic Whittaker derivative}\label{sec:qmfDer}

The goal of this section is to prove Theorem \ref{thm:WqmfDer}, which we restate here:
\begin{theorem} Suppose $\ell = 4$ and $\beta \in V_{m \lambda_3}$ with $m \geq 0$.  Then there is a nonzero constant $B_{m}$ so that for all $w \in W_J$ rank one and $g \in G_2(\R)$, one has
	\[\{D^{2m}W_{w,\ell}(g),\beta\} =B_m \langle P_m(w),\beta \rangle_I W_{pr_I(w),\ell+m}(g).\]
\end{theorem}

We begin with the following proposition.
\begin{proposition}  Let $\beta \in V_{m \lambda 3} \subseteq (\wedge^2 J^0)^{\otimes m}$.  Suppose $\ell =4$ and $w$ is rank one.  Then the function $F_{m,\beta}:G_2(\R) \rightarrow S^{2m+2\ell}(V_2)$ defined as $\{D^{2m}W_{w,\ell}(g),\beta\}$ is quaternionic.\end{proposition}
\begin{proof} Fix a rank one $w \in W_J$.  Let $\chi = \chi_w$ be the character of $N_J$ given as $\chi(n) = e^{i \langle w, \overline{n}\rangle}$. Here $\overline{n}$ is the image of $n$ in $W_J \simeq N_J^{ab}$, the abelianization of $N_J$. Let $L: \Pi_{min} \rightarrow \C$ be the unique (up to scalar multiple) moderate growth linear functional satisfying $L(n v) = \chi(n) L(v)$ for all $n \in N(\R)$ and $v \in \Pi_{min}$.  (Such an $L$ exists by a global argument: The global minimal representation has nonzero Fourier coefficients, so there is an $L_0$ for some $w_0$.  Now $L(v) := L_0(m_w v)$ for an appropriate $m_w$ is the desired functional.  The uniqueness of $L$ follows from \cite{pollackQDS}.)
	
Now let $x_j$ be a basis of the minimal $K = (\SU(2) \times E_7)/\mu_2$-type of $\Pi_{min}$ and $x_j^\vee$ in $S^8(V_2)$ the dual basis. Note that $W_w(g) = \sum_{j}L(g x_j) \otimes x_j^\vee$.  Then $E=\sum_j x_j \otimes x_j^\vee$ is in $(V_{min} \otimes S^8(V_2))^K$.  One obtains that $D^{2m} E \in (V_{min} \otimes S^8(V_2) \otimes \p^{\otimes 2m})^K$.  This latter space maps $K':=\SU(2) \times \SU(2)_{s} \times F_4^I(\R)$-equivariantly to 
\[S^{8+2m}(V_2) \otimes \det(V_2^s)^{\otimes m} \otimes (\wedge^2 J_0)^{\otimes m}.\]
Finally, mapping $(\wedge^2 J^0)^{\otimes m}$ to $V_{m \lambda_3}$, we obtain a $K'$ invariant element $E'$ in $V_{min} \otimes (S^{8+2m} \otimes \mathbf{1} \otimes V_{m\lambda_3})$.  By Huang-Pandzic-Savin \cite{HPS}, this $E'$ is either $0$ or the minimal type of $\pi_{m+4} \otimes V_{m\lambda_3}$.  Contracting now against some $\beta \in V_{m\lambda_3}$, we obtain some (possibly $0$) multiple of $S^{2m+8}(V_2) \subseteq \pi_{m+4}$.  Applying $L(g \cdot)$, it follows that $F_{m,\beta}$ is quaternionic.
\end{proof}

\subsection{General strategy}
As before, let $W_{w,\ell}(g)$ be the generalized Whittaker function of weight $\ell$ associated to $w \in W_J$, which is positive semi-definite.  Set $W_{w,\ell}^{-\ell}(g) = (x^{2\ell},W_w(g))$ be the component multiplying $y^{2\ell}/(2\ell)!$.  (See \cite{pollackQDS} for the definition of $x,y$.)  Here $(\,,\,)$ is an $\SL_2(\C)$-equivariant pairing on $S^{\cdot}(V_2)$.

Consider the quantity $\{D^{2m} W_{w,\ell}(m),\beta\}$.  Then \[\{D^{2m}W_{w}(m),k \cdot \beta\} = \{k^{-1} D^{2m}W_w(m),\beta\} = \{D^{2m}W_w(mk),\beta\}.\]
Thus if we can compute $\{D^{2m} W_w(m),\beta\}$ for $\beta = (x \wedge y)^{\otimes m}$ where $E_{x,y}$ is assumed singular and isotropic, then we can compute this quantity for general $\beta$.

What we actually do is compute $(x^{2m+2\ell},\{D^{2m}W_w(m),\beta\})$ for $\beta =(x \wedge y)^{\otimes m}$ where $E_{x,y}$ is assumed singular and isotropic. To setup the computation, we fix a basis $x_1 = x, x_2 = y, \ldots$ of $J^0$, fix a basis of $W_J = W_Q \oplus V_2^s \otimes J^0$ that is the union of bases of $W_\Q$ and of the tensor product basis $\{x_s,y_s\} \otimes \{x_1, x_2, \ldots\}$ of $V_2^s \otimes J^0$.  Then we fix our basis of $\p \simeq V_2^{\ell} \otimes W_J$ to be the tensor product basis of $\{x,y\}$ with the above fixed basis of $W_J$.

Now, it is clear that  $(x^{2m+2\ell},\{D^{2m}W_w(m),\beta\})$ only contains the terms in $D^{2m}$ where the $X_{\alpha_i}$ equal one of
\[ X_1 =y \otimes (x_s \otimes x), X_2=y \otimes (x_s \otimes y), X_3= y \otimes (y_s \otimes x), X_4=y \otimes (y_s \otimes y).\]
Note also that because $E_{x,y}$ is isotropic, the above Lie algebra elements all commute.  We obtain that
\begin{align*}(x^{2m+\ell},\{D^{2m}W_w(m),(x \wedge y)^{\otimes m}\}) &= 2^m \sum_{j=0}^{m}{ (-1)^j\binom{m}{j} (X_1 X_4)^{m-j} (X_2 X_3)^j W_{w,\ell}^{-\ell}(m)} \\ &= 2^m (X_1 X_4 - X_2 X_3)^m W_{w,\ell}^{-\ell}(m).\end{align*}

We now observe the following fact: if $k \in F_4^I(\R)$, then
\[(x^{2m+2\ell},\{D^{2m}W_{w,\ell}(m),k \cdot \beta\}) = (x^{2m+2\ell},\{k^{-1} D^{2m}\cdot W_{w,\ell}(m), \beta\}) = (x^{2m+2\ell},\{D^{2m}W_{w,\ell}(mk), \beta\}).\]
Thus, if we can compute the right hand side, then we can compute the left hand side.

To compute the quantity $(x^{2m+2\ell},\{D^{2m}W_{w,\ell}(mk), \beta\})$, we will represent $W_{w,\ell}^{-\ell}(g)$ as an integral, and differentiate under the integral sign.  This is inspired by the work of McGlade-Pollack \cite{mcGladePollack}. More exactly, set 
\[a_{\ell,v}(g)=\frac{(e_{\ell},g^{-1}v)^{\ell}}{(pr(g^{-1}v),pr(g^{-1}v))^{\ell+\frac{1}{2}}}.\]
Here $pr: \g(J) \otimes \R \rightarrow \su_2$ is the projection onto the Lie algebra of the long root $\SU_2$ and we write $(X,Y) = B(X,Y)$ for short.

We prove the following theorem.  To setup the theorem, recall from \cite{pollackQDS} that if $z \in J$ with $\tr(z) = 0$, then
\[V(z) = (0,iz,-z,0), V^*(z) = (0,-iz,-z,0).\]
Moreover, let $\nu: M_J \rightarrow \GL_1$ be the similitude character on the Levi of the Heisenberg parabolic of $G_J$.  There is an identification $M_J \simeq H_J$ (see \cite[Lemma 4.3.1]{pollackQDS}), and $\nu$ is the similitude character of $H_J$ via this identification.  I.e., for $h \in H_J$, $\nu$ satisfies $\langle h v,hv'\rangle = \nu(h) \langle v,v' \rangle$ for all $v,v' \in W_J$ and $\langle \,,\,\rangle$ Freudenthal's symplectic form on $W_J$.
\begin{theorem}\label{thm:intalv} Let the notation be as above.  Let $N_w \subseteq N$ consist of the $n$ with $\langle w, \overline{n} \rangle  = 0$.  Suppose $g \in M_J(\R)$, the Levi of the Heisenberg parabolic of $G_J(\R)$ and $w \in W_J$ is positive semidefinite.  Set $w_1 = g^{-1} w$.  There is a nonzero constant $B'_{\ell,m}$, independent of $w \geq 0$ and independent of $g$ so that the integral
\[\int_{\R = N_w\backslash N} e^{-i\langle w,\overline{n}\rangle}(X_1X_4-X_2X_3)^m a_{\ell,v}(ng)\,dn\]
is equal to
\[B'_{\ell,m} \nu(g)^{m} W_{w,\ell+m}^{-(\ell+m)}(g) (\langle V(x),w_1\rangle \langle V^*(y),w_1 \rangle - \langle V(y),w_1\rangle \langle V^*(x),w_1 \rangle)^{m}.\]
\end{theorem}

Note that the $m=0$ case of Theorem \ref{thm:intalv} represents the function $W_{w,\ell}^{-\ell}(g)$ as a integral, and the cases $m > 0$ of this theorem compute (by exchaning the order of integration and differentiation) the derivatives $(X_1X_4-X_2X_3)^m W_{w,\ell}^{-\ell}(g)$. We justify this exchange of integration and differentiation in subsection \ref{subsec:tech}.

\begin{corollary}\label{cor:DmWell} Suppose $g$ is in the Levi of the Heisenberg parabolic of $G_2(\R)$, and $\beta' \in V_{m\lambda_3}$.  Then
\[(x^{2m+2\ell},\{D^{2m} W_{w,\ell}(g),\beta'\}) = B_{m,\ell}'' W_{w,\ell+m}^{-(\ell+m)}(g) \langle P_{m}(w),\beta' \rangle_I.
\]
for a nonzero constant $B_{m,\ell}''$, independent of $w$ and $g$.
\end{corollary}
\begin{proof} Suppose $w = (a,b,c,d)$.  First consider the case $\beta' = \beta$.  We must simplify the quantity
\[\langle V(x),w_1\rangle \langle V^*(y),w_1 \rangle - \langle V(y),w_1\rangle \langle V^*(x),w_1 \rangle.\]
If $w_1 = w_1' + (0,b_1,-c_1,0)$ with $\tr(b_1) = \tr(c_1) = 0$ and $w_1' \in W_\Q$, then 
\begin{align*} \langle V(x),w_1\rangle \langle V^*(y),w_1 \rangle - \langle V(y),w_1\rangle \langle V^*(x),w_1 \rangle &= (i(x,c_1)-(x,b_1))(-i(y,c_1)-(y,b_1)) \\ &\,\,-(-i(x,c_1)-(x,b_1))(i(y,c_1)-(y,b_1)) \\ &= 2i ((x,b_1)(y,c_1)-(x,c_1)(y,b_1)) \\ &= 2i (x\wedge y,b_1\wedge c_1).
\end{align*}
Consequently, if $g \in \GL_2^s$, the Heisenberg Levi on $G_2$, and $w = w' + (0,b',-c',0)$, then $b_1 \wedge c_1 = \det(g)^{-1} b' \wedge c'$, so 
\[\langle V(x),w_1\rangle \langle V^*(y),w_1 \rangle - \langle V(y),w_1\rangle \langle V^*(x),w_1 \rangle = (2i)\det(g)^{-1} (x\wedge y, b'\wedge c').\]

In general, if $\beta' = \sum_{j}{\alpha_j k_j (x\wedge y)^{\otimes m}}$, with $\alpha_j \in \C$ and $k_j \in F_4^I(\R)$, one finds that
\[\sum_{j}\alpha_j (\langle V(x),k_j^{-1} g^{-1}w\rangle \langle V^*(y),k_j^{-1} g^{-1}w \rangle - \langle V(y),k_j^{-1} g^{-1}w\rangle \langle V^*(x),k_j^{-1} g^{-1}w \rangle)^m\]
is equal to
\begin{align*}(2i)^m \det(g)^{-m}\sum_j \alpha_j (k_j (x\wedge y), b'\wedge c')^m &= (2i)^m \det(g)^{-m} \sum_j \alpha_j (k_j (x\wedge y)^{\otimes m}, (b'\wedge c')^{\otimes m})\\ &= (2i)^m \det(g)^{-m} (\beta, (b'\wedge c')^{\otimes m}) \\ &= (-2i)^m \det(g)^{-m} (\beta,(b\wedge c)^{\otimes m})\end{align*}
if $w= (a,b,c,d)$.

The $\det(g)^{-m}$ cancels the $\nu(g)^m$ from Theorem \ref{thm:intalv}, giving the corollary.
\end{proof}

Theorem \ref{thm:WqmfDer} follows easily from Corollary \ref{cor:DmWell}:
\begin{proof}[Proof of Theorem \ref{thm:WqmfDer}] Both sides of the desired equality transform on the left under $N(\R)$ in the same way, and on the right under $K_{G_2}$ in the same way.  Moreover, for $\ell=4$, they are both known to be quaternionic functions.  Thus to prove their equality, it suffices to pair against $x^{2\ell}$, and evaluate on $g$ in the Levi of the Heisenberg parabolic.  But this is precisely what is done in Corollary \ref{cor:DmWell}, so the theorem is proved.	
\end{proof}

\subsection{Some derivatives} We now focus on proving Theorem \ref{thm:intalv}.  We must consider some derivatives of the function
\[a_{\ell,v}(g)=\frac{(e_{\ell},g^{-1}v)^{\ell}}{(pr(g^{-1}v),pr(g^{-1}v))^{\ell+\frac{1}{2}}}.\]

Suppose $X = x \otimes w \in \p$.  Then $X \cdot e_{\ell} = 0$, and thus $X\{g \mapsto (e_{\ell},g^{-1}v)\} = 0$.  Consequently, the function $(e_\ell,g^{-1}v)^{\ell}$ can be considered constant for the purposes of differentiating with respect to $x \otimes W \subseteq \p$.  We therefore must just differentiate the function
\[b_{\ell,v}(g) = (pr(g^{-1}v),pr(g^{-1}v))^{-(\ell+\frac{1}{2})}.\]

We obtain
\begin{align*}X b_{\ell,v}(g) &= (-(\ell+\frac{1}{2}))\cdot (pr(g^{-1}v),pr(g^{-1}v))^{-(\ell+\frac{3}{2})}\cdot(2)\cdot(-pr([X,g^{-1}v]),pr(g^{-1}v)) \\ &= (2\ell+1) (pr([X,g^{-1}v]),pr(g^{-1}v)) (pr(g^{-1}v),pr(g^{-1}v))^{-(\ell+\frac{3}{2})}.\end{align*}
We write $\Cvone{X} = (pr([X,g^{-1}v]),pr(g^{-1}v))$ and $\Cvtwo{X_1}{X_2} = (pr([X_1,g^{-1}v]),pr([X_2,g^{-1}v]))$ and $\Cv= (pr(g^{-1}v),pr(g^{-1}v)).$  Thus,
\[X b_{\ell,v}(g) = (2\ell+1)\Cvone{X} \Cv^{-(\ell+\frac{3}{2})}.\]

Now suppose $X_1 = X = x \otimes w_1$, and $X_2 = x \otimes w_2$ is such that $\Span\{w_1,w_2\}$ is isotropic and singular.  Because it is isotropic, $[X_1,X_2] = 0$.  Because it is singular, $\Span{X_1,X_2}$ consists of rank one elements of $\g(J)$.  Recall that $X \in \g(J)$ is rank one means $[X,[X,y]] + 2B(X,y)X = 0$ for all $y \in \g(J)$, where $B(\,,\,)$ is the bilinear form proportional to the Killing form defined in \cite[section 4.2.2]{pollackQDS}.  Thus, by symmetrizing and using that $X_1,X_2$ commute, one arrives at $-[X_1,[X_2,y]] = B(X_1,y)X_2 + B(X_2,y)X_1$.  Thus, $pr([X_1,[X_2,y]]) = 0$.  Using this, we differentiate $b_{\ell,v}(g)$ again to obtain
\[ X_2 X_1 b_{\ell,v}(g) = -(2\ell+1)\Cvtwo{X_1}{X_2}\Cv^{-(\ell+\frac{3}{2})} + (2\ell+1)(2\ell+3)\Cvone{X_1}\Cvone{X_2} \Cv^{-(\ell+\frac{5}{2})}.\]

For $0 \leq k \leq \lfloor n/2 \rfloor$, define $\mathcal{C}_{v,n,k}$ to be the symmetric sum of terms of the form 
\[\Cvtwo{X_1}{X_2} \cdots \Cvtwo{X_{2k-1}}{X_{2k}} \Cvone{X_{2k+1}} \cdots \Cvone{X_n}.\]
Then we have the following proposition.

\begin{proposition} Suppose $X_1,\ldots, X_n$ are such that 
	\begin{enumerate}
		\item $X_j = x \otimes w_j$ for all $j$ with
		\item $\Span(w_1, \ldots, w_n)$ singular and isotropic.
	\end{enumerate}
	Then 
	\[X_n \cdots X_1 b_{\ell,v}(g) = \sum_{k=0}^{\lfloor n/2 \rfloor}{(-1)^k 2^{n-k} (\ell+\frac{1}{2})_{n-k} \mathcal{C}_{v,n,k} \Cv^{-(\ell+n-k+\frac{1}{2})}}.\]
\end{proposition}
\begin{proof} We proceed by induction, noting that the proposition is true for $n=1$ and $n=2$ as checked above.  Note that 
	\[ X_{n+1} \Cv^{-(\ell+n-k+\frac{1}{2})} = 2 (\ell+n-k+\frac{1}{2}) \Cvone{X_{n+1}} \Cv^{-(\ell+n+1-k+\frac{1}{2})}.\]
	Thus, making the induction assumption, $X_{n+1} \cdots X_{1} b_{\ell,v}(g)$ is equal to 
	\[\sum_{k=0}^{\lfloor n/2 \rfloor}{(-1)^k 2^{n+1-k} (\ell+\frac{1}{2})_{n+1-k} \mathcal{C}_{v,n,k} \Cvone{X_{n+1}} \Cv^{-(\ell+n+1-k+\frac{1}{2})}} \]
	plus
	\[ \sum_{k=0}^{\lfloor n/2 \rfloor}{(-1)^{k-1} 2^{n-k} (\ell+\frac{1}{2})_{n-k}(-X_{n+1}\mathcal{C}_{v,n,k}) \Cv^{-(\ell+n-k+\frac{1}{2})}}.\]
	
	But now note that
	\[ \mathcal{C}_{v,n,j} \Cvone{X_{n+1}}+ (-X_{n+1} \mathcal{C}_{v,n,j-1}) = \mathcal{C}_{v,n+1,j}.\]
	The proposition follows.
\end{proof}

The next step is to calculate the $\Cvone{X_j}$ and $\Cvtwo{X_i}{X_j}$, then integrate.

\subsection{Some Lie algebra calculations}
We set $v = e \otimes w$ and $x = n g \in N_J(\R) M_J(\R)$.

Then if $n = \exp(e \otimes \overline{n})$, $n^{-1} v = e \otimes w + \langle w, \overline{n} \rangle \mm{0}{1}{0}{0}$, so $x^{-1} v = e \otimes w_1 + z_1 \mm{0}{1}{0}{0}$ where $w_1 =g^{-1} w$ and $z_1 = \nu(g)^{-1} \langle w, \overline{n} \rangle $.

Now recall the long root $\sl_2$ triple of $\g(J)$ given by 
\begin{itemize}
	\item $e_{\ell} = \frac{1}{4}(ie+f) \otimes r_0(i)$, 
	\item $f_{\ell} = \frac{1}{4} (ie-f) \otimes r_0(-i)$ and 
	\item $h_{\ell} = \frac{i}{2}(\mm{0}{1}{-1}{0} + n_L(-1) + n_L^\vee(1))$.
\end{itemize}
 For $X \in \g$ we have $pr(X) = B(X,f_{\ell}) e_{\ell} + \frac{1}{2} B(X,h_{\ell})h_{\ell} + B(X,e_{\ell}) f_{\ell}$.  Note that the exceptional Cayley transform $C$ \cite[section 5]{pollackQDS} takes $e_{\ell}, h_{\ell}, f_{\ell}$ to $\mm{0}{1}{0}{0}, \mm{1}{}{}{-1}, \mm{0}{0}{1}{0}$.  We obtain
\[pr(x^{-1}v) = \frac{1}{4} (\alpha^* e_{\ell} - iz_1 h_{\ell}-\alpha f_{\ell})\]
where\footnote{This is slightly different than how $\alpha$ is defined in \cite{pollackG2}.} $\alpha = \langle w_1, r_0(i) \rangle$. 

We now compute $\Cvone{X_j} = (pr([X_j,x^{-1}v]),pr(x^{-1}v))$.  We have
\begin{enumerate}
	\item $h_3 = \frac{1}{2} \mm{-1}{i}{i}{1}$, which under $C$ goes to $-ie \otimes (1,0,0,0)$.
	\item $h_1(X) = \frac{1}{2}(ie+f) \otimes V(X)$, which goes to $-ie \otimes (0,X,0,0)$
	\item $h_{-1}(Z) = \frac{i}{2}(n_L(Z)+n_L^\vee(Z)) + \frac{1}{2} M(\Phi_{1,Z})$, which goes to $-ie \otimes (0,0,Z,0)$
	\item $h_{-3} = \frac{1}{4} (ie-f) \otimes r_0(i)$, which goes to $-ie \otimes (0,0,0,1)$.
\end{enumerate}

Denote by $pr_{\p}$ the projection from $\g(J) \otimes \R \rightarrow \p$.  Let's suppose $C^{-1}(pr_{\p}(x^{-1}v)) = e \otimes w_{e} + f \otimes w_f$.  Write also $X_j = C(e \otimes w_j)$ with $w_j = (0,b_j,c_j,0)$.  Now 
\begin{align*} \Cvone{X_j} &= ([C^{-1}X_j, C^{-1}x^{-1}v], C^{-1}(pr(x^{-1}v))) \\ &= \frac{1}{4}([e \otimes w_j, e\otimes w_e + f\otimes w_f], \alpha^* E - iz_1 H - \alpha F) \\ &= \frac{1}{4}( \langle w_j, w_e \rangle  E - \frac{1}{2}\langle w_j, w_f \rangle H, \alpha^* E - iz_1 H - \alpha F) \\ &= -\frac{\alpha}{4} \langle w_j, w_e \rangle  + i \frac{z_1}{4} \langle w_j, w_f \rangle 
\end{align*}

Here $E = \mm{0}{1}{0}{0}, F = \mm{0}{0}{1}{0}$ and $H = \mm{1}{}{}{-1}$.

We now compute
\begin{align*}
\langle w_j, w_f \rangle &= (-e \otimes w_j, f \otimes w_f) \\ &= (-e \otimes w_j, C^{-1}(pr_\p(x^{-1}v))) \\ &= (-C(e \otimes w_j), x^{-1}v = e \otimes w_1 + z_1 E) \\ &= (-ih_1(b_j)-ih_{-1}(c_j), e \otimes w_1 + z_1 E) \\ &= -i(h_1(b_j),e \otimes w_1) \\ &= \frac{1}{2}((e-if) \otimes V(b_j), e\otimes w_1) \\ &= -\frac{i}{2} \langle V(b_j),w_1 \rangle. 
\end{align*}

Similarly we compute
\begin{align*}
\langle w_j,w_e \rangle = (f \otimes w_j,e \otimes w_e) \\ &= (f \otimes w_j, C^{-1}(pr_\p(x^{-1}v))) \\ &= (C(f\otimes w_j),x^{-1}v= e \otimes w_1 + z_1 E) \\ &= (i\overline{h_{-1}}(b_j) -i \overline{h_1}(c_j), e \otimes w_1 + z_1 E) \\ &= (-i \overline{h_1}(c_j), e \otimes w_1) \\ &= -i(\frac{1}{2}(-ie+f) \otimes V^*(c_j),e \otimes w_1) \\ &= - \frac{i}{2} \langle V^*(c_j),w_1 \rangle.
\end{align*}

We therefore obtain
\[\Cvone{X_j} = \frac{i \alpha}{8} \langle V^*(c_j), w_1 \rangle + \frac{z_1}{8}\langle V(b_j),w_1 \rangle.\]

We now compute $\Cvtwo{X_i}{X_j}$.  We assume $X_i = C(e \otimes w_i = e \otimes (0,b_i,c_i,0))$, $X_j = C(e \otimes w_j = e \otimes (0,b_j,c_j,0))$.  Then for $k = i,j$, we have
\begin{align*} pr([X_k,x^{-1}v]) &= pr([C(e \otimes w_k), pr_\p(x^{-1}v)]) \\ &= pr([C(e \otimes w_k), C(e \otimes w_e + f \otimes w_f)]) \\ &= (pr \circ C)([e \otimes w_k, e \otimes w_e + f \otimes w_f]) \\ &= (pr \circ C)(\langle w_k, w_e \rangle E - \frac{\langle w_k,w_f\rangle}{2}H) \\ &= \langle w_k, w_e \rangle e_{\ell}- \frac{\langle w_k,w_f\rangle}{2} h_{\ell}.\end{align*}
Thus
\[\Cvtwo{X_i}{X_j} = \frac{1}{2} \langle w_i,w_f \rangle \langle w_j,w_f \rangle = -\frac{1}{8} \langle V(b_i),w_1 \rangle \langle V(b_j),w_1 \rangle.\] 

We summarize what we've proved in a proposition.
\begin{proposition} Suppose $X_i = C(e \otimes w_i = e \otimes (0,b_i,c_i,0))$, $X_j = C(e \otimes w_j = e \otimes (0,b_j,c_j,0))$.  Then one has
	\begin{enumerate}
		\item $(e_\ell,x^{-1}v) = -\frac{\alpha}{4}$;
		\item $\Cv = -\frac{1}{8}(|\alpha|^2+z_1^2)$;
		\item $\Cvone{X_j} = \frac{i \alpha}{8} \langle V^*(c_j), w_1 \rangle + \frac{z_1}{8}\langle V(b_j),w_1 \rangle$;
		\item $\Cvtwo{X_i}{X_j}= -\frac{1}{8} \langle V(b_i),w_1 \rangle \langle V(b_j),w_1 \rangle$.
	\end{enumerate}
	Here $x=ng \in N_J(\R)M_J(\R)$, $\alpha = \langle w_1, r_0(i)\rangle$, $w_1 = g^{-1} w$ and $z_1 = \nu(g)^{-1} \langle w, \overline{n}\rangle$.
\end{proposition}

Rewriting the above, we have
\begin{enumerate}
	\item $-\Cvone{X_1} = -\frac{z_1}{8} \langle V(x),w_1 \rangle$
	\item $-\Cvone{X_2} = -\frac{z_1}{8} \langle V(y),w_1 \rangle$
	\item $-\Cvone{X_3} = i\frac{\alpha}{8} \langle V^*(x),w_1 \rangle$
	\item $-\Cvone{X_4} = i\frac{\alpha}{8} \langle V^*(y),w_1 \rangle$
	\item $-\Cvtwo{X_1}{X_1} = \frac{1}{8} \langle V(x),w_1\rangle \langle V(x),w_1 \rangle$
	\item $-\Cvtwo{X_1}{X_2} = \frac{1}{8} \langle V(x),w_1\rangle \langle V(y),w_1 \rangle$
	\item $-\Cvtwo{X_2}{X_2} = \frac{1}{8} \langle V(y),w_1\rangle \langle V(y),w_1 \rangle$
\end{enumerate}

\subsection{More calculations}
Putting together the work we have done above, we obtain the following lemma.  Define $\Cv^* = - \Cv$, $\Cvone{X_j}^* = -\Cvone{X_j}$, $\Cvtwo{X_i}{X_j}^* = -\Cvtwo{X_i}{X_j}$ and $\mathcal{C}_{v,n,k}^*(X_1 \cdots X_n)$ to be the quantity one obtains by replacing every $\Cvone{X_j}$ with a $\Cvone{X_j}^*$ and every $\Cvtwo{X_i}{X_j}$ with a $\Cvtwo{X_i}{X_j}^*$. 
\begin{lemma} Suppose $n=2m$.  Consider
	\[ \sum_{r=0}^{m}(-1)^r \binom{m}{r} \mathcal{C}_{v,2m,k}^*((X_1X_4)^{m-r}(X_2X_3)^{r}).\]
	This quantity is
	\[\left(\frac{1}{8}\right)^k \left(\frac{i\alpha}{8}\right)^m \left(\frac{-z_1}{8}\right)^{m-2k} \binom{m}{2k} \frac{(2k)!}{k!2^k} (\langle V(x),w_1\rangle \langle V^*(y),w_1 \rangle - \langle V(y),w_1\rangle \langle V^*(x),w_1 \rangle)^{m}.
	\]
	(Interpret this formula to mean that it is $0$ if $2k > m$.)
\end{lemma}
\begin{proof} We evaluate $\mathcal{C}_{v,2m,k}^*((X_1X_4)^{m-r}(X_2X_3)^{r})$.  The point is that every nonzero term that goes into the definition of this sum is the same.  One obtains that any such term is equal to
	\[\left(\frac{1}{8}\right)^k \left(\frac{i\alpha}{8}\right)^m \left(\frac{-z_1}{8}\right)^{m-2k} (\langle V(x),w_1\rangle \langle V^*(y),w_1 \rangle)^{m-r} (\langle V(y),w_1\rangle \langle V^*(x),w_1 \rangle)^{r}.\]
	There are
	\[\frac{1}{k!}\binom{m}{2} \binom{m-2}{2} \binom{m-4}{2} \cdots \binom{m-2k+2}{2} = \binom{m}{2k} \frac{(2k)!}{k!2^k}\]
	such nonzero terms.  Thus
	\begin{align*}\mathcal{C}_{v,2m,k}^*((X_1X_4)^{m-r}(X_2X_3)^{r}) &= \binom{m}{2k} \frac{(2k)!}{k!2^k} \left(\frac{1}{8}\right)^k \left(\frac{i\alpha}{8}\right)^m \left(\frac{-z_1}{8}\right)^{m-2k} \\ &\,\times (\langle V(x),w_1\rangle \langle V^*(y),w_1 \rangle)^{m-r} (\langle V(y),w_1\rangle \langle V^*(x),w_1 \rangle)^{r}\end{align*}
	and the lemma follows.
\end{proof}

Putting everything together, we obtain the following proposition.  Set 
\[b_{\ell,v}^*(g) =(-(pr(g^{-1}v),pr(g^{-1}v)))^{-(\ell+\frac{1}{2})}.\]

\begin{proposition} One has 
	\begin{align*} (X_1 X_4-X_2 X_3)^{m} b_{\ell,v}^*(g) &= 8^{\ell+\frac{1}{2}} (-i\alpha)^m(\langle V(x),w_1\rangle \langle V^*(y),w_1 \rangle - \langle V(y),w_1\rangle \langle V^*(x),w_1 \rangle)^{m} \\ &\,\times \left(\sum_{k=0}^{\lfloor m/2 \rfloor} (-1)^{k} 2^{2m-k} z_1^{m-2k} \binom{m}{2k} \frac{(2k)!}{k!2^k} (\ell + \frac{1}{2})_{2m-k} (|\alpha|^2+z_1^2)^{-(\ell+2m-k+\frac{1}{2})}\right).\end{align*}
\end{proposition}

For a positive real number $\beta$ set
\[I_{\ell,m}(\beta) = \int_{\R}{e^{-iz} \left(\sum_{k=0}^{\lfloor m/2 \rfloor} \frac{C_{m,k} z^{m-2k} (\ell+\frac{1}{2})_{2m-k}}{(\beta^2+z^2)^{\ell+2m-k+\frac{1}{2}}}\right)\,dz}\]
where
\[C_{m,k} = (-1)^k 2^{2m-k} \binom{m}{2k} \frac{(2k)!}{k!2^k}.\]

Set $\alpha_{\ell,v}^*(g) = -(e_\ell,g^{-1}v) b_{\ell,v}^*(g)$.  We have proved: 
\begin{proposition} Set $z = \nu(g) z_1$. One has
	\begin{align*}\int_{\R = N_w\backslash N} e^{-i\langle w,\overline{n}\rangle}(X_1X_4-X_2X_3)^m a_{\ell,v}^*(ng)\,dn &=  (-i)^m 2^{\ell+\frac{3}{2}} (\nu(g)\alpha)^{\ell+m} |\nu(g)| \nu(g)^{\ell+2m} I_{\ell,m}(|\nu(g)\alpha|)\\ &\,\times (\langle V(x),w_1\rangle \langle V^*(y),w_1 \rangle - \langle V(y),w_1\rangle \langle V^*(x),w_1 \rangle)^{m}.\end{align*}
\end{proposition}

The following proposition now implies Theorem \ref{thm:intalv}.
\begin{proposition}\label{prop:Iint} One has
	\[I_{\ell,m}(\beta) = \frac{2^{1-\ell} i^m}{(\frac{1}{2})_{\ell}} \beta^{-(\ell+m)} K_{\ell+m}(\beta).\]
\end{proposition}

\subsection{Evaluation of an integral} In this subsection, we prove Proposition \ref{prop:Iint}. 

For $m=0$, from equation (5) of \cite{mathworldKBessel}, one obtains
\[I_{\ell,0}(\beta) = \frac{2}{(2\ell-1)(2\ell-3) \cdots (3)(1)} \beta^{-\ell} K_{\ell}(\beta).\]

Set, for $r > 0$, $v \in \Z_{>0}$,
\[I_{v,0}(r,\beta) = \int_{\R}{e^{-irz} (\beta^2+z^2)^{-(v+\frac{1}{2})}\,dz}.\]
Changing variables, one obtains
\[I_{v,0}(r,\beta) =D_v(r/\beta)^v K_v(r\beta).\]
where $D_v =2^{-v+1} \frac{1}{(\frac{1}{2})_{v}}$.
Differentiating under the integral sign, and using that
\[C_{m,k} (\ell+\frac{1}{2})_{2m-k} (-i)^{m-2k} D_{\ell+2m-k} = \frac{2^{1-\ell}(-i)^m}{(\frac{1}{2})_{\ell}}, \]
we get
\[I_{\ell,m}(\beta) = \frac{2^{1-\ell}(-i)^m}{(\frac{1}{2})_{\ell}} \left(\sum_{k=0}^{\lfloor m/2 \rfloor} \binom{m}{2k} \frac{(2k)!}{k!2^k} \partial_r^{m-2k}( (r/\beta)^{\ell+2m-k} K_{\ell+2m-k}(r\beta))|_{r=1}\right).\]
Finally, making the variable change $u = r \beta$, one gets
\[I_{\ell,m}(\beta) = \frac{2^{1-\ell}(-i)^m}{(\frac{1}{2})_{\ell}} \left(\sum_{k=0}^{\lfloor m/2 \rfloor} \binom{m}{2k} \frac{(2k)!}{k!2^k} \beta^{-2\ell-3m} \partial_u^{m-2k}( u^{\ell+2m-k} K_{\ell+2m-k}(u))|_{u=\beta}\right).\]

Set
\[c_n^j = \frac{n!}{j!(n-2j)!2^j} = \binom{n}{2j} \frac{(2j)!}{j!2^j}.\]
\begin{lemma} Suppose $b \geq n$.  One has
	\[\partial_u^n(u^b K_b(u)) = \sum_{j=0}^{\lfloor n/2 \rfloor}{(-1)^{n-j} c_n^j u^{b-j} K_{b-n+j}(u)}.\]
\end{lemma}
\begin{proof} The proof is by induction, using the recurrence $c_{n+1}^j = c_n^j + (n-2j+2) c_n^{j-1}$.  See also \cite{OEIS}.
\end{proof}

From the immediately verified identity
\[\sum_{0 \leq j+k =\text{ fixed} \leq \lfloor m/2 \rfloor} (-1)^j c_m^k c_{m-2k}^j = 0\]
if $j+k > 0$, we obtain Proposition \ref{prop:Iint}.

\subsection{Technical justification}\label{subsec:tech} We still must justify our differentiation under the integral.  In other words, we must justify the identity
\[\int_{N_w\backslash N \simeq \R}{Z_1 \cdots Z_r\left(e^{-i\langle w,\overline{n}\rangle}(a_{\ell,v}(ng)\right)\,dn} = Z_1 \cdots Z_r\left(\int_{N_w\backslash N \simeq \R}{e^{-i\langle w,\overline{n}\rangle}a_{\ell,v}(ng)\,dn}\right).\]
Here the $Z_i$ are in $\p$ and differentiate with respect to the $g$ variable.

Write $n= n(z)$ if $n \mapsto z$ under the identification $N_w\backslash N \simeq \R$.  To do this justification, it suffices to show (for all non-negative integers $r$) that there is a small neighborhood $U$ of $g$, and a positive function $A_r(z)$ on $\R$, so that for all $x \in U$,
\[|Z_1 \cdots Z_r a_{\ell,v}(n(z)x)| \leq A_r(z) \text{ and } \int_{\R}A_r(z)\,dz < \infty.\]
The function $A_r(z)$ can depend upon $g$ and $Z_1, \ldots, Z_r$, but we drop them from the notation.

It is easy to see (e.g., by induction) that the derivative $Z_1 \cdots Z_r(a_{\ell,v}(nx))$ is of the form
\[\frac{(e_{\ell},x^{-1}n^{-1}v)^{\ell}}{(pr(x^{-1}n^{-1}v),pr(x^{-1}n^{-1}v))^{\ell+r+\frac{1}{2}}} P_{Z_1,\ldots,Z_r}(x,n)\]
where $P_{Z_1,\ldots,Z_r}(x,n)$ consists of sums of products of terms of the form $(pr([Z_j,x^{-1}n^{-1}v],pr(x^{-1}n^{-1}v))$, $(pr([Z_j,x^{-1}n^{-1}v],pr([Z_k,x^{-1}n^{-1}v]))$ etc.  The key point is that, if $n = n(z)$, then when written as a polynomial in $z$, $P_{Z_1,\ldots,Z_r}(n,z)$ has degree at most $2r$.  The coefficients of this polynomial depend on $x$ and $Z_1,\ldots, Z_r$, but are easily seen to be bounded for $x$ in a small compact set around $g$.

To finish the proof, we now must bound
\[\left|\frac{(e_{\ell},x^{-1}n^{-1}v)^{\ell}}{(pr(x^{-1}n^{-1}v),pr(x^{-1}n^{-1}v))^{\ell+r+\frac{1}{2}}}\right| \leq ||pr(x^{-1}n^{-1}v)||^{-(\ell+2r+1)}.\]
Here $||v|| = |B(v,\theta(v))|^{1/2}$ is the $K$-equivariant norm on $\g(J) \otimes \R$, where $\theta$ is the Cartan involution.  Write $g = n_g m_g k_g$ in the Iwasawa decomposition.  We take small open neighborhoods around $n_g, m_g$ and $k_g$, and let $U$ be the product of these neighborhoods.  Then, if $n = n(z)$, $||pr(x^{-1}n^{-1}v)||$ is bounded away from $0$ for $z$ small, independent of $x$, and is bounded below by $(|z|^2 + |\alpha_0|^2)^{1/2}$ for $z$ large, with $\alpha_0$ independent of $x \in U$.  The existence of $A_r(z)$ with the desired properties follows.

\section{Arithmeticity of modular forms on $G_2$}\label{subsec:arithG2}
The purpose of this section is to prove Theorem \ref{thm:introArith}, which we recall here.  Suppose $\varphi$ is a cuspidal quaternionic modular form on $G_2$ of weight $\ell \geq 1$.  Then 
\[\varphi_Z(g_f g_\infty) = \sum_{\chi}{a_\chi(g_f)W_\chi(g_\infty)}\]
is its Fourier expansion. The locally constant functions $a_\chi: G_2(\A_f) \rightarrow \C$ are its Fourier coefficients.  We say that $\varphi$ has Fourier coefficients in a ring $R$ if $a_\chi(g_f) \in R$ for all characters $\chi$ of $N(\Q)\backslash N(\A)$ and all $g_f \in G_2(\A_f)$.  We write $S_{\ell}(G_2; R)$ for the space cuspidal quaternionic modular forms on $G_2$ of weight $\ell$ with Fourier coefficients in $R$.

Let $\Q_{cyc} = \Q(\mu_\infty)$ be the cyclotomic extension of $\Q$.
\begin{theorem}\label{thm:ArithG2} Suppose $\ell \geq 6$ is even.  Then there is a basis of the cuspidal quaternionic modular forms of weight $\ell$ with all Fourier coefficients in $\Q_{cyc}$.  In other words, $S_{\ell}(G_2,\C) = S_{\ell}(G_2,\Q_{cyc}) \otimes_{\Q_{cyc}} \C$.
\end{theorem}

The proof of this theorem has the following steps:
\begin{enumerate}
	\item Set $S_{\ell}(G_2)_{\Theta}$ the subspace of $S_{\ell}(G_2,\C)$ consisting of theta lifts from algebraic modular forms on $F_4^I$.  It is clear that it is a $G_2(\A_f)$ submodule.  Moreover, as an application of Theorem \ref{thm:WqmfDer}, it is easy to see that $S_{\ell}(G_2)_{\Theta}$ is defined over $\Q_{cyc}$, i.e., that it has a basis consisting of elements with Fourier coefficients in $\Q_{cyc}$.
	\item To show that every element of $S_{\ell}(G_2;\C)$ is a theta lift, one uses the Siegel-Weil theorem of \cite{pollackSW} and the Rankin-Selberg integral of \cite{gurevichSegal,Segal2017}.
	\item For the previous step to go through, a certain archimedean Zeta integral must be shown to be non-vanishing.  One shows the non-vanishing of this integral by a global method, using Corollary \ref{cor:introG2FC}.
\end{enumerate}

We will break the proof into various lemmas.
\begin{lemma} Suppose $\phi \in V_{min}$ is $\phi = \sum_{j}{\mu_j g_j \cdot \phi_{0}}$ with $\mu_j \in \Q$ and $\phi_0$ the spherical vector.  Let $\mathcal{A}(F_4^I,U;V_{m\lambda_3})$ be the algebraic modular forms on $F_4^I$ of level $U \subseteq F_4^I(\A_f)$ and for the representation $V_{m\lambda_3}$. Then there is a lattice $\Lambda \subseteq \mathcal{A}(F_4^I,U;V_{m\lambda_3})$ so that if $\alpha \in \Lambda$, then $\Theta_\phi(\alpha) \in S_{\ell+m}(G_2;\C)$ has Fourier coefficients in $\Q_{cyc}$.
\end{lemma}
\begin{proof}
Write 
\begin{equation}\label{eqn:thetaliftLevel}\varphi(g) = \int_{[F_4^I]}{\{D_{\p}^{2m} \Theta_\phi(g,h),\alpha(h)\}\,dh}.\end{equation}

Now 
\[F_4^I(\Q) \backslash F_4^I(\R) F_4^I(\A_f)/U = \cup_{j=1}^{N} \Gamma_{j}\backslash F_4^I(\R) \gamma_j\]
where $\gamma_j \in F_4^I(\A_f)$.  In other words,
\[ F_4^I(\Q)\backslash F_4^I(\A) = \cup_{j=1}^{N} \Gamma_j \backslash F_4^I(\R) \gamma_j U.\]
Thus the integral over $[F_4^I]$ is a finite sum of terms $c_j \{D_{\p}^{2m} \Theta_\phi(g,1),\alpha(\gamma_j)\}$ where $c_j = \frac{meas(U)}{|\Gamma_j|}$ is rational.  But now Theorem \ref{thm:WqmfDer} implies that if the Fourier coefficients of $\Theta_\phi(g)$ are in some ring $R$, and $\alpha(\gamma_j)$ is in an appropriate lattice, then the Fourier coefficients of $\{D_{\p}^{2m} \Theta_\phi(g,1),\alpha(\gamma_j)\}$ are in $R$.  Thus we obtain the fact that the Fourier coefficients of theta lifts can all be made in $\Q_{cyc} = \Q(\mu_\infty)$, as soon as we prove the same result for the Fourier coefficients of $\Theta_\phi(g)$.

For the latter, simply observe the following identities: suppose $g_f \in G_J(\A_f)$.  Then we can write $g_f  = u_f m_f k_f$ with $u_f$ in the $N_J$, $k_f \in G_J(\widehat{\Z}) = K_f$, and $m_f \in H_J(\A_f)$.  Then we can further write $m_f = (m_\Q m_\R^{-1}) k'$ with $m_\Q \in H_J(\Q)$ and $k' \in K_f$.  This follows from strong approximation on the simply connected group $H_J^1$.  Thus, if $a(\omega)(g_f)$ denotes a Fourier coefficient of $\Theta_{Gan}(g)$ (level one), then
\[a(\omega)(g_f) = \psi(\langle \omega, u_f\rangle) a(\omega)(m_\Q m_\R^{-1}).\]
Additionally, 
\[a(\omega)(m_\Q m_\R^{-1}) W_\omega(g_\infty)=\Theta_\omega(m_\Q m_\R^{-1} g_\infty)= \Theta_{\omega \cdot m_\Q}(m_\R^{-1}g_\infty) = a_{\omega \cdot m_\Q}(1) W_{\omega \cdot m_\Q}(m_\R^{-1} g_\infty).\]
This last term is $\det(m_\R)^{-\ell} |\det(m_\R)|^{-1} a_{\omega \cdot m_\Q}(1) W_\omega(g_\infty).$  Thus since all the $a_{\omega'}(1)$ are integral, all $a(\omega)(g_f) \in \Q_{cyc}$.  It thus follows that if $\phi = \sum_{j}{\mu_j g_j \cdot \phi_{0}}$ with $\mu_j \in \Q$, then all the Fourier coefficients of $\Theta_\phi$ are in $\Q_{cyc}$.  This completes the argument.
\end{proof}

Write $S_{\ell}(G_2)_{\Theta}$ for the space of all lifts $\varphi$ as in equation \eqref{eqn:thetaliftLevel}.
\begin{lemma} The subspace $S_{\ell}(G_2)_{\Theta}$ is a $G_2(\A_f)$-submodule.
\end{lemma}
\begin{proof}
This is clear.
\end{proof}

Recall the projection $\p^{\otimes m} \rightarrow S^{2m}(V_{\ell}) \otimes (\wedge^2 J^0)^{\otimes m}.$ Let $W \subseteq (\wedge^2 J^0)^{\otimes m}$ be the set of all $w \in  (\wedge^2 J^0)^{\otimes m}$ for which $\{w, v\} = 0$ for all $v \in V_{m \lambda_3}$.  We set $V_m^* = (\wedge^2 J^0)^{\otimes m}/W$, and let $P:\p^{\otimes m} \rightarrow S^{2m}(V_{\ell}) \otimes V_m^*$ be the composite projection.

To prove that $S_{\ell}(G_2)_{\Theta} = S_{\ell}(G_2;\C)$ it suffices to show that if $\varphi \in S_{\ell}(G_2;\C)$ generates an irreducible representation $\pi$, then
\begin{equation}\label{eqn:thetaLift2}\int_{[G_2]}{\{\varphi(g),P(D_\p^{2m}\Theta_\phi(g,h))\}\,dg} \neq 0\end{equation}
for some $\varphi \in V_{min}$.  Indeed, in this case, the submodule $S_{\ell}(G_2)_{\Theta}$ has orthocomplement equal to $0$.  Moreover, we can assume that $\varphi$ is a pure tensor in $\pi$.

Suppose $E$ is a totally real cubic \'etale extension of $\Q$, and let $S_E$ be the group of type $\Spin_8$ defined in terms of $E$ that has $S_E(\R)$ compact.  See \cite{pollackSW} for a precise definition.  To prove \eqref{eqn:thetaLift2}, it then further suffices to show that
\begin{equation}\label{eqn:thetaLift3}\int_{[G_2]\times [S_E]}{\{\varphi(g),P(D_\p^{2m}\Theta_\phi(g,h))\}\,dg} \neq 0\end{equation}
for some such $E$.

The integral in \eqref{eqn:thetaLift3} can be evaluated using the main theorem of \cite{pollackSW} and the Rankin-Selberg integral studied in \cite{gurevichSegal,Segal2017} (see also \cite{pollackG2}).  To set up the result, following \cite{pollackSW}, write $G_E$ for a certain simply connected group of absolute Dynkin type $D_4$, defined in terms of $E$ and split over $\R$.

We have
\[D_{\p}^{2m} \otimes \Theta(g) = D_{\p}^{2m}\left( \sum_{v}\Theta_{v_j}(g)\otimes [x^{4+j}][y^{4-j}]\right) = \sum_{\alpha, j} \Theta_{u_\alpha v_j}(g) \otimes v_j^\vee \otimes u_\alpha^\vee\]
for elements $u_\alpha \in U(\g(J)\otimes \C)$. Thus by Corollary 9.4.8 of \cite{pollackSW},
\begin{equation}\label{eqn:SWint1}\int_{[G_2] \times [S_E] }{\{\varphi(g),P(D_{\p}^{2m}\Theta_\phi(g,h))\}\,dg\,dh} = \int_{[G_2]}{\{\varphi(g),\sum_{\alpha,j} E_1(\phi,u_\alpha v_j)(g,s=5) \otimes P(v_j^\vee \otimes u_\alpha^\vee)\}\,dg}.\end{equation}
Here $E_1(\phi, u_\alpha v_j,s=5)$ is the Siegel-Weil Eisenstein series on the group $G_E$.

The integral of \eqref{eqn:SWint1} can now be written as a partial $L$-function times some local Zeta integrals at bad finite places (including the archimedean place).  Specifically, we have the following proposition.  Moreover, these local zeta integrals at the finite places can be trivialized with Siegel-Weil inducing data for the Eisenstein series $E_1$ on $G_E$.  Specifically, we have the following proposition.

Suppose $\chi: N(\Q)\backslash N(\A) \rightarrow \C^\times$ is a unitary character.  To setup the proposition, we define
\[I_{\infty,\chi}(s) = \int_{N_{0,\chi}\backslash G_2(\R)}\{W_\chi(g),\sum_{\alpha,j} f_{u_\alpha v_j}(\gamma_{0,\chi} g,s) \otimes P(v_j^\vee \otimes u_\alpha^\vee)\}\,dg.\]
This is the local archimedean Zeta integral that comes from the Rankin-Selberg integral \eqref{eqn:SWint1}.  The notation is from \cite[Theorem 5.2]{pollackG2}, which is a restatement of a Theorem of \cite{gurevichSegal,Segal2017}. Here also $f_{u_\alpha v_j}(g,s)$ is the Siegel-Weil inducing section from \cite{pollackSW}.  It follows from Proposition \ref{prop:ZetaArchAbs} below that $I_{\infty,\chi}(s)$ converges absolutely for $Re(s)>1$.

\begin{proposition}\label{prop:SWRS} Let $\chi$ be a unitary character of $N(\Q)\backslash N(\A)$ for which $a_\chi(\varphi)(1) \neq 0$. Recall that from the theory of binary cubic forms one associates to $\varphi$ is a rank three $\Z$-module $R_\chi$ in a cubic \'etale algebra $E$ over $\Q$. Suppose $S$ is a set of finite places of $\Q$ that satisfies the following conditions:
\begin{enumerate}
	\item If $p \notin S$, then $\pi_p$ is unramified and $\varphi$ is spherical at $p$;
	\item If $p \notin S$, then $R_\chi \otimes \Z_p$ is a ring, and in fact the maximal order of $E\otimes \Z_p$;
	\item $S \supseteq \{2,3\}$.
\end{enumerate}
Then the finite vector $\phi$ can be chosen so that $\phi$ is spherical outside $S$, and the integral \eqref{eqn:SWint1} is equal to $L^S(\pi,Std,s=3) I_{\infty,E}(s=3)$.  Moreover, if $\varphi$ is unramified at all finite primes, and $R_\chi = \Z \times \Z \times \Z$ in $\Q \times \Q \times \Q$, then $S$ may be chosen to be empty.
\end{proposition}
Observe that if $\phi$ is level one and has $\Z \times \Z \times \Z$ Fourier coefficient nonzero, then we may take $S$ to be empty.
\begin{proof} The fact that the global integral represents the partial $L$-function is from \cite{gurevichSegal,Segal2017}, for a slightly larger set $S$.  In \cite{pollackG2} the set $S$ is shrunk to that in the statement of the proposition, except that \cite{pollackG2} includes $2,3 \in S$ in all cases.  Then, in \cite{CDDHPRcompleted}, the case where $\varphi$ is level one and $R_\chi$ is $\Z \times \Z \times \Z$ is handled.
	
That the bad local integrals may be trivialized with some data is in \cite[Section 7]{Segal2017}.  What we state and use is slightly stronger.  Specifically, we must verify that the bad local integrals can be trivialized for Siegel-Weil inducing sections. This follows simply because the Siegel-Weil inducing sections make up all of the induced representation $Ind_{P_E(\Q_p)}^{G_E(\Q_p)}(\delta_{P_E})$, where $P_E$ is the Heisenberg parabolic of $G_E$.  To see this, recall that $Ind_{P_E(\Q_p)}^{G_E(\Q_p)}(\delta_{P_E})$ is generated by any vector which is not annihilated by the long intertwining operator, which turns out to be given by an absolutely convergent integral.  The restriction of the spherical vector from $G_J(\Q_p)$ is positively valued on $G_E(\Q_p)$, so it cannot be annihilated by the long intertwining operator.

Finally, what we have stated is that $\phi$ may be chosen to trivialize the integral, without changing the cusp form $\varphi$.  This is slightly stronger from what is stated in \cite{Segal2017}.  This claim follows from Lemma \ref{lem:SBfinite} below.
\end{proof}

It turns out the $L$-value $L^S(\pi,Std,s=3)$ is always nonzero.  This is a direct consequence of the main theorem of \cite{muicG2}.
\begin{theorem}\label{thm:G2partialL} Let $\pi$ be a cuspidal automorphic representation of $G_2$ over $\Q$.  Then the Euler product defining the partial standard $L$-function $L^S(\pi,Std,s)$ converges absolutely for $Re(s) > 2$.\end{theorem}
\begin{proof} The local factors $\pi_p$ are unitarizable.  In \cite{muicG2}, the unitary dual of $p$-adic $G_2$ is completely and explicitly determined.  In particular, when $\pi_p$ is spherical, one has tight bounds on the Satake parameters of $\pi_p$.  These bounds imply the absolute convergence statement of the theorem.
\end{proof}

Finally, it turns out that if $\ell \geq 6$ is even, the archimedean Zeta integral $I_{\infty,\chi}(s=3)$ is nonzero for all non-degenerate $\chi$.
\begin{proposition}  Suppose $\ell \geq 6$ is even.  Then $I_{\infty,\chi}(s=3)$ is nonzero for all non-degenerate $\chi$.
\end{proposition}
\begin{proof} A change of variables in the integral $I_{\infty,\chi}(s)$ shows that the non-vanishing of $I_{\infty,\chi}(s=3)$ is equivalent for all $\chi$.  So, we take $\chi$ with $R_\chi = \Z \times \Z \times \Z$.  We will prove that this integral is nonvanishing by a global argument, using Proposition \ref{prop:SWRS}, Corollary \ref{cor:introG2FC}, and Theorem \ref{thm:G2partialL}.  Fix now $E = \Q \times \Q \times \Q$.

Let $\beta_{K,m}$ be as in the proof of Corollary \ref{cor:introG2FC}. Set $\beta_0 = \int_{S_E(\R)}{k \cdot \beta_{K,m}\,dk}$.  Then one sees that $\Theta(k \cdot \beta_{K,m})$ has $\Z \times \Z \times \Z$ Fourier coefficient equal to $6$, so $\Theta(\beta_0)$ does as well.

We can write $\Theta(\beta_0)$ as a finite sum of level one cuspidal eigenforms forms $\varphi_j$, with $\langle \varphi_j, \Theta(\beta_0) \rangle \neq 0$.  Thus there is some such $\varphi$ with $\Z \times \Z \times \Z$ Fourier coefficient nonzero; fix this $\varphi$. 

Now 
\[\int_{[G_2] \times [S_E] }{\{\varphi(g),P(D_{\p}^{2m}\Theta(g,h))\}\,dg\,dh} = L(\pi_\varphi,Std,s=3) I_{\infty,\chi}(s=3)\] 
from Proposition \ref{prop:SWRS}.

We have $\{\varphi,\Theta(\beta_0) \} \neq 0$.  So
\[\int_{[G_2]} \{\varphi(g) \otimes \beta_0, P(D_{\p}^{2m} \Theta(g,1))\} \,dg \neq 0.\]
Let $w(h) = \int_{[G_2]  }{\{\varphi(g),P(D_{\p}^{2m}\Theta(g,h))\}\,dg}$.  Then 
\[\int_{[S_{E}]}{w(h)\,dh} = |\Gamma_{S_E}|^{-1} \int_{S_E(\R)}{w(k)\,dk}\]
where $\Gamma_{S_E}$ is some finite group.  Here we are using that $S_E(\A) = S_E(\Q) S_E(\R) S_{E}(\widehat{\Z})$.

But this latter integral is nonzero, because it is nonzero after pairing with $\beta_0$.  Consequently, we have deduced that $I_{\infty,\chi}(s=3)$ is nonzero. 
\end{proof}

\begin{proof}[Proof of Theorem \ref{thm:ArithG2}]
We have proved that the space of lifts $S_{\ell}(G_2)_{\Theta}$ has a $\Q_{cyc}$ structure, from the Fourier coefficients.  We have also prove that if $\ell \geq 6$ is even, then $S_{\ell}(G_2)_{\Theta} = S_{\ell}(G_2;\C)$.  This proves the theorem.\end{proof}

We end with some of the technical details that were used in the proofs above.

\begin{lemma}\label{lem:SBfinite} Let $V_p$ denote the space of the representation $\pi_p$, and suppose $L: V_p \rightarrow \C$ is an $(N,\chi)$ functional. Given $v \in V_p$, there is a Schwartz-Bruhat function $\Phi$ on $\g_E \otimes \Q_p$ so that 
	\[I_p(\Phi,v,s) = \int_{N_{0,E}\backslash G_2(\Q_p)}{L(g v) f(\gamma_{0,E} g,\Phi,s)\,dg}\]
	is equal to $L(v)$, independent of $s$.  
\end{lemma}
\begin{proof} Write $\gamma_{0,E}^{-1} E_{13} = e \otimes \omega$ in the notation of \cite{pollackG2}. We have
\[I_p(\Phi,v,s) = \int_{\GL_1 \times N_{0,E}\backslash G_2}{|t|^s \Phi(t g^{-1} e \otimes \omega)L(g v)\,dg\,dt}.\]
The function $\Phi$ is on $\mathfrak{g}_{E}$, and we have $\mathfrak{g}_{E} = \g_2 \oplus E^0 \otimes V_7$.  In this decomposition, we can write $e \otimes \omega = (e \otimes \omega') + (\alpha_0 e_1 + \alpha_1 e_{3}^*)$, where $\alpha_0, \alpha_1$ are a basis of $E^0$, the trace $0$ elements of $E$.

We take $\Phi$ a pure tensor, $\Phi = \Phi_{\g_2} \otimes \Phi_{E^ \otimes V_7}$.  We make $\Phi_{E^0 \otimes V_7}$ be the characteristic function of a set very close to $\alpha_0 e_1 + \alpha_1 e_3^*$. Let $Z_{\GL_2}$ be the center of the $\GL_2$ Levi of the Heisenberg parabolic on $G_2$.  Then $\Phi_{E^0 \otimes V_7}(t g^{-1}) \neq 0$ implies $g \in Z_{\GL_2}(t) N_{Heis} K_{G_2}(p^M)$ for some $M >>0$. Here $K_{G_2}(p^M)$ is the elements of $G_2(\Q_p)$ that are $1$ modulo $p^M$ in the $7 \times 7$ matrix representation of $G_2$. Here we are using that if $h \in G_2$, with $h^{-1} e_1 = e_1 + \delta_1$ and $h^{-1} e_3^* = e_3^* + \delta_2$, with $\delta_j \in p^M V_7(\Z_p)$, then there is $k \in K_{G_2}(p^M)$ so that $(hk)^{-1} e_1 = e_1$ and $(hk)^{-1} e_3^* = e_3^*$.  (Indeed, this latter fact can be proved by using $K_{\GL_2}(p^M)$ and also unipotent elements in $N(p^M \Z_p)$.)
	
Thus we must evaluate 
	\[\int_{\GL_1 \times G_a}{L(z(t) v)\Phi_{\g_2}(t (e \otimes \omega' + z E_{13}))|t|^s\psi(z)\,dt\,dz}.\]
We choose $\Phi_{\g_2}$ to be a pure tensor in our root basis of $\g_2$, so that $\Phi_{\g_2}(t (e \otimes \omega' + z E_{13})) = \Phi_{\g_2}'( t(e \otimes \omega')) \Phi_{E_{13}}(t z)$.  But $\int_{G_a}{\Phi_{E_{13}}(t z)\psi(z)\,dz} = \widehat{\Phi_{E_{13}}}(t^{-1}) |t|^{-1}$. By choosing $\Phi_{E_{13}}$ so that its Fourier transform is supported near $t=1$, we see that we can trivialize the integral to a constant multiple of $L(v)$.  This proves the lemma.
\end{proof}

We now prove the absolute convergence of the archimedean Zeta integral. The integral in question is
\begin{equation}\label{eqn:archZeta1}\int_{N_{0,E}\backslash G_2(\R)}{ \{W_\chi(g), f(\gamma_0 g,s)\}\,dg}.\end{equation}
\begin{proposition}\label{prop:ZetaArchAbs} The integral \eqref{eqn:archZeta1} converges absolutely for $Re(s) > 1$.
\end{proposition}
\begin{proof} Let $\Phi$ be a Sschwartz function on $\g_\R$.  We obtain an inducing section from $\Phi$ as
	\[f(g,\Phi,s) = \int_{\GL_1(\R)}{|t|^s \Phi(t g^{-1}E_{13})\,dt}.\]
We will check that every inducing section in $I(s)$ is of this form, and we will prove the proposition for these inducing sections.
	
For the first part, observe that if $f \in I(s)$, then restricting to $K_\infty$ we obtain $f(k_1 k,s) = f(k,s)$ for all $k_1 \in K_\infty \cap M(\R)$. These are the $k_1 \in K_\infty$ for which $k_1 E_{13} = \pm E_{13}$, and note that the negative sign does indeed occur. Now let $\beta$ be an arbitrary even smooth function on $\g_E\otimes \R$ and $\alpha$ a smooth compactly supported function on $\R_{>0}$.  We set $\Phi(v) = \alpha(||v||) \beta\left(\frac{v}{||v||}\right)$; this is a Scwhartz function.
	
We have now
	\[f(k,\Phi,s) = \int_{\R^\times}{|t|^s \alpha(|t|) \beta(k^{-1} E_{13})\,dt} = 2 \beta(k^{-1} E_{13})\int_{\R_{>0}}{t^s \alpha(t)\,dt}.\]
We have used the evenness of $\beta$.  But because $\beta(k^{-1}E_{13})$ gives an arbitrary function on $(K_\infty \cap M(\R))\backslash K_\infty$, we see that every inducing section in $I(s)$ is an $f(g,\Phi,s)$.
	
We thus now proceed to prove the absolute convergence of the double integral
	\[\int_{\GL_1(\R)}\int_{N_{0,E}\backslash G_2(\R)} W_\chi(g) |t|^s \Phi(t g^{-1} e \otimes \omega)\,dg\,dt\]
	when $Re(s) > 1$.
	
	The bound we use below on $W_\chi(mk)$ is independent of $k$, so it suffices to integrate over $\GL_1 \times (\mathrm{G}_a \times \GL_2)$.  Without loss of generality, we can assume that $\Phi$ is a pure tensor of the appropriate sort.  We thus must bound the integral
	\[\int_{\GL_1 \times \mathrm{G}_a \times \GL_2} |t|^s \Phi_{13}(t \det(m)^{-1} z)\Phi_{\g_2}(t m^{-1} \omega') |\det(m)|^{-3}\Phi_{M_2}(t m^{-1}) W_\chi(m)\,dt\,dz\,dm.\]
Here $\Phi_{M_2}$ is a Schwartz function on the $2\times 2$ matrices $M_2(\R)$, and the rest of the notation is as in Lemma \ref{lem:SBfinite}.
	
One has $\int_{\R}{|\Phi_{13}(t \det(m)^{-1}z)|\,dz} < C |det(m)| |t|^{-1}$ and $\Phi_{\g_2}$ is bounded above, so we must bound
	\[\int_{\GL_1 \times \GL_2} |t|^{s-1}  |\det(m)|^{-2}\Phi_{M_2}(t m^{-1}) W_\chi(m)\,dt\,dm.\]
	
	Make the variable change $t \mapsto \det(m) t$, and set $m' = \det(m) m^{-1}$.  Then we must bound
	\[\int_{\GL_1 \times \GL_2} |t|^{s-1}  |\det(m)|^{s-3}\Phi_{M_2}(t m') W_\chi(m)\,dt\,dm.\]
	
	Now we claim $\int_{\GL_1}{|t|^{s-1} \Phi(t m')\,dt}$ is, for $s > 1$ (assumed real now) bounded by $C \frac{||m||^{s-1}}{s-1}$.  Indeed, $\Phi(t m')$ is rapidly decreasing, so $|\Phi(tm')| < C \max\{1,|t|^{-N}||m||^{-N}\}$ for an $N$ sufficiently large of our choosing.  Thus
	\[\int_{\GL_1}{|t|^{s-1} |\Phi(tm')|\,dt} < C\left( \int_{0}^{||m||}{|t|^{s-1}\,\frac{dt}{t}} + \int_{||m||}^{\infty}{|t|^{s-1-N} ||m||^{-N}\,\frac{dt}{t}}\right).\]
	Dropping terms of the form $\frac{1}{s-1}$ (since we are fixing $s$), both integrals above are bounded by $C ||m||^{s-1}$.  Thus we must bound
	\[\int_{\GL_2}{||m||^{s-1} |\det(m)|^{s-3} W_\chi(m)\,dm}.\]
	We break this integral into two pieces, one where $||m||\leq 1$ and the other with $||m|| > 1$.  The first integral has a compact domain, so can be ignored.  To show the convergence of the second integral, it suffices to show that $|W_\chi(m)| < \phi(||m||)$, where $\phi$ is a rapidly decreasing function.  And since the $K$-Bessel function is rapidly decreasing, it suffices to show that $|\langle \omega', m r_0(i)\rangle| \geq C ||m||$ for all $m \in \GL_2(\R)$.
	
	Both sides of the desired inequality scale linearly with the center of $\GL_2(\R)$, so it suffices to check that 
	\[\frac{|\langle \omega', m r_0(i)\rangle|^2}{||m||^2}\]
	is bounded away from $0$ for $m \in \SL_2(\R)$.  Furthermore, by a change of variables again (now or in the initial integral), we can assume $\omega' = (0,1,-1,0)$.  Then if 
	\[m = \mb{1}{x}{0}{1} \mb{y^{1/2}}{}{}{y^{-1/2}} k,\]
we wish to bound below the quantity
	\[y^{-3} |h(z)| ((y^2+x^2+1)/y)^{-1} = \frac{((x-1)^2+y^2)(x^2+y^2)}{(x^2+y^2+1)y^2}\]
	where $h(z) = z^2-z = z(z-1)$.
	
	Finally, to see that this rational function in $x,y$ is bounded below for $y>0$, we work in polar coordinates.  We have
	\[\frac{((x-1)^2+y^2)(x^2+y^2)}{(x^2+y^2+1)y^2} = \left(1 - \frac{2x}{r^2+1}\right)\left(1+ \frac{x^2}{y^2}\right) = \left(1+\frac{(x-1)^2}{y^2}\right)\left(1 - \frac{1}{r^2+1}\right).\]
	
	If $|r| \geq 1/2$ then $r^2 \geq 1/4$ so $1+r^2 \geq 5/4$ so $1/(1+r^2) \leq 4/5$ and then $1-(1/(r^2+1)) \geq 1/5$, so the quantity is at least $1/5$.  If $|r| \leq 1/2$, then
	\[ 2x/(r^2+1) = 2r \cos(\theta)/(r^2+1) \leq 2r/(r^2+1) \leq 2(1/2)/((1/2)^2+1) = 4/5\]
	because $f(r) = 2r/(r^2+1)$ is increasing on $(0,1)$.
	This proves the claim, and thus the proposition.
\end{proof}

\bibliography{nsfANT2020new}
\bibliographystyle{amsalpha}
\end{document}